\documentclass[12pt]{amsart}  	

\usepackage[all]{xy}						
\usepackage{amssymb}
\usepackage{verbatim}
\usepackage{amscd,latexsym,amsthm,amsfonts,amssymb,amsmath,amsxtra,dsfont}
\usepackage[mathscr]{eucal}
\usepackage[colorlinks]{hyperref}
\usepackage{mathtools}

\hypersetup{linkcolor=black, citecolor=black}

\newcommand{\BA}{{\mathbb {A}}}

\newcommand{\BC}{{\mathbb {C}}}

\newcommand{\BF}{{\mathbb {F}}}
\newcommand{\BG}{{\mathbb {G}}}

\newcommand{\BL}{{\mathbb {L}}}

\newcommand{\BN}{{\mathbb {N}}}

\newcommand{\BQ}{{\mathbb {Q}}}

\newcommand{\BZ}{{\mathbb {Z}}}

\newcommand{\CC}{{\mathcal {C}}}

\newcommand{\CF}{{\mathcal {F}}}

\newcommand{\CI}{{\mathcal {I}}}
\newcommand{\CJ}{{\mathcal {J}}}

\newcommand{\CM}{{\mathcal {M}}}

\newcommand{\CO}{{\mathcal {O}}}
\newcommand{\CP}{{\mathcal {P}}}

\newcommand{\CS}{{\mathcal {S}}}
\newcommand{\CT}{{\mathcal {T}}}

\newcommand{\CV}{{\mathcal {V}}}

\newcommand{\CZ}{{\mathcal {Z}}}

\newcommand{\FA}{{\mathfrak {A}}}

\newcommand{\FF}{{\mathfrak {F}}}

\newcommand{\FL}{{\mathfrak {L}}}

\newcommand{\Fa}{{\mathfrak {a}}}

\newcommand{\Fj}{{\mathfrak {j}}}

\newcommand{\Fl}{{\mathfrak {l}}}

\newcommand{\Fp}{{\mathfrak {p}}}

\newcommand{\Fr}{{\mathfrak {r}}}
\newcommand{\Fs}{{\mathfrak {s}}}

\newcommand{\Fu}{{\mathfrak {u}}}

\newcommand{\RD}{{\mathrm {D}}}

\newcommand{\RG}{{\mathrm {G}}}

\newcommand{\RI}{{\mathrm {I}}}
\newcommand{\RJ}{{\mathrm {J}}}

\newcommand{\RM}{{\mathrm {M}}}

\newcommand{\RR}{{\mathrm {R}}}

\newcommand{\RZ}{{\mathrm {Z}}}

\newcommand{\Ad}{{\mathrm{Ad}}}

\newcommand{\ab}{{\mathrm{ab}}}
\newcommand{\ac}{{\mathrm{ac}}}

\newcommand{\bk}{{\mathrm{BK}}}

\newcommand{\ch}{{\mathrm{ch}}}

\newcommand{\GL}{{\mathrm{GL}}}

\newcommand{\Hom}{{\mathrm{Hom}}}

\newcommand{\Ind}{{\mathrm{Ind}}}
\newcommand{\ind}{{\mathrm{ind}}}
\newcommand{\inv}{{\mathrm{inv}}}

\newcommand{\Id}{{\mathrm{Id}}}

\newcommand{\Mat}{{\mathrm{Mat}}}

\newcommand{\ord}{{\mathrm{ord}}}

\newcommand{\PSR}{\mathrm{PSR}}
\newcommand{\pvs}{\mathrm{pvs}}
\newcommand{\pv}{\mathrm{pv}}

\renewcommand{\Re}{{\mathrm{Re}}}

\newcommand{\Res}{{\mathrm{Res}}}

\newcommand{\Spec}{{\mathrm{Spec}}}
\newcommand{\SO}{{\mathrm{SO}}}

\newcommand{\Sym}{{\mathrm{Sym}}}

\newcommand{\Sp}{{\mathrm{Sp}}}

\newcommand{\std}{{\mathrm{std}}}

\newcommand{\supp}{{\mathrm{supp}}}

\newcommand{\tr}{{\mathrm{tr}}}

\newcommand{\ud}{\,\mathrm{d}}
\newcommand{\vol}{{\mathrm{vol}}}

\newcommand{\ovl}{\overline}

\newcommand{\wt}{\widetilde}
\newcommand{\wh}{\widehat}

\newcommand{\ol}{\overline}

\newcommand{\bs}{\backslash}

\newcommand{\one}{\mathds{1}}

\def\alp{{\alpha}}

\def\bet{{\beta}}
\def\bks{{\backslash}}

\def\del{{\delta}}
\def\Del{{\Delta}}
\def\diag{{\rm diag}}

\def\eps{{\epsilon}}

\def\veps{{\varepsilon}}

\def\sig{{\sigma}}

\def\zet{{\zeta}}
\def\std{\rm std}

\def\Ome{{\Omega}}

\def\gam{{\gamma}}
\def\Gam{{\Gamma}}

\def\wb{\overline} 

\def\vpi{\varpi}

\def\vphi{\varphi}

\def\p{\prime}

\newtheorem{thm}{Theorem}[section]
\newtheorem{dfn}[thm]{Definition}
\newtheorem{rmk}[thm]{Remark}

\newtheorem{pro}[thm]{Proposition}
\newtheorem{lem}[thm]{Lemma}
\newtheorem{cor}[thm]{Corollary}

\newtheorem{cnj}[thm]{Conjecture}

\makeatletter

\newcommand{\Rmnum}[1]{\expandafter\@slowromancap\romannumeral #1@}
\makeatother

\newcommand{\apair}[1]{\left\langle {#1} \right\rangle}

\newcommand{\ppair}[1]{\left({#1}\right)}

\begin{document}
\renewcommand{\theequation}{\arabic{equation}}
\numberwithin{equation}{section}

\title[Harmonic Analysis and $\gamma$-Functions]
{Harmonic Analysis and Gamma Functions on Symplectic Groups}

\author{Dihua Jiang}
\address{School of Mathematics,
University of Minnesota,
Minneapolis, MN 55455, USA}
\email{dhjiang@math.umn.edu}

\author{Zhilin Luo}
\address{School of Mathematics,
University of Minnesota,
Minneapolis, MN 55455, USA}
\email{luoxx537@umn.edu}

\author{Lei Zhang}
\address{Department of Mathematics,
National University of Singapore,
Singapore 119076}
\email{matzhlei@nus.edu.sg}


\date{\today}

\dedicatory{In memory of Ilya Piatetski-Shapiro and Stephen Rallis}

\thanks{The research of the first named author is supported in part by the NSF Grants DMS--1600685 and DMS--1901802,
and that of the third named author is supported in part by AcRF Tier 1 grants R-146-000-237-114,  R-146-000-277-114, and R-146-000-308-114 of National University of Singapore.}


\begin{abstract}
Over a $p$-adic local field $F$ of characteristic zero, we develop a new type of harmonic analysis on an extended symplectic group $G=\BG_m\times\Sp_{2n}$. It is associated to the Langlands
$\gamma$-functions attached to any irreducible admissible representations $\chi\otimes\pi$ of $G(F)$ and the standard representation $\rho$ of the dual group $G^\vee(\BC)$, and confirms a series of the conjectures
in the local theory of the Braverman-Kazhdan proposal (\cite{BK00}) for the case under consideration.
Meanwhile, we develop a new type of harmonic analysis on $\GL_1(F)$, which is associated to a $\gamma$-function $\beta_\psi(\chi_s)$ (a product of $n+1$ certain abelian $\gamma$-functions).
Our work on $\GL_1(F)$ plays an indispensable role in the development of our work on $G(F)$. These two types of harmonic analyses both specialize to the well-known local theory developed in Tate's thesis (\cite{Tt50})
when $n=0$. The approach is to use the compactification of $\Sp_{2n}$ in the Grassmannian variety of $\Sp_{4n}$,
with which we are able to utilize the well developed local theory of Piatetski-Shapiro and Rallis (\cite{PSR86} and many other works) on the doubling local zeta integrals for the standard $L$-functions of $\Sp_{2n}$.

The method can be viewed as an extension of the work of Godement-Jacquet (\cite{GJ72}) for the standard $L$-function of $\GL_n$ and is expected to work for all classical groups.
We will consider the Archimedean local theory and the global theory in our future work.
\end{abstract}

\maketitle

\tableofcontents


\section{Introduction}\label{sec-I}


In his celebrated 1950 Princeton Thesis (\cite{Tt50} and also \cite{CF67}), J. Tate establishes the basic analytic properties (the Hecke theory) of the Hecke $L$-functions $L(s,\chi)$ associated to grossencharacters $\chi$ of a global field $k$
by Fourier analysis over global fields. In this approach, the Hecke $L$-functions $L(s,\chi)$ can be expressed by means of global zeta integrals
that are the Mellin transform on $\BA^\times$, the idele group of the global field $k$, of Schwartz functions on $\BA$, the adele ring of $k$. The meromorphic continuation and the functional equation
of the Hecke $L$-functions are essentially consequences of the Fourier transform and the associated Poisson summation formula defined over the global field $k$.
This adelic approach is based on the Euler product structure of the $L$-functions, which reflects the local-global principle in the theory.

This line of ideas was taken up by
R. Godement and H. Jacquet in 1972 (\cite{GJ72}) to represent the standard $L$-functions associated to irreducible cuspidal automorphic representations of $\GL_n(\BA)$ by their global zeta integrals:
\begin{equation}\label{GJ-gzi}
\CZ(s, \phi, \varphi_\pi, \chi)
:=
\int_{\GL_n(\BA)}\phi(g)\varphi_\pi(g)\chi(\det(g))|\det(g)|_\BA^{s+{\frac{n-1}{2}}}\ud g,
\end{equation}
where $\phi$ is a certain function of Schwartz type in $\CS(\GL_n(\BA))$, $\varphi_\pi$ is a matrix coefficient associated to the irreducible cuspidal
automorphic representation $\pi$ of $\GL_n(\BA)$, $\chi$ is an automorphic character of $\BA^\times$, $s\in\BC$, and ${\rm d} g$ is the Haar measure on $\GL_n(\BA)$.
In \cite[Theorem 13.8]{GJ72}, they prove that the global zeta integral in \eqref{GJ-gzi} converges absolutely for $\Re(s)>\frac{n+1}{2}$, can be
analytically continued to an entire function in $s\in\BC$, and satisfies the following functional equation
\begin{equation}\label{GJ-gfe}
\CZ(s, \phi, \varphi_\pi, \chi)
=
\CZ(1-s, \wh{\phi}, \varphi_\pi^\vee, \chi^{-1})
\end{equation}
where $\wh{\phi}$ is the restriction of the Fourier transform on the affine space $\Mat_n(\BA)$ of $n\times n$ matrices, and $\varphi_\pi^\vee(g):=\varphi_\pi(g^{-1})$. It is important to note that the global functional equation in \eqref{GJ-gfe} is the
consequence of the Poisson summation formula associated to the Fourier transform over the affine space $\Mat_n(\BA)$. In consequence, they establish
the Hecke theory for the standard (principal) $L$-functions of $\GL_n$.
This is a beautiful generalization
of the method of the Tate thesis and of the special case by T. Tamagawa (\cite{Tm63}).

For simplicity, we take $\RG$ to be a reductive algebraic group that splits over a number field $k$, and $\RG^\vee(\BC)$ the complex dual group of $\RG$.
Let $\rho$ be a complex representation of $\RG^\vee(\BC)$ in a complex vector space $V_\rho$ with dimension $d_\rho$. For any irreducible cuspidal automorphic
representation $\pi$ of $\RG(\BA)$, R. Langlands introduced in \cite{L70} the notion of automorphic $L$-functions associated to the pair $(\pi,\rho)$, which are defined
by the following Euler product of local $L$-factors
\begin{equation}\label{L-lfn}
L(s,\pi,\rho):=\prod_{\nu\in|k|}L(s,\pi_\nu,\rho)
\end{equation}
where $|k|$ denotes the set of all local places of $k$. Note that at the finite local places where the representations $\pi_\nu$ are unramified, the local $L$-functions can be defined to be 
\[
L(s,\pi_\nu,\rho):=\det(\RI_{d_\rho}-\rho(c(\pi_\nu))q_\nu^{-s})^{-1}
\]
where $c(\pi_\nu)$ is the Frobenius-Hecke conjugacy class in $\RG^\vee(\BC)$ attached to $\pi_\nu$ through 
the Satake isomorphism and $q_\nu$ is the cardinality of the residue field of $k_\nu$. At the Archimedean 
local places, the local $L$-functions are defined through the Langlands classification theory (\cite{L89}). Finally 
at remaining ramified local places, the local $L$-functions may be defined through the Langlands-Shahidi method or 
the Rankin-Selberg method, and are eventually confirmed through the local Langlands 
conjecture, which is known in many cases, but still one of the major problems in the local theory of 
the Langlands program. 
Langlands proved that the Euler product in \eqref{L-lfn} converges absolutely for $\Re(s)$ sufficiently positive and
made his conjecture that the Hecke theory holds for the general automorphic $L$-functions $L(s,\pi,\rho)$. In order to understand the nature of the Langlands conjecture, it is essential to
understand the analytic nature of the global $L$-functions, and that of the local $L$-functions $L(s,\pi_\nu,\rho)$ and the associated local $\gamma$-functions $\gam(s,\pi_\nu,\rho,\psi_\nu)$.
It is a desirable question: 

{\sl How to establish the basic analytic properties (or the analogous Hecke theory) for general Langlands automorphic $L$-functions $L(s,\pi,\rho)$ via harmonic analysis
on $\RG(\BA)$ for general reductive algebraic group $\RG$ defined over $k$?}

In 2000, A. Braverman and D. Kazhdan proposed in \cite{BK00} a framework (a series of conjectures) to establish the Hecke theory or the Langlands conjecture for general
automorphic $L$-functions $L(s,\pi,\rho)$ using the {\sl harmonic  analysis} on $\RG(\BA)$,
for reductive algebraic groups $\RG$ defined over a number field $k$.
The global aspect of the Braverman-Kazhdan proposal is to understand the Langlands conjecture for the complete automorphic $L$-functions $L(s,\pi,\rho)$, while the local aspect is to understand the analytic properties of
the local $\gamma$-functions $\gam(s,\pi_\nu,\rho,\psi_\nu)$ and the local $L$-functions $L(s,\pi_\nu,\rho)$.
So far, the work of Godement-Jacquet (\cite{GJ72}) still stands as the only established case that can be understood by means of
harmonic analysis on $\RG(\BA)$ according to the the Braverman-Kazhdan proposal.

Our objective is to explore new possible examples of the local and global theory of automorphic $L$-functions by developing new types of harmonic analysis on $\RG(\BA)$ as proposed in \cite{BK00}. As we explained before, the Braverman-Kazhdan proposal intends to generalize the work of Godement-Jacquet to 
general reductive groups $\RG$ with any finite dimensional representation $\rho$ of the $L$-group $^L\RG$. 
It is well-known that the doubling integrals of I. Piatetski-Shapiro and S. Rallis in \cite{GPSR87} is an 
extension of the work of Godement-Jacquet on $\GL_n$ to classical groups. It is natural to start our 
investigation with the automorphic $L$-functions that can be represented by the doubling zeta integrals of Piatetski-Shapiro and Rallis in the framework of the Braverman-Kazhdan proposal. 

In this paper we consider a case over a non-Archimedean local field $F$ of characteristic zero from the doubling method of Piatetski-Shapiro and Rallis for symplectic groups $\Sp_{2n}$. In this case, we have to consider the extended symplectic group $G=\BG_m\times\Sp_{2n}$, where $\BG_m$ is the algebraic group over $F$ such that $\BG_m(F)=F^\times$. In the framework of the Braverman-Kazhdan proposal, the 
algebraic group considered is the unit group of a reductive monoid $\CM_\rho$ associated to a given $\rho$
via the Vinberg theory of universal monoids. For the standard representation $\rho$ of the dual group 
$\SO_{2n+1}(\BC)$ of $\Sp_{2n}$, the correct algebraic group for the Braverman-Kazhdan proposal 
is $G=\BG_m\times\Sp_{2n}$ (see \cite{Li18} and \cite{Sh18}, for instance).
We intend to establish the local theory of harmonic analysis on $G(F)$ associated 
to the standard local $\gamma$-functions and $L$-functions of $G$. We leave the Archimedean local theory and the global theory for future consideration. 

As an algebraic group defined over $F$, the extended symplectic group $G=\BG_m\times\Sp_{2n}$ has its complex dual group: $G^\vee(\BC)=\BC^\times\times\SO_{2n+1}(\BC)$.
Let $\rho$ be the standard representation of $G^\vee(\BC)$, given by
\begin{equation}\label{rho}
\rho={\std}\colon G^\vee(\BC)=\BC^\times\times\SO_{\text{$2n$}+1}(\BC)\rightarrow\GL_{\text{$2n$}+1}(\BC)
\end{equation}
where the restriction of $\rho$ to $\BC^\times$ is the multiplication and the restriction of $\rho$ to $\SO_{2n+1}(\BC)$ is the standard representation (the natural embedding into $\GL_{2n+1}(\BC)$) of $\SO_{2n+1}(\BC)$.

Let $\Pi(G)$ be the set of equivalence classes of irreducible admissible representations of $G(F)$. It clear that the members in $\Pi(G)$ are of the form $\chi\otimes \pi$ with
$\pi\in\Pi(\Sp_{2n})$ and $\chi$ a quasi-character of $\BG_m(F)$.
By the local Langlands transfer from $G(F)$ to $\GL_{2n+1}(F)$, one is able to define the local $L$-functions $L(s,\chi\otimes \pi,\rho)$ in terms of the local $L$-functions for $\GL_{2n+1}(F)$, and define the local
$\gamma$-functions $\gamma(s,\chi\otimes \pi,\rho,\psi)$ with a nontrivial additive character $\psi$ of $F$, as well.

The backbone in the local harmonic analysis on $G(F)$ associated to the $\gamma$-functions $\gamma(s,\chi\otimes \pi,\rho,\psi)$ in the Braverman-Kazhdan proposal is their conjecture that asserts
the existence of a $G(F)$-invariant distribution $\Phi_{\rho,\psi}$ on $G(F)$. This $G(F)$-invariant distribution $\Phi_{\rho,\psi}$ on $G(F)$ defines a Fourier operator
 $\CF_{\rho,\psi}$ over $G(F)$ by
\begin{equation}\label{CFrho}
\CF_{\rho,\psi}(\phi)(a,h):=|a|^{-2n-1}(\Phi_{\rho,\psi}\ast\phi^\vee)(a,h),
\end{equation}
for $(a,h)\in G(F)$, where $\phi\in\CC_c^\infty(G)$, the space of smooth, compactly supported functions on $G(F)$, and $\phi^\vee(g)=\phi(g^{-1})$. This Fourier operator $\CF_{\rho,\psi}$
was called a $\rho$-Fourier transform in \cite{BK00}.

It is Conjecture 5.4 of \cite{BK00} that the Fourier operator $\CF_{\rho,\psi}$ over $G(F)$ extends to a unitary operator on the space of square-integrable functions on $G(F)$ with respect to a certain measure, and
preserves a conjectural $\rho$-Schwartz space $\CS_\rho(G)$. In order to understand the $\rho$-Schwartz space $\CS_\rho(G)$, they consider the local zeta integral:
\begin{equation}\label{rho-LZI}
\CZ(s, \phi, \varphi_{\chi\otimes \pi})
:=
\int_{G(F)}\phi(g)\varphi_{\chi\otimes \pi}(g)|\sigma(g)|^{s+n}\ud g,
\end{equation}
and assert a series of conjectural statements in \cite[Conjecture 5.11]{BK00} that $\CZ(s, \phi, \varphi_{\chi\otimes \pi})$ converges for $\Re(s)$ sufficiently positive, admits meromorphic continuation to $s\in\BC$, and is a
holomorphic multiple of the Langlands local $L$-function $L(s,\chi\otimes \pi,\rho)$. Moreover, with the Fourier operator $\CF_{\rho,\psi}$ over $G(F)$, the local zeta integral
$\CZ(s, \phi, \varphi_{\chi\otimes \pi})$ should satisfy the following functional equation:
\begin{equation}\label{rho-FE}
\CZ(1-s, \CF_{\rho,\psi}(\phi), \varphi_{\chi\otimes \pi}^\vee)=\gamma(s,\chi\otimes \pi,\rho,\psi)\CZ(s, \phi, \varphi_{\chi\otimes \pi}),
\end{equation}
where $\gamma(s,\chi\otimes \pi,\rho,\psi)$ is the Langlands local $\gamma$-function associated to $\chi\otimes\pi$ and $\rho$. Finally, for any $\chi\otimes \pi\in \Pi(G)$, the
$\chi\otimes\pi$-Fourier coefficient of the $G(F)$-invariant distribution $\Phi_{\rho,\psi}$ on $G(F)$ should yield the Langlands local $\gamma$-function:
\begin{equation}\label{rho-gamma}
(\chi\otimes \pi)(\Phi_{\rho,\psi})=\gamma(\cdot,\chi\otimes \pi,\rho,\psi)\cdot\Id_{\chi\otimes \pi},
\end{equation}
except for those $\chi\otimes \pi\in \Pi(G)$ such that $\gamma(\cdot,\chi\otimes \pi,\rho,\psi)$ has a pole.

It is clear that the local theory of the Braverman-Kazhdan proposal depends on the existence of the $G(F)$-invariant distribution $\Phi_{\rho,\psi}$ on $G(F)$. On the one hand, it is desirable to have
a conceptual understanding of such $G(F)$-invariant distributions $\Phi_{\rho,\psi}$
for any reductive algebraic group $G$ and any irreducible finite dimensional representations $\rho$ of the $L$-group $^LG$. On the other hand, it is fundamentally important
to have an explicit construction of such $G(F)$-invariant distributions $\Phi_{\rho,\psi}$ with analytic significance so that the analytic aspects of the relevant local theory can be well established.

In \cite{BK00}, Braverman and Kazhdan provide a conjectural description of the $G(F)$-invariant distribution based on the framework of {\sl algebraic integrations} as outlined by Kazhdan in \cite{K00}. In particular, they
formulate their conjecture on the construction of the $G(F)$-invariant distributions $\Phi_{\rho,\psi}$ over finite fields (\cite[Section 9]{BK00}). The finite field case has been well studied by S. Cheng and B. C. Ng\^o in \cite{CN18},
by T.-H. Chen in \cite{C16}, and by G. Laumon and E. Letellier in \cite{LL19}. Over a local field $F$ of characteristic zero, Z. Luo in \cite{Luo19} explains the spherical part of the $G(F)$-invariant
distribution $\Phi_{\rho,\psi}^K$ for
any split connected reductive group $G$ and any $\rho$ with $K$ the standard maximal compact subgroup of $G(F)$ (see also \cite{Gz18} and \cite{Sak18}). Motivated by his consideration of the Langlands beyond endoscopy proposal via the Arthur-Selberg trace formula (\cite{FLN10}), Ng\^o explains the conjectural description of the $G(F)$-invariant distributions in the framework of Hankel transform and related harmonic analysis (\cite{N20}). 
However,
the case of Godement-Jacquet for the standard representation $\rho$ of $\GL_n(\BC)$ (\cite{GJ72}) is still the only case so far that the local theory of the Braverman-Kazhdan proposal has been fully understood.
It is worthwhile to mention that some preliminary work towards the understanding of the local theory of the Braverman-Kazhdan proposal was taken place in the work of W.-W. Li (\cite{Li17} and \cite{Li18}),
of F. Shahidi (\cite{Sh18}), and of J. Getz and B. Liu (\cite{GL20}).
In Section \ref{ssec-mrG}, we are going to explain what we have done for the case 
$G=\BG_m\times\Sp_{2n}$ over a $p$-adic local field $F$ of characteristic zero.

\subsection{Harmonic analysis on $G$}\label{ssec-mrG}

With $\rho$ as given in \eqref{rho}, we first define a {\bf $G(F)$-invariant distribution}
on $G(F)=\BG_m(F)\times\Sp_{2n}(F)$ (Definition \ref{dfn-SDFT} and Proposition \ref{pro:Phi}) to be
\begin{align}\label{SDG}
\Phi_{\rho,\psi}(a,h):=
c_0\cdot\eta_{\pvs,\psi}(a\det(h+\RI_{2n}))\cdot|\det(h+\RI_{2n})|^{-\frac{2n+1}{2}}
\end{align}
where $\eta_{\pvs,\psi}(x)$ is a locally constant function on $F^\times$ (Theorem \ref{thm:eta}). A rich theory to understand the function $\eta_{\pvs,\psi}(x)$
will be explained with more details in Section \ref{ssec-eta-lt}. Here $c_0$ is a constant depending on the choice of relevant measures.
Then we define the associated {\bf Fourier operator} $\CF_{\rho,\psi}$ (Definition \ref{dfn-SDFTG}) to be
\begin{align}\label{FOG}
\CF_{\rho,\psi}(\phi)(a,h)
&:=
(\Phi_{\rho,\psi}*\phi^\vee)(a,h)\nonumber\\
&=
\int_{F^\times}^\pv\int_{\Sp_{2n}(F)}\Phi_{\rho,\psi}(t,y)\phi^\vee(t^{-1}a,y^{-1}h)\ud y\ud^*t,
\end{align}
for all $\phi\in\CC^\infty_c(G)$, the space of smooth, compactly supported functions on $G(F)$, where $\phi^\vee(g)=\phi(g^{-1})$ and ${\rm d}^* t$ is the local Tamagawa measure on $F^\times$.
We may have to take the {\sl principal value integration} on ${\rm d}^*t$-integration because of the generalized function $\eta_{\pvs,\psi}(x)$
in the definition of the $G(F)$-invariant distribution $\Phi_{\rho,\psi}$ in \eqref{SDG}.
Comparing with \eqref{CFrho}, one notices that our definition of the Fourier operator $\CF_{\rho,\psi}$ has a different normalization.
In Definition \ref{dfn:rho-SsG}, we introduce the space $\CS_\rho(G)$ of {\bf Schwartz functions} of $G(F)$, which consists of certain smooth functions on $G(F)$ and contains $\CC^\infty_c(G)$ as a proper subspace.

The first main result of this paper is to prove Conjecture 5.4 in \cite{BK00} for the case under consideration (Theorems \ref{thm:Pl-FOG} and \ref{thm:FOG-Ss}).

\begin{thm}[Fourier Operator]\label{thm:L2}
The Fourier operator $\CF_{\rho,\psi}$ extends to a unitary operator on the space $L^2(G,\ud^*a\ud h)$ of square-integrable functions on $G(F)$  and satisfies the following identity:
\[
\CF_{\rho,\psi^{-1}}\circ\CF_{\rho,\psi}=\Id.
\]
Moreover, the Fourier operator $\CF_{\rho,\psi}$ stabilizes the Schwartz space $\CS_\rho(G)$.
\end{thm}
The advantage of our normalization of the Fourier operator $\CF_{\rho,\psi}$ in \eqref{FOG} is that it extends to a unitary operator in the space $L^2(G,\ud^*a\ud h)$ with a measure independent of $\rho$,
comparing to Conjecture 5.4 of \cite{BK00}.

The second main result is to understand the Fourier coefficients of the $G(F)$-invariant distribution $\Phi_{\rho,\psi}$ on $G(F)$. For any irreducible admissible representation
$\chi\otimes\pi\in\Pi(G)$, if a distribution $\Phi$ on $G(F)$ is {\sl essentially compact},
then one can define the $\chi\otimes\pi$-{\sl Fourier coefficient} of $\Phi$ by the convolution operator $(\chi\otimes\pi)(\Phi)$ (Section \ref{ssec-pf-Thm13}). Since the distribution $\Phi_{\rho,\psi}$ on
$G(F)$ is {\sl not} essentially compact, we define the $\chi\otimes\pi$-Fourier coefficient of $\Phi_{\rho,\psi}$ by decomposing it in terms of the distributions in the {\sl Bernstein center} $\CZ(G)$ of $G(F)$
in Theorem \ref{thm:Phi-gamma}. The following result follows from Theorem \ref{thm:Phi-gamma} and Corollary \ref{gam-gam}.

\begin{thm}[Fourier Coefficients and Local $\gamma$-functions]\label{thm:LGF}
For any $\chi\otimes\pi\in\Pi(G)$, the $\chi_s\otimes\pi$-Fourier coefficient of the $G(F)$-invariant distribution $\Phi_{\rho,\psi}$ is well-defined and is given by
\[
(\chi_s\otimes\pi)(\Phi_{\rho,\psi})=\gam(\frac{1}{2},\chi_{s}^{-1}\otimes \wt{\pi},\rho,\psi)\Id_{\chi_s\otimes \pi},
\]
except for those $\chi_s\otimes\pi$ such that $\gam(\frac{1}{2},\chi_{s}^{-1}\otimes \wt{\pi},\rho,\psi)$ has a pole, where $\chi_s(a)=\chi(a)|a|^s$,
$\gam(\frac{1}{2},\chi_{s}^{-1}\otimes \wt{\pi},\rho,\psi)$ is the Langlands local $\gamma$-function associated to the contragredient $\chi_{s}^{-1}\otimes \wt{\pi}$ of $\chi_s\otimes\pi$.
\end{thm}
This precise identity was first found in \cite[Lemma~2.4.4]{Luo19}, where a rigorous proof was given only when $\chi\otimes \pi$ is unramified. It is easy to notice
the difference between our formula in Theorem \ref{thm:LGF} and the conjectured formula of Braverman-Kazhdan in \eqref{rho-gamma}.

In order to understand the $\rho$-Schwartz space $\CS_\rho(G)$ in terms of the matrix coefficients of $\chi\otimes\pi\in\Pi(G)$, we consider the {\sl local zeta integrals}.
For $\phi\in\CS_\rho(G)$ and $\varphi\in\CC(\chi\otimes\pi)$, the space of matrix coefficients associated to the representation $\chi\otimes\pi\in\Pi(G)$,
the {\bf local zeta integral} as defined in \eqref{lzi} is
\begin{equation}\label{LZI-0}
\CZ(s,\phi,\vphi): =\int_{G(F)} \phi(a,h)\vphi(a,h) |a|^{s-\frac{1}{2}}\ud^{*} a \ud h,
\end{equation}
where $g=(a,h)\in G(F)=\BG_m(F)\times\Sp_{2n}(F)$.
Note that when $n=0$ we consider $\Sp_{2n}(F)$ as the trivial group and then $\CZ(s,\phi,\varphi)$ is the local zeta integral of Tate (\cite{Tt50}). Comparing with \eqref{rho-LZI}, the local zeta integral in
\eqref{LZI-0} has a different normalization, due to the fact that our Schwartz space $\CS_\rho(G)$ is a subspace of $L^2(G,\ud^*a\ud h)$.

The following is Theorem \ref{thm-lfe} that establishes the {\bf local functional equation}
for the local zeta integrals with the Fourier operator $\CF_{\rho,\psi}$.

\begin{thm}[Functional Equation and Local $\gamma$-functions]\label{thm:LFE-CF-BK}
The local zeta integrals $\CZ(s,\phi,\vphi)$ converge absolutely for $\Re(s)$ sufficiently positive, admit meromorphic continuation to $s\in\BC$, and satisfy the following functional equation:
\begin{align}\label{lfe-0}
\CZ(1-s, \CF_{\rho,\psi}(\phi) , \vphi^\vee)  =\gam(s,\chi\otimes \pi,\rho,\psi) \CZ(s,\phi,\vphi),
\end{align}
where $\varphi\in \CC(\chi\otimes\pi)$,
${\vphi}^\vee(g)=\varphi(g^{-1})\in\CC(\chi^{-1}\otimes\wt{\pi})$ with $\chi^{-1}\otimes\wt{\pi}$ being the contragredient of $\chi\otimes\pi$, and $\phi\in\CS_\rho(G)$.
\end{thm}

As a consequence, we prove in Corollaries \ref{cor:gfgg} and \ref{gam-gam} that the Langlands $\gamma$-function $\gam(s,\chi\otimes \pi,\rho,\psi)$ is a gamma function in the sense of Gelfand and Graev (\cite{GG63}, \cite{GGPS}, and \cite{ST66}).
In other words, as distributions on $G(F)$, we have that
\begin{align}
\CF_{\rho,\psi}^*(\varphi_{\chi^{-1}_{s}\otimes\wt{\pi}})=\gamma(\frac{1}{2},\chi_{s}\otimes\pi,\rho,\psi)\varphi_{\chi_s\otimes\pi}
\end{align}
holds, via meromorphic continuation in $s$,
for any irreducible admissible representation $\chi_s\otimes\pi$ of $G(F)$, where $\varphi_{\chi_s\otimes\pi}$ is a matrix coefficient of $\chi_s\otimes\pi$ and
$\varphi_{\chi^{-1}_{s}\otimes\wt{\pi}}=\varphi_{\chi_s\otimes\pi}^\vee$ is the matrix coefficient of the contragredient $\chi^{-1}_{s}\otimes\wt{\pi}$ of $\chi_s\otimes\pi$. As a consequence, we obtain that
\begin{align}\label{gam-gam=1}
\gamma(\frac{1}{2},\chi_{s}\otimes\pi,\rho,\psi)\cdot\gamma(\frac{1}{2},\chi_{s}^{-1}\otimes\wt{\pi},\rho,\psi^{-1})=1.
\end{align}
When $n=0$, this is the classical identity for the gamma functions of $F^\times$.

The relation of the local zeta integrals $\CZ(s,\phi,\vphi)$ with the local $L$-function $L(s,\chi\otimes\pi,\rho)$ for the given $(G,\rho)$ is explained in the following result, which is a combination of
Theorem \ref{thm:rho-Ss} and Proposition \ref{pro:FT-bf}.

\begin{thm}[Local $L$-functions]\label{thm:LLF}
For any $\phi\in\CS_\rho(G)$ and any $\varphi\in\CC(\chi\otimes\pi)$, the local zeta integral $\CZ(s,\phi,\vphi)$ is a holomorphic multiple of the Langlands
local $L$-function $L(s,\chi\otimes\pi,\rho)$. Moreover, when $F$ has the residue characteristic $p\neq 2$ and $\chi\otimes\pi$ is unramified,
if $\phi$ is taken to be the basic function $\BL_\rho\in\CS_\rho(G)$ and $\varphi$ is taken to be the normalized unramified matrix coefficient $\varphi_\circ$, then $\CZ(s,\BL_\rho,\vphi_\circ)=L(s,\chi\otimes\pi,\rho)$.
\end{thm}

We take the assumption that $F$ has the residue characteristic $p\neq 2$ when consider the basic function $\BL_\rho$ in the theorem, due to a technical reason that will be explained in Section \ref{ssec-bf}.

It is desirable to know the relation between the definition of the $G(F)$-invariant distribution $\Phi_{\rho,\psi}(g)$ on $G(F)$ and the construction of such distributions as proposed by Braverman and Kazhdan in
Conjecture 7.11 of \cite{BK00} and that proposed by Ng\^o in Section 6.2 of \cite{N20}. It is also desirable to understand the relation between our definition of the Fourier operator $\CF_{\rho,\psi}$ over $G(F)$
and that of L. Lafforgue defined by using the Plancherel decomposition and the properties of local $L$-functions in \cite[Proposition II.6]{Lf14} and \cite[Definition 6.1]{Lf16}. It is expected that our construction of
the $G(F)$-invariant distribution $\Phi_{\rho,\psi}$ on $G(F)$ should have its geometric background related to the reductive monoid $\CM_\rho$ of $(G,\rho)$, which is not discussed in this paper.
We plan to come back to all those issues in our future work.


\subsection{Main idea of the approach}\label{ssec-mip}

The main idea of our approach is to consider the compactification of the symplectic group $\Sp_{2n}$ in the Grassmannian variety associated to $\Sp_{4n}$. This geometric structure was first used by a pioneering
work of I. Piatetski-Shapiro and S. Rallis to study the analytic properties of the standard $L$-functions for symplectic groups $\Sp_{2n}$ by their doubling method (\cite[Part A]{GPSR87} and \cite{PSR86}). It can be
displayed in the following diagram:
\begin{equation}\label{2OpenX}
\begin{matrix}
&&&&\\
&&\Sp_{4n}&&\\
&&&&\\
&&\downarrow&&\\
&&&&\\
M^{\ab}_\Del w_\Del N_\Del&\rightarrow& X_{P_\Del}&\leftarrow& M^{\ab}_\Del(\Sp_{2n}\times\{\RI_{2n}\})\cong \BG_m\times\Sp_{2n}\\
&&&&
\end{matrix}
\end{equation}
where $P_\Del=M_\Del N_\Del$ is a Siegel parabolic subgroup of $\Sp_{4n}$, $w_\Del$ is the Weyl element of $\Sp_{4n}$ that takes $P_\Del$ to its opposite $P_\Del^-$, $M_\Del^\ab=[M_\Del,M_\Del]\bks M_\Del$
is the maximal abelian quotient of $M_\Del$, and
$X_{P_\Del}:=[P_\Del,P_\Del]\bks\Sp_{4n}$. On the left-hand side, $M^{\ab}_\Del w_\Del N_\Del$ is Zariski open in $X_{P_\Del}$. On the right-hand side, $G=\BG_m\times\Sp_{2n}$ is also embedded into $X_{P_\Del}$ as a Zariski
open subset via the doubling embedding as explained in Section \ref{ssec-lzidm}.

It is important to note that the analytic properties of the local zeta integrals of Piatetski-Shapiro and Rallis is essentially determined by the analytic properties of the intertwining operator
$\RM_{w_\Del}(s,\chi)$ from the normalized induced representation $\RI(s,\chi)$ of $\Sp_{4n}(F)$ to $\RI(-s,\chi^{-1})$, as defined in \eqref{lio-0}.
Geometrically, it focuses on the left-hand side of Diagram \eqref{2OpenX}.
As explained in Section \ref{sec-LT-DZI}, from a series of work (\cite{PSR86}, \cite{Ik92}, \cite{LR05}, \cite{Y14}, \cite{Ik17}, and \cite{Kk18}), the analytic properties of the local zeta integrals
of Piatetski-Shapiro and Rallis can be well established if one can properly normalize the local intertwining operator $\RM_{w_\Del}(s,\chi)$ (Section \ref{ssec-lgf}) by the
following formula:
\begin{equation}\label{nlio-beta}
\RM^\dag_{w_{\Del}}(s,\chi,\psi)=\chi_{s}(2)^{2n}|2|^{n(2n+1)}\beta_{\psi}(\chi_s)\cdot	
\RM_{w_{\Del}}(s,\chi),
\end{equation}
with the factor $\beta_{\psi}(\chi_s)$ given by a finite product of local $\gamma$-functions of abelian type:
\begin{equation}\label{beta}
\beta_{\psi}(\chi_s)=\gam(s-\frac{2n-1}{2},\chi,\psi)\prod_{r=1}^{n}\gam(2s-2n+2r,\chi^{2},\psi).
\end{equation}
Note that the factor $\chi_{s}(2)^{2n}|2|^{n(2n+1)}$ occurs due to the geometric structure of $P_\Del$ relative to the standard Siegel parabolic subgroup of $\Sp_{4n}$.

In \cite{BK02}, in order to understand the local theory of the zeta integrals of Piatetski-Shapiro and Rallis, Braverman and Kazhdan developed
Fourier analysis on the algebraic variety $X_{P_\Del}$, which is related to a reductive algebraic monoid associated to $(G,\rho)$ (refer to \cite[Chapter 7]{Li18} for a detailed discussion). Geometrically, this sits in the middle of Diagram \eqref{2OpenX}.
To this end, Braverman and Kazhdan propose in \cite{BK02} a framework
(for all parabolic subgroups of a general $G$) to define a {\bf Fourier transform} $\CF_{X,\psi}$ on $X_{P_\Del}(F)$ by the following integral:
\begin{equation}\label{FTX}
\CF_{X,\psi}(f)(g):=\int_{F^\times}^\pv\eta^{\bk}_{\psi}(x)|x|^{-\frac{2n+1}{2}}\int_{N_\Del(F)}f(w_\Del n \Fs_x g)\ud n\ud^* x,
\end{equation}
for $f\in\CC^\infty_c(X_{P_\Del})$, the space of all compactly supported, smooth functions on $X_{P_\Del}(F)$, where $\eta^{\bk}_{\psi}(x)$ is the Braverman-Kazhdan generalized function on $F^\times$ that normalizes the
relevant local intertwining operator and recovers the local normalizing factor $\beta(\chi_s)$ via convolution.
We refer to Definition \ref{dfn:FTX} for other unexplained notations here. It is expected from \cite{BK02} that the Fourier transform $\CF_{X,\psi}$ over $X_{P_\Del}(F)$ should be compatible with
the normalized local intertwining operator
$\RM^\dag_{w_{\Del}}(s,\chi,\psi)$, in the sense of Theorem \ref{thm:CFP-M} in Section \ref{sec-etapvs-FT}. This expectation has been discussed in detail by Getz, Hsu and Leslie in their recent preprint (\cite{GHL}). 

In this paper, in order to obtain more explicit analytic properties of the Schwartz spaces that are defined in 
Sections \ref{sec-etapvs-FT} and \ref{sec-FOG}, 
we find an analytic construction (formula) of the generalized function $\eta_{\pvs,\psi}(x)$ on $F^\times$, based on the theory of local zeta functions attached to certain prehomogeneous vector spaces,
where $\pvs$ is for {\sl prehomogeneous vector space}, in Sections \ref{sec-FEbeta} and \ref{sec-FTbeta}. With replacement of the Braverman-Kazhdan generalized function $\eta^{\bk}_{\psi}(x)$ by our generalized function
$\eta_{\pvs,\psi}(x)$ in the definition of the Fourier transform $\CF_{X,\psi}$ in \eqref{FTX}, we are able to prove Theorem \ref{thm:CFP-M} on the compatibility of $\CF_{X,\psi}$ with $\RM^\dag_{w_{\Del}}(s,\chi,\psi)$.
Theorem \ref{thm:CFP-M} is one of the technical key results in the paper, which makes possible for us to establish Theorem \ref{thm:L2} on the Fourier operator $\CF_{\rho,\psi}$ on $G(F)=\BG_m(F)\times\Sp_{2n}(F)$
via the compactification of $\Sp_{2n}$ in the Grassmannian variety of $\Sp_{4n}$ as displayed in Diagram \eqref{2OpenX}, and prove Theorem \ref{thm:asymSf} that characterizes the functions in the Schwartz space $\CS_\pvs(X_{P_\Del})$ by means of their asymptotic behavior on the $M_\Del^\ab$-part.

It is expected that the proof of Theorems \ref{thm:LGF}, \ref{thm:LFE-CF-BK}, and \ref{thm:LLF} needs the following conjecture.

\begin{cnj}[Multiplicity One]\label{cnj:mo}
For any irreducible admissible representation $\chi\otimes\pi$ of $G(F)$, the Schwartz space $\CS_\rho(G)$ enjoys the following multiplicity-one property:
\begin{align}\label{m-one}
\dim \Hom_{G(F)\times G(F)}(\CS_\rho(G),(\chi\otimes\pi)\otimes(\chi^{-1}\otimes\wt{\pi}))\leq 1,
\end{align}
where $\chi^{-1}\otimes\wt{\pi}$ is the contragredient of $\chi\otimes\pi$.
\end{cnj}
In Section \ref{sec-MOGF}, we prove a weaker version of Conjecture \ref{cnj:mo} with the space $\CS_\rho(G)$ replaced by the space $\CC_c^\infty(G)$ (Lemma \ref{lem:wmo}). This is enough to prove Proposition
\ref{pro:lfe} and define a gamma function $\Gamma_{\rho,\psi}(s,\chi\otimes\pi)$, based on Theorem \ref{thm:L2} as proved in Section \ref{sec-FOG}.
Then we show this gamma function enjoys the desired properties (Corollary \ref{cor:gfgg} and Theorem \ref{thm:Phi-gamma}).
In order to prove that the gamma function $\Gamma_{\rho,\psi}(s,\chi\otimes\pi)$ as defined in Proposition \ref{pro:lfe} is equal to the Langlands $\gamma$-function $\gam(s,\chi\otimes\pi,\rho,\psi)$
(Corollary \ref{gam-gam}),
we prove Theorem \ref{thm-lfe} by using the functional equation for the Piatetski-Shapiro and Rallis zeta integrals, which completes the proof of
Theorems \ref{thm:LGF} and \ref{thm:LFE-CF-BK}. Theorem \ref{thm:LLF} is proved by using the work \cite[Part A]{GPSR87}, \cite{LR05} and \cite{Y14}.

\subsection{Harmonic analysis on $\GL_1(F)$ associated to $\beta_\psi(\chi_s)$}\label{ssec-eta-lt}

In this paper, we also develop {\sl harmonic analysis} on $\BG_m(F)=\GL_1(F)$ associated to the $\gamma$-function $\beta_\psi(\chi_s)$ as in \eqref{beta} (Sections \ref{sec-FEbeta} and \ref{sec-FTbeta}).
This theory recovers the local theory in Tate's thesis (\cite{Tt50}) when $n=0$. The key ingredient in the theory is the analytic construction of the generalized function $\eta_{\pvs,\psi}(x)$ on $F^\times$, which
is parallel to the explicit construction of the distribution $\Phi_{\rho,\psi}$ in the harmonic analysis on $G(F)$ associated to the Langlands local $\gamma$-function $\gamma(s,\chi\otimes\pi,\rho,\psi)$ as
described in Section \ref{ssec-mrG}.

Our analytic construction of the generalized function $\eta_{\pvs,\psi}(x)$ on $F^\times$ uses the theory of local zeta integrals associated to prehomogeneous vector space
$(\GL_{2n+1}(F),\Sym^2(F^{2n+1}))$. This theory has been developed through the work of
A. Weil (\cite{W65}), J. Igusa (\cite{Ig78} and \cite{Ig00}), Piatetski-Shapiro and Rallis (\cite{PSR87}), G. Shimura (\cite{S97}), and T. Ikeda (\cite{Ik17}), which extends the thesis of Tate (\cite{Tt50}), the work of Tamagawa (\cite{Tm63}), and the cases considered by Weil in \cite{W82}.

 A starting point of our approach is the following functional equation for the local zeta integrals $\CZ_{\Phi}(s,\chi)$ associated to prehomogeneous vector space $(\GL_{2n+1},\Sym^2(F^{2n+1}))$:
\begin{equation}\label{fe-pvs}
 |2|^{-2ns+2n^2}\chi^{-2n}(2)\CZ_{\varrho\cdot\wh{\Phi}}(-s-\frac{1}{2},\chi^{-1})
= \beta_\psi(\chi_{s})\CZ_{\Phi}(s+\frac{1}{2}-(n+1), \chi),
\end{equation}
with $\beta_\psi(\chi_s)$ given in \eqref{beta} occurring as the $\gamma$-factor. We refer to Proposition \ref{lm:Ikeda} (or {\cite[Theorem 2.1]{Ik17}}) for unexplained notations.

Our goal is to find a generalized function $\eta(x)$ on $F^\times$ so that a theory of harmonic analysis on $\GL_1(F)$ can be established for the given $\gamma$-function $\beta_\psi(\chi_s)$.
To this end, we consider the $F$-morphism
\begin{align}\label{det}
\det\ :\ \Sym^2(F^{2n+1})\rightarrow F,
\end{align}
and develop the harmonic analysis on $F^\times$ for the $\gamma$-function $\beta_\psi(\chi_s)$ by means of the relevant fiber integrations.
More precisely, we study the behavior of the functions on $F^\times$ obtained through the fiber integrations, as defined in \eqref{radon-1} and \eqref{radon-2}, from
$\CC_c^\infty(\Sym^2(F^{2n+1}))$, and the behavior of a generalized Fourier transform on $F^\times$ obtained from the usual Fourier transform on the affine space $\Sym^2(F^{2n+1})$ through such fiber integrations. As a consequence, we obtain the following functional equation on
$F^\times$ in Theorem \ref{thm:LT}:
\begin{equation}\label{FE-eta}
\int_{F^\times}\FL(f)(t)\chi_{s+\frac{2n+1}{2}}(t)^{-1}\ud^* t
=\beta_{\psi}(\chi_s)\int_{F^\times}f(t)\chi_{s+\frac{2n+1}{2}}(t)\ud^* t
\end{equation}
with the same $\beta_\psi(\chi_s)$ occurring as the $\gamma$-factor. Here $f\in \CS^+_{\pvs}(F^\times)$ are functions obtained through the fiber integration from $\CC_c^\infty(\Sym^2(F^{2n+1}))$
(Definition \ref{schwartzF}), and $\FL(f)$ (as defined in \eqref{LT})
is the generalized Fourier transform on $F^\times$ from the Schwartz space $\CS^+_{\pvs}(F^\times)$ to the Schwartz space $\CS^-_{\pvs}(F^\times)$ (Definition \ref{schwartzF}), which
is deduced from the usual Fourier transform on the affine space $\Sym^2(F^{2n+1})$. We refer to Sections \ref{ssec-bbf} and \ref{ssec-ssgft}, and
in particular Theorem \ref{thm:LT} for unexplained notations here. It is important to point out that the functional equation in \eqref{FE-eta} holds as meromorphic functions in $s\in\BC$. However, the domain of
convergence of the one side of the functional equation has {\sl no} overlap with that of the other side when $n>0$. We only find a way to establish such a functional equation in a $\GL_1$-theory by using the theory on the
prehomogeneous vector space $(\GL_{2n+1}(F),\Sym^2(F^{2n+1}))$.

The next step in our approach is to construct our generalized function $\eta_{\pvs,\psi}$ on $F^\times$, such that the the generalized Fourier transform $\FL(f)$ turns out to be a
Fourier (convolution) operator
with kernel function $\eta_{\pvs,\psi}$ as in Theorem \ref{thm:eta}:
\begin{equation}\label{LT-eta}
\FL(f)(t)=\FL_{\eta_{\pvs,\psi}}(f)(t)=((\eta_{\pvs,\psi}|\cdot|^{\frac{2n+1}{2}})*f^\vee)(t).
\end{equation}
The generalized function $\eta_{\pvs,\psi}$ on $F^\times$ will be explicitly defined in \eqref{def:eta-rho-psi} with the property that $\eta_{\pvs,\psi}(\chi_s)=\beta_\psi(\chi_s)$ via convolution, which is the
$\chi_s$-Fourier coefficient of $\eta_{\pvs,\psi}(x)$ for any quasi-character $\chi_s$ of $F^\times$. To further develop such a theory, we
characterize the spaces of functions from relevant constructions by their asymptotic behaviours and establish a version of the classical Paley-Wiener Theorem for the spaces $\CS^\pm_{\pvs}(F^\times)$ under the
Mellin transform (Theorem \ref{thm:PWM}). It is very important to point out that the functional equation (in \eqref{FE-eta}) as established in Theorem \ref{thm:LT} with the Fourier operator $\FL_{\eta_{\pvs,\psi}}$
(in \eqref{LT-eta} and  as given in Theorem \ref{thm:eta}) plays an indispensable role in our proof of Theorem \ref{thm:CFP-M}, which is the key technical result towards the harmonic analysis on $G(F)$ developed in
Sections \ref{sec-FOG} and \ref{sec-Thm-1-2-3}. The work in Sections \ref{sec-FEbeta} and \ref{sec-FTbeta} can be summarized in the following diagram:
\[
\tiny{
\begin{xy}
 (-36,15)*+{\CC_c^\infty(S_{2n+1})}="S1";
 (36,15)*+{\CC_c^\infty(S_{2n+1})}="S2";
 (-53,0)*+{\CS_n^+(F^\times)}="A+";
 (-36,0)*+{\CS_{n,\beta}^+(F^\times)}="B+";
 (-12,0)*+{\CS_{\pvs}^+(F^\times)}="C+";
 (12,0)*+{\CS_{\pvs}^-(F^\times)}="C-";
 (36,0)*+{\CS_{n,\beta}^-(F^\times)}="B-";
 (53,0)*+{\CS_n^-(F^\times)}="A-";
 (-53,-15)*+{\CZ_{n}^+(\widehat{F^\times})}="Z1+";
 (-36,-15)*+{\CZ_{n,\beta}^+(\widehat{F^\times})}="Z2+";
 (36,-15)*+{\CZ_{n,\beta}^-(\widehat{F^\times})}="Z2-";
 (53,-15)*+{\CZ_{n}^-(\widehat{F^\times})}="Z1-";
{\ar@{->}^{\text{Fourier Transform}} "S1";"S2"};
{\ar@{->>}_{F.I.} "S1";"B+"};
{\ar@{->>}^{F.I.} "S2";"B-"};
 {\ar@{}|(0.47){\displaystyle{\supset}} "A+";"B+"};
{\ar@{->}^{|\cdot|^{-2n}}_{\simeq} "B+";"C+"};
{\ar@{->}^{\FL_{\eta_{\pvs,\psi}}} "C+";"C-"};
{\ar@/_1pc/@{.>}_{F.E.~\beta(\chi_s)} "C+";"C-"};
{\ar@{<-}^{|\cdot|^{n+1}}_{\simeq} "C-";"B-"};
{\ar@{}|(0.52){\displaystyle{\subset}} "B-";"A-"};
 {\ar@{}|(0.47){\displaystyle{\supset}} "Z1+";"Z2+"};
{\ar@{}|(0.52){\displaystyle{\subset}} "Z2-";"Z1-"};
 {\ar@{<->}^{\CM} "A+";"Z1+"};
 {\ar@{<->}^{\CM} "B+";"Z2+"};
 {\ar@{<->}^{\CM} "A-";"Z1-"};
 {\ar@{<->}^{\CM} "B-";"Z2-"};
 \end{xy}
}
\]
where $F.I.$ stands for the fiber integration defined in Definition \ref{schwartzF}, $F.E.$ stands for the functional equation for $\beta(\chi_s)$ in Theorem \ref{thm:LT}, $\CM$ is the Mellin transform as defined in
\eqref{mellin}, and $\CC_c^\infty(S_{2n+1})=\CC_c^\infty(\Sym^2(F^{2n+1}))$. We refer to Sections \ref{sec-FEbeta} and \ref{sec-FTbeta} (see also Section \ref{ssec-sp}) for other unexplained notations here.

\subsection{Structure of the paper}\label{ssec-sp}

We recall the basic local theory of the zeta integrals of Piatetski-Shapiro and Rallis in Section \ref{sec-LT-DZI}.
Section \ref{ssec-lzidm} is to review briefly, the main results of Piatetski-Shapiro and Rallis
on the analytic properties of their local zeta integrals with connection to the local $L$-functions and the local $\gamma$-functions of $\Sp_{2n}$ via their doubling method (\cite{GPSR87} and \cite{PSR86}).
Based on the work of E. Lapid and S. Rallis (\cite{LR05}),
of T. Ikeda (\cite{Ik92}), of S. Yamana (\cite{Y14}) and of H. Kakuhama (\cite{Kk18}), Section \ref{ssec-lgf} discusses briefly the right normalization of the local intertwining operator $\RM_{w_\Del}(s,\chi)$, so that
the local functional equation of the local doubling zeta integrals yields the Langlands local $\gamma$-functions for the standard $L$-functions of $\Sp_{2n}$ (Corollary \ref{FE-Langlands gamma}).

Sections \ref{sec-FEbeta} and \ref{sec-FTbeta} is devoted to develop harmonic analysis on $F^\times$ associated to the $\gamma$-functions $\beta_\psi(\chi_s)$. Section \ref{ssec-zi-sm} reviews the functional
equation for the local zeta integrals associated to the prehomogeneous vector space $(\GL_m(F),\Sym^2(F^m))$, following the work of Piatetski-Shapiro and Rallis (\cite[Appendix 1]{PSR87}) and
of Ikeda (\cite{Ik17}) (Proposition \ref{pro:zeta-pvs}).
In Section \ref{ssec-zi-io}, we discuss analytic properties of the local intertwining operators $\RM_{w_\Del}(s,\chi)$
in terms of those of the local zeta integrals associated to $(\GL_{2n+1}(F),\Sym^2(F^{2n+1}))$ (Propositions \ref{pro:pole-Mw} and \ref{pro:analyticalgoodsection}).
In Section \ref{ssec-fife}, by using fiber integrations along the $F$-morphism $\det$ as displayed in \eqref{det}, we deduce a preliminary version of the local functional equation on $\GL_1(F)$
(Corollary \ref{fe-Radon}) from that in Proposition \ref{lm:Ikeda} (which is deduced from \cite[Theorem 2.1]{Ik17}).

In Section \ref{ssec-bbf}, in order to understand the functions on
$F^\times$ obtained by the fiber integration through the determinant morphism $\det$ (Proposition \ref{pro-2fi}), we introduce two spaces of functions $\CS_n^\pm(F^\times)$ in Definitions \ref{CS+} and \ref{CS-},
which are characterized by their
asymptotic behaviors near the boundary point $x=0$ of $F^\times$, respectively. Their behaviors under the Mellin transform $\CM(s,\chi)$ can be deduced from \cite[Theorem 5.3]{Ig78}
(Proposition \ref{Igusa78-53}). Their image spaces under the Mellin transforms are the space $Z_n^\pm(\wh{F^\times})$, as given in Definition \ref{Z+-}, in which the functions
are characterized by their possible poles. With these spaces of functions, we are able to re-normalize the preliminary functional equation for $\beta_\psi(\chi_s)$ in Corollary \ref{fe-Radon} to obtain
the functional equation for $\beta_\psi(\chi_s)$ in Theorem \ref{thm:LT}, with (normalized) Schwartz spaces $\CS_\pvs^\pm(F^\times)$ and the Fourier operator $\FL_{\eta_{\pvs,\psi}}$
from the space $\CS_\pvs^+(F^\times)$ to
the space $\CS_\pvs^-(F^\times)$. This is done in Section \ref{ssec-ssgft}. The functional equation in Theorem \ref{thm:LT}, as displayed in \eqref{FE-eta}, explains the deep symmetry of the
$\gamma$-functions $\beta_\psi(\chi_s)$
in terms of harmonic analysis on $F^\times$.
We continue with Section \ref{ssec-etapvs} to construct in an explicit and analytic way a {\sl generalized function} $\eta_{\pvs,\psi}(x)$ on $F^\times$
as in \eqref{def:eta-rho-psi},
such that the generalized Fourier transform $\FL$ (as defined in \eqref{LT})
from the space $\CS_\pvs^+(F^\times)$ to the space $\CS_\pvs^-(F^\times)$ becomes a Fourier operator $\FL_{\eta_{\pvs,\psi}}$ with $\eta_{\pvs,\psi}(x)$ as
the kernel function in
Theorem \ref{thm:eta}, as displayed in \eqref{LT-eta}. Theorem \ref{thm:eta} is another technical key result in the paper. In Section \ref{ssec-gcd}, we characterize the spaces $\CS_\pvs^\pm(F^\times)$ by their
image spaces $\CZ_{n,\beta}^\pm(\wh{F^\times})$ under the Mellin transform (Theorem \ref{thm:PWM}), which is a version of the classical Paley-Wiener Theorem for Mellin transform. The spaces
$\CZ^{\pm}_{n,\bet}(\wh{F^\times})$ are defined in Definition \ref{defin:Mellin-gcd} by means of their possible poles.

In Section \ref{sec-etapvs-FT}, we are ready to define (Definition \ref{dfn:FTX}) the Fourier transform $\CF_{X,\psi}$ over $X_{P_\Del}(F)$,
by using the generalized function $\eta_{\pvs,\psi}$ on $F^\times$ in Section \ref{sec-FTbeta}. In Section \ref{ssec-FoL2}, we show that the integral in \eqref{FTX} defining the Fourier transform $\CF_{X,\psi}$
converges in Corollary \ref{FX-conv},
by using properties of the Radon transform $\RR_X$ (the ${\rm d} n$-integration part in \eqref{FTX}) (Lemma \ref{lem:Rx} and Proposition \ref{pro:Fg}). Theorem \ref{thm:CFP-M}
shows that the Fourier transform $\CF_{X,\psi}$ is compatible with $\beta_\psi(\chi_s)\cdot\RM_{w_\Del}(s,\chi)$ with respect to the projections $\CP_{\chi_s}$ and $\CP_{\chi_s^{-1}}$.
Then we extend $\CF_{X,\psi}$ to the space $L^2(X_{P_\Del})$ of square-integrable functions on $X_{P_\Del}(F)$ (Proposition \ref{pro:Fx-norm}, Corollary \ref{Fx-invs}, and Remark \ref{FXunitary}).
In Section \ref{ssec-FoSs}, we introduce the Schwartz space $\CS_\pvs(X_{P_\Del})$ (Definition \ref{dfn:SSFX}) and show that the Schwartz space is $\CF_{X,\psi}$-stable  and
produces the space of {\sl good sections} $\RI^\dag(s,\chi)$ via the projection $\CP_{\chi_s}$ (Proposition \ref{pro:FTX}). Theorem \ref{thm:asymSf} characterizes the functions in $\CS_\pvs(X_{P_\Del})$ by means of their asymptotic behavior on the $M_\Del^\ab$-part.

Section \ref{sec-FOG} is to develop the harmonic analysis on $G(F)$ and prove Theorem \ref{thm:L2} on the Fourier operator $\CF_{\rho,\psi}$ on $G(F)$. Section \ref{sec-MOGF} is to define the gamma function
$\Gamma_{\rho,\psi}(s,\chi\otimes\pi)$ and establish its desired properties. Sections \ref{sec-FOG} and \ref{sec-MOGF} complete the theory of harmonic analysis and gamma functions associated to the $G(F)$-invariant
distribution $\Phi_{\rho,\psi}$ on $G(F)$.

Section \ref{sec-Thm-1-2-3} is to show that the gamma function $\Gamma_{\rho,\psi}(s,\chi\otimes\pi)$ as defined in Proposition \ref{pro:lfe} is equal to the Langlands $\gamma$-function
$\gam(s,\chi\otimes\pi,\rho,\psi)$ (Corollary \ref{gam-gam}) and to complete the proof of
Theorems \ref{thm:LGF}, \ref{thm:LFE-CF-BK}, and \ref{thm:LLF}. The proofs in this section depend on the local theory of the Piatetski-Shapiro and Rallis zeta integrals.

Although most of the arguments in the paper automatically covers the case when $n=0$, which is the local theory of the Tate thesis (\cite{Tt50}), we may assume in this paper that $n>0$ in order to be more focused
on the essence of the work.

It takes us a long period of time to make progress and finish up this part of the local theory. We would like to thank C. P. Mok for a series of lectures in the one-week workshop at University of Minnesota in May, 2013,
on the relevant topics, which refreshed our interests in the Braverman-Kazhdan proposal. We have benefited from our participation in the two workshops at the American Institute of
Mathematics (Automorphic Kernel Functions, 2015, and Functoriality and the Trace Formula, 2017). We would like to thank the Institute for providing productive discussion environment and hospitality during the workshops.
During the 2017 Workshop, the authors presented their preliminary calculation of the transfer from the left-hand side of Diagram \eqref{2OpenX} to its right-hand side in a discussion group.
We would like to thank Zhiwei Yun for his clear explanation of the geometric structure behind our calculations, and thank Ng\^o and Shahidi for their helpful conversations and encouragement.
We would also like to thank the organizers Altug, Arthur, Casselman, and Kaletha for inviting us to the 2017 Workshop. During the one-month program on the Langlands Program: Endoscopy and Beyond at the Institute for Mathematical Sciences, National University of Singapore, from December 17, 2018 to January 19, 2019, we got together and made partial progress on this project. We would like to thank the Institute for invitation and hospitality during the program.
We would also like to thank Ng\^o for inviting us to his two-week program in June 2019, where we gave more detailed lectures on our progress and the main obstacles in this project, and for his explanation of
possible relation of our work with his own construction and that of L. Lafforgue. We are grateful to H. Jacquet for his encouraging historical comments on our work and his work with R. Godement.
Thank the referee for  many helpful comments and suggestions, which greatly improved the exposition of this paper.


\section{Local Theory of Piatetski-Shapiro and Rallis}\label{sec-LT-DZI}


Let $F$ be a local field of characteristic zero, which is a finite extension of $\BQ_p$ for some prime $p$. Let $\psi$ be a non-trivial additive character of $F$. We consider
$G=\BG_m\times\Sp_{2n}$, where $\Sp_{2n}$ is the symplectic group associated to the non-degenerate symplectic space $(V,\RJ_n)$ with
$$
\RJ_n=\begin{pmatrix}0&\RI_n\\-\RI_n&0\end{pmatrix}.
$$
Fix a basis $\{e_1,e_2,\cdots,e_{2n-1},e_{2n}\}$ of $V$ with respect to the symplectic form $\langle\cdot,\cdot\rangle_{\RJ_n}$ defined by $\RJ_n$. It follows that
$\langle e_i,e_{j}\rangle_{\RJ_n}=-\langle e_j,e_{i}\rangle_{\RJ_n}$ for any $i$ and $j$, and also $\langle e_i,e_{j}\rangle_{\RJ_n}=\langle e_{n+i},e_{n+j}\rangle_{\RJ_n}=0$ and
$\langle e_i,e_{n+j}\rangle_{\RJ_n}=\delta_{i,j}$ for $i,j=1,2,\cdots,n$.
 The symplectic group $\Sp(V)=\Sp_{2n}$ acts on the vector space by the left action: $v\mapsto  g\cdot v$ for any $v\in V$ and any $g\in\Sp_{2n}$.
Under the fixed basis, $\Sp_{2n}(F)$, as matrices,  acts on row vectors $\vec{v}$, as the coordinate of $v$, by the right
multiplication.

We fix a flag of totally isotropic subspaces of $V$:
$$
\{0\}\subset\{e_{n+1}\}\subset\{e_{n+1},e_{n+2}\}\subset\cdots\subset\{e_{n+1},e_{n+2},\cdots,e_{2n}\}\subset V,
$$
which determines a Borel subgroup $B=TU$ of $\Sp_{2n}$, with a maximal $F$-split torus $T$ and the unipotent radical $U$. The elements of $T(F)$ are of the form
\begin{equation}
 t:=\diag(t_{1},\cdots,t_n,t_{1}^{-1},\cdots,t_n^{-1}),
\end{equation}
with $t_i\in F^\times$, and the elements of $U(F)$ are of the form
$$
u=\begin{pmatrix}Z&X\\ 0&Z^*\end{pmatrix}
$$
where $Z$ is an upper-triangular maximal unipotent matrix in $\GL_n(F)$ and $Z^*={^t\!Z}^{-1}$, and $X$ is an $n\times n$-matrix with $Z\cdot{^t\!X}=X\cdot{^t\!Z}$. The standard Siegel parabolic subgroup
is the stabilizer of the totally isotropic partial flag or a Lagrangian of $V$
$$
\{0\}\subset\{e_{n+1},e_{n+2},\cdots,e_{2n}\}\subset V.
$$
The general linear group $\GL_n(F)$ is embedded as the Levi subgroup by
\begin{equation}\label{siegel-levi}
a\in\GL_n(F) \mapsto \begin{pmatrix}a&0\\0&a^*\end{pmatrix}\in\Sp_{2n}(F)
\end{equation}
where $a^*={^t\!a^{-1}}$.

Let $\CO=\CO_F$ be the ring of integers of $F$ and $\CP=\CP_F$ be the maximal ideal of $\CO_F$. Take $K_{2n}:=\Sp_{2n}(\CO)$ to be the fixed maximal open compact subgroup of $\Sp_{2n}(F)$.
The Iwasawa decomposition is given by $\Sp_{2n}(F)=B(F)K_{2n}$. Let ${\rm d} g$ be the Haar measure on $\Sp_{2n}(F)$ such that $K_{2n}$ has volume one. Let ${\rm d} u$ be the Haar measure on $U(F)$ such that $U(\CO)$ has volume one.
Take the Haar measure ${\rm d} t$ on $T(F)$ such that $T(\CO)$ has volume one.
For any integrable function $f(g)$ on $\Sp_{2n}(F)$, one has
\begin{eqnarray}
\int_{G(F)}f(g)\ud g
&=&
\int_{K_{2n}}\int_{T(F)}\int_{U(F)}f(utk) \delta_B(t)^{-1}\ud u \ud t \ud k\nonumber\\
&=&
\int_{K_{2n}}\int_{U(F)}\int_{T(F)}f(tuk)\ud t \ud u \ud k,
\end{eqnarray}
where $\delta_B(t)=|\det\Ad (t)\vert_{\Fu}|_F$ with $\Fu$ being the Lie algebra of $U$.

\subsection{Local zeta integrals of doubling method}\label{ssec-lzidm}

We recall the local zeta integrals introduced by Piatetski-Shapiro and Rallis via the doubling method (\cite[Part A]{GPSR87}).
The geometric structure of the doubling method for symplectic groups can be briefly described as follows.

Define a $4n$-dimensional symplectic vector space given by $$W=V^{+}\oplus V^{-},$$
where $V^{+}$ is the given $2n$-dimensional symplectic vector space $V$ with symplectic form $\langle\cdot,\cdot\rangle_{\RJ_n}$, and $V^{-}$ is the $2n$-dimensional
symplectic vector space with symplectic form $\langle\cdot,\cdot\rangle_{-\RJ_n}$. The non-degenerate symplectic form on $W$ is given by
$\langle\cdot,\cdot\rangle_{\RJ_n} \oplus \langle\cdot,\cdot\rangle_{-\RJ_n}$.
Then the diagonal embedding of $V$ into $W$ has the image $$L_{\Del} = \{ (v,v)\in W\mid   v\in V \}$$ which is a Lagrangian subspace of $W$. Let $P_{\Del}$ be the stabilizer of $L_{\Del}$, which
is a Siegel parabolic subgroup of $\Sp(W)$ with the unipotent radical $N_{\Del}$ and the associated longest Weyl element $w_{\Del}$ that takes $P_{\Del}$ to its opposite.
More precisely, following the choice in \cite{LR05}, we take
\begin{align}\label{wDel0}
w_\Del=(\RI_{V^{+}},-\RI_{V^{-}})\in \Sp(V^+)\times\Sp(V^-)
\end{align}
where $\RI_{V^\pm}$ is the identity element in $\Sp(V^\pm)$, respectively.
With the natural embedding of $\Sp(V^+)\times \Sp(V^-)$ into $\Sp(W)$, the subgroup
$\Sp(V^+)\times \Sp(V^-)$ acts on $W$ by
\begin{equation}\label{action-2n2n}
(g_{1},g_{2})\cdot (v_1,v_2) = (g_1\cdot v_1,g_2\cdot v_2)
\end{equation}
for $(v_1,v_2) \in W$ and for $(g_1,g_2)\in\Sp(V^+)\times \Sp(V^-)$. It follows that under this action, the stabilizer of the diagonal Lagrangian subspace
$L_{\Del}$ is $\Sp(V)^{\Del}$, which is the image of the diagonal embedding of $\Sp(V)$ in $\Sp(V^+)\times \Sp(V^-)$. As proved in \cite[Part A]{GPSR87}, the double coset
\begin{align}\label{z-open}
P_\Del\cdot(\Sp(V^+)\times\Sp(V^-))=P_\Del\cdot(\Sp(V^+)\times\{\RI_{V^-}\})
\end{align}
is Zariski open dense in $\Sp(W)$. It is clear that the Levi subgroup $M_\Del$ of $P_{\Del}$ is isomorphic to $\GL(L_\Del)$.

We fix the ordered basis $\{ e_{1},e_{2},...,e_{2n} \}$ for $V^{+}$. For $V^{-}$, we take the same basis, but denote it by
$\{ f_{1},f_{2},...,f_{2n} \}$. It follows that
$$\langle f_i,f_{j}\rangle_{-\RJ_n}=\langle f_{n+i},f_{n+j}\rangle_{-\RJ_n}=0\quad {\rm and} \quad\langle f_i,f_{n+j}\rangle_{-\RJ_n}=-\delta_{i,j}
$$
for $i,j=1,2,\cdots,n$.
Accordingly, we take an ordered basis of $W=V^{+}\oplus V^-$ as follows:
\begin{equation}\label{sbasis}
\{ e_{1}, e_{2},...,e_{n}, f_{1}, f_{2},...,f_{n}, e_{n+1},e_{n+2},...,e_{2n}, -f_{n+1},-f_{n+2},...,-f_{2n}\}.
\end{equation}
It is clear that the matrix defining the symplectic form $\langle \cdot,\cdot\rangle_{\RJ_n\oplus-\RJ_n}$ of $W$ is $\RJ_{2n}$, with the ordered basis.
The embedding of $\Sp(V^+)\times \Sp(V^-)$ into $\Sp(W)$ as defined in \eqref{action-2n2n} has the following explicit expression for $\Sp_{2n}(F)\times\Sp_{2n}(F)$ embedded into $\Sp_{4n}(F)$:
\begin{align*}
(
\left(
\begin{smallmatrix}
A	&	B\\
C	&	D
\end{smallmatrix}
\right)
,
\left(
\begin{smallmatrix}
M	&	N\\
P	&	Q
\end{smallmatrix}
\right)
)
\mapsto
\left(
\begin{smallmatrix}
A 	&	&	B	&	\\
	&M	&		&-N	\\
C 	&	&	D	&	\\
 	&-P	&		&Q	\\
\end{smallmatrix}
\right).
\end{align*}
Especially, the group $\Sp_{2n}(F)\times \{\RI_{2n}\}$ embeds into $\Sp_{4n}(F)$ by
\begin{align}\label{embsp}
(
\left(
\begin{smallmatrix}
A	&	B\\
C	&	D
\end{smallmatrix}
\right)
,
\left(
\begin{smallmatrix}
\RI_{n}	&	\\
	&	\RI_{n}
\end{smallmatrix}
\right)
)
\mapsto
\left(
\begin{smallmatrix}
A 	&	&	B	&	\\
	&\RI_{n}	&		&	\\
C 	&	&	D	&	\\
 	&	&		&\RI_{n}	\\
\end{smallmatrix}
\right);
\end{align}
and
\begin{equation}\label{wDel}
w_\Del= \left(\begin{smallmatrix}
\RI_n&&&\\
&-\RI_n&&\\
&&\RI_n&\\
&&&-\RI_n 	
\end{smallmatrix}\right).	
\end{equation}
In the standard Lagrangian $L_{\std}$ in $W$, we may take an ordered basis
$$
\{ e_{n+1},e_{n+2},...,e_{2n},-f_{n+1},-f_{n+2},...,-f_{2n} \}.
$$
The stabilizer of $L_{\std}$ is the standard Siegel parabolic subgroup $P_{\std}$ in $\Sp(W)$. With the ordered basis \eqref{sbasis},  $P_{\std}(F)$ consists of matrices of the following form in $\Sp_{4n}(F)$,
\begin{equation}\label{siegelp}
\left(
\begin{matrix}
A	&	X	\\
0	&	A^*
\end{matrix}
\right),
\end{equation}
with $A\in\GL_{2n}(F)$ and $A^*={^t\!A^{-1}}$.
On the other hand, the Lagrangian subspace $L_{\Del}$, corresponding to the diagonal embedding of $V$ into $W$, has an ordered basis
$$
\{ e_{1}+f_{1},e_{2}+f_{2},...,e_{2n}+f_{2n} \}.
$$
Under the choice of the ordered basis for the $4n$-dimensional symplectic vector space $W$,
we have the following diagram for $F$-rational points of relevant algebraic groups (i.e. matrices over $F$) and the relevant varieties:
\begin{equation}\label{twoOpen}
\begin{matrix}
&&&&\\
&&\Sp_{4n}&&\\
&&&&\\
&&\downarrow&&\\
&&&&\\
P_\Del\bks P_\Del w_\Del N_\Del&\rightarrow& P_\Del\bks\Sp_{4n}&\leftarrow&P_\Del\bks P_\Del(\Sp_{2n}\times\{\RI_{2n}\})\\
&&&&
\end{matrix}
\end{equation}
with both $P_\Del\bks P_\Del w_\Del N_\Del$ and $P_\Del\bks P_\Del(\Sp_{2n}\times\{\RI_{2n}\})$ being Zariski open in $P_\Del\bks\Sp_{4n}$.

In order to find the explicit relation between $P_\Del$ and $P_{\std}$, which is important to the explicit calculation for the transformation from the left-hand side of the Diagram in \eqref{twoOpen} to the right-hand side,
we take
\begin{equation}\label{g0}
g_{0} =
\left(
\begin{smallmatrix}
	0&	0	&\frac{-\RI_{n}}{2}&\frac{-\RI_n}{2}\\
 \frac{\RI_n}{2}& \frac{-\RI_{n}}{2}&0	&	0	\\
\RI_{n}	&\RI_{n}	&	0&	0\\
0	&	0	&\RI_{n}&-\RI_{n}
\end{smallmatrix}
\right)\in\Sp_{4n}(F).
\end{equation}
By a direct verification, $g_{0}$ takes $L_{\std}$ to $L_{\Del}$.
It follows that as subgroups of $\Sp_{4n}(F)$, one has that $P_{\Del}(F) = g_{0}^{-1}P_{\std}(F)g_{0}$, with
\begin{equation}\label{g0-1}
g_{0}^{-1} =
\left(
\begin{smallmatrix}
    0 &     \RI_{n} & \frac{\RI_{n}}{2} &  0             \\
    0  &    -\RI_{n}  &	\frac{\RI_n}{2}&    0   \\
     -\RI_{n} &  0    &   0	&  \frac{\RI_n}{2}    \\
    -\RI_{n} &0	   &		0 & \frac{-\RI_{n}}{2}
\end{smallmatrix}
\right).
\end{equation}
Similar to the definition in \cite[Page 314]{LR05},
we normalize the abelianization morphism
\begin{equation}\label{M-Mab}
\begin{matrix}
\Fa&:&M_\Del&\rightarrow &M_\Del^{\ab}&\cong&\BG_m\\
&&&&&&\\
   & &m_\Del&\mapsto&\ol{m_\Del}=[M_\Del,M_\Del]m_\Del&\mapsto&\det_{L_\Del}(m_\Del)^{-1}.
\end{matrix}
\end{equation}
More precisely, for $A\in \GL_{2n}(F)$, we write
$m_{\std}(A)=\left(\begin{smallmatrix}
A&0\\0&A^*	
\end{smallmatrix}\right)
\in M_{\std}(F)$.  Then for
$$
m_\Del=m_\Del(A)=g_0^{-1}m_{\std}(A)g_0\in M_{\Del}(F),
$$
we have $\Fa(m_\Del(A))=\det(A)$. It is clear that
the modular character is calculated as follows:
\begin{align}\label{dPDel}
\delta_{P_\Del}(m_\Del(A))=|\Fa(m_\Del(A))|^{2n+1}=|\det(A)|^{2n+1}.
\end{align}
We define a section of $\Fa$ to be
\begin{equation}\label{section:Fs}
\Fs\ :\ \BG_m\cong M_\Del^{\ab}\rightarrow M_\Del,\quad a\mapsto \Fs_{a}
\end{equation}
with the property that $\Fa(\Fs_{a})=a$ for any $a\in\BG_m(F)$. For instance, we may take $\Fs_{a}=g_0^{-1}m_{\std}(\diag(a,\RI_{2n-1}))g_0$.
For a character $\chi$ of $F^\times$, define
\begin{equation}\label{chis}
\chi_{s}(m_\Del(A)): = \chi(\Fa(m_\Del(A)))|\Fa(m_\Del(A))|^s=\chi(\det(A))|\det(A)|^{s}.
\end{equation}
As usual, we define the normalized smooth induced representation of $\Sp_{4n}(F)$: $\RI(s,\chi) := \Ind^{\Sp_{4n}(F)}_{P_{\Del}(F)}(\chi_{s})$.
For the Weyl element $w_{\Del}$, which takes $P_\Del$ to its opposite, the local intertwining operator $\RM_{w_{\Del}}(s,\chi)$
from $\RI(s,\chi)$ to $\RI(-s,\chi^{-1})$ is defined by
\begin{equation}\label{lio-0}
\RM_{w_{\Del}}(s,\chi)(f_{\chi_{s}})(g) = \int_{N_{\Del}(F)}f_{\chi_{s}}(w_{\Del}n_\Del g)\ud n_\Del.
\end{equation}
The operator $\RM_{w_{\Del}}(s,\chi)$ is defined for $\Re(s)$ sufficiently positive, and has meromorphic continuation to the whole complex plane.

For any irreducible admissible representation $\pi$ of $\Sp_{2n}(F)$ with its contragredient $\wt{\pi}$,
the {\sl local zeta integrals $\RZ(f_{\chi_s},\varphi_{v_{\pi},v_{\wt{\pi}}})$ of Piatetski-Shapiro and Rallis }(\cite[Part A]{GPSR87} and \cite{PSR86}) is defined by
\begin{align}\label{doublingzeta}
\RZ(f_{\chi_{s}},\varphi_{v_\pi,v_{\wt{\pi}}})
= \int_{\Sp_{2n}(F)}f_{\chi_{s}}(g,1)\varphi_{v_{\pi},v_{\wt{\pi}}}(g)\ud g,
\end{align}
with $\varphi_{v_{\pi},v_{\wt{\pi}}}(g):=\langle v_{\wt{\pi}},\pi(g)v_{\pi}\rangle$ being the matrix coefficient associated to
$v_{\pi}\in\pi$ and $v_{\wt{\pi}}\in\wt{\pi}$, and $f_{\chi_{s}}\in \RI (s,\chi)$ being any section.

The local theory for $\RZ(s,f_{\chi_s},\phi_{v_{\pi},v_{\wt{\pi}}})$
can be summarized in the following theorem
of Piatetski-Shapiro and Rallis (\cite{PSR86} and \cite[Part A]{GPSR87}). We take $K_\Del=g_0^{-1}K_{4n}g_0$, which is a maximal open compact subgroup of $\Sp_{4n}(F)$, such that $\Sp_{4n}(F)=P_\Del(F)\cdot K_\Del$
is the Iwasawa decomposition of $\Sp_{4n}(F)$. When the residue characteristic of $F$ is odd, we have that $g_0\in K_{4n}=\Sp_{4n}(\CO)$.

\begin{thm}[Piatetski-Shapiro and Rallis]\label{thm:PSR86}
Let $\pi$ be an irreducible admissible representation of $\Sp_{2n}(F)$ and $\wt{\pi}$ be its contragredient.
Assume that $f_{\chi_{s}}\in \RI (s,\chi)$ is a holomorphic section, i.e. its restriction to $K_\Del \times \BC$ is holomorphic. Let $\varphi_{v_{\pi},v_{\wt{\pi}}}$ be the matrix coefficient associated to
$v_{\pi}\in\pi$ and $v_{\wt{\pi}}\in\wt{\pi}$.
\begin{enumerate}
\item The local zeta integral $\RZ(f_{{\chi}_{s}},\varphi_{v_{\pi },v_{\wt{\pi}}})$ as defined in \eqref{doublingzeta} converges absolutely when $\Re(s)$ is sufficiently positive,
and admits meromorphic continuation to $s\in \BC$.
\item When $\pi$ and $\chi$ are unramified, the integral $\RZ (f_{{\chi}_{s}},\varphi_{v_{\pi },v_{\wt{\pi}}})$ computes the unramified local $L$-function $L(s+\frac{1}{2},\pi \times\chi ,\std)$.
\item The $\Hom$-space
$
\Hom_{\Sp_{2n}(F )\times \Sp_{2n}(F)}(\RI (s,\chi ), \pi \otimes \wt{\pi})
$
is at most one dimensional for all but a finite number of $q^{-s}$.
\item There exists a meromorphic function $\Gam_{\PSR}(\pi ,\chi ,s)$ so that
the local functional equation
$$
\RZ (\RM_{w_{\Del} }(s,\chi )(f_{\chi_{s }}),\varphi_{v_{\pi },v_{\wt{\pi}}})
 = \Gam_{\PSR}(\pi ,\chi ,s)  \RZ (f_{\chi_{s }},\varphi_{v_{\pi },v_{\wt{\pi}}})
$$
holds. Here $\Gam_{\PSR}(\pi ,\chi ,s)$ denotes the local $\gamma$-function defined by Piatetski-Shapiro and Rallis in this case.
\end{enumerate}
\end{thm}


\subsection{Local $\gamma$-functions}\label{ssec-lgf}
One has to compare the Piatetski-Shapiro and Rallis local $\gamma$-function $\Gam_{\PSR}(\pi ,\chi ,s)$ and the Langlands local
$\gamma$-function $\gam(s,\chi\otimes\pi,\rho,\psi )$, with $\rho$ as defined in \eqref{rho}.
This has been explicitly discussed
by Lapid and Rallis in \cite{LR05} (and also \cite{Ik92} and \cite{Kk18}).

Let $\psi$ be a non-trivial additive character of $F$. According to \cite{Ik92}, \cite{LR05} and \cite{Kk18}, one defines the {\sl normalized} local intertwining operator:
\begin{equation}\label{eq:IT-eta-rho}
\RM^\dag_{w_{\Del}}(s,\chi,\psi):=\chi_{s}(2)^{2n}|2|^{n(2n+1)}\beta_{\psi}(\chi_s)\cdot	
\RM_{w_{\Del}}(s,\chi),
\end{equation}
with the function $\beta_{\psi}(\chi_s)$ given by
\begin{equation}\label{eq:eta-intertwining}
\beta_{\psi}(\chi_s)=\gam(s-\frac{2n-1}{2},\chi,\psi)\prod_{r=1}^{n}\gam(2s-2n+2r,\chi^{2},\psi),
\end{equation}
a finite product of local $\gamma$-functions of abelian type. Note that the factor $\chi_{s}(2)^{2n}|2|^{n(2n+1)}$ occurs due to the geometry of the maximal parabolic subgroup $P_\Del$ relative to the standard
maximal parabolic subgroup $P_{\std}$. With this normalized local intertwining operator $\RM^\dag_{w_{\Del}}(s,\chi,\psi)$, the local functional equation in Theorem \ref{thm:PSR86}, Part (4) can be re-written as follows.

\begin{cor}\label{FE-Langlands gamma}
The local zeta integral $\RZ(f_{\chi_{s}},\phi)$ of Piatetski-Shapiro and Rallis satisfies the functional equation
\begin{equation}\label{eq:FE-dag}
\RZ(\RM^\dag_{w_{\Del}}(s,\chi,\psi)f_{\chi_{s}}, \varphi_{v_{\pi },v_{\wt{\pi}}}) =\pi(-1) \gamma(s+\frac{1}{2},\chi\otimes\pi,\rho,\psi)\RZ(f_{\chi_{s}}, \varphi_{v_{\pi },v_{\wt{\pi}}}), 		
\end{equation}
with the Langlands $\gamma$-function $\gamma(s+\frac{1}{2},\chi\otimes\pi,\rho,\psi)$.
\end{cor}

From the functional equation in Theorem \ref{thm:PSR86}, Part (4) and that in Corollary \ref{FE-Langlands gamma}, one can easily deduce the relation between the Piatetski-Shapiro and Rallis local $\gamma$-function
$\Gam_{\PSR}(\pi ,\chi ,s)$ and the Langlands local
$\gamma$-function $\gam(s,\chi\otimes\pi,\rho,\psi)$:
\[
\gam(s,\chi\otimes\pi,\rho,\psi)=\pi(-1)\chi_{s}(2)^{2n}|2|^{n(2n+1)}\beta_{\psi}(\chi_s)\cdot\Gam_{\PSR}(\pi ,\chi ,s).
\]

Note that the normalized intertwining operator $\RM^\dag_{w_{\Del}}(s,\chi,\psi)$
depends only on a choice of additive characters $\psi$.
Moreover, by \cite[Appendix B.3]{Y14} (and also \cite[Lemma 1.1]{Ik92}), one deduces that
\begin{equation}\label{eq:M-FE}
\RM^\dag_{w_{\Del}}(-s,\chi^{-1},\psi^{-1})\circ \RM^\dag_{w_{\Del}}(s,\chi,\psi)={\rm Id},
\end{equation}
which is an identity required for the normalization of local intertwining operators.
As before, we take the maximal compact subgroup $K_\Del=g_0^{-1}\Sp_{4n}(\CO)g_{0}$ of $\Sp(W)$. When $\chi$ is unramified, define
$f^\circ_{\chi_s}$ to be the $K_\Del$-invariant function  in $\RI(s,\chi)$ that is normalized by $f^\circ_{\chi_s}(\RI_{4n})=1$.
If $\psi$ is of conductor $\CO$, recall from  \cite[Proposition 3.1 (5)]{Y14} that
\begin{equation}\label{eq:M-Unramified}
\RM^\dag_{w_{\Del}}(s,\chi,\psi)(b_{2n}(s,\chi)f^\circ_{\chi_s})=b_{2n}(-s,\chi^{-1})f^\circ_{\chi_s^{-1}},	
\end{equation}
where
$b_{2n}(s,\chi)=L(s+\frac{2n+1}{2},\chi)\prod^{n}_{r=1}L(2s+2r-1,\chi^2)$.


\section{Functional Equation for $\beta_\psi(\chi_s)$}\label{sec-FEbeta}


In the spirit of the Braverman-Kazhdan proposal, we study the finite product of local gamma factors of abelian type:
\[
\beta_{\psi}(\chi_s)=\gam(s-\frac{2n-1}{2},\chi,\psi)\prod_{r=1}^{n}\gam(2s-2n+2r,\chi^{2},\psi),
\]
as defined in \eqref{eq:eta-intertwining}, in terms of the local harmonic analysis on $F^\times$. To this end, we are going to find a pair of distributions on $F^\times$,
which are of the Mellin transform type and satisfy a functional equation with
$\beta_{\psi}(\chi_s)$ as the proportional factor, i.e. the gamma factor of the functional equation (Theorem \ref{thm:LT}).
We deduce the two distributions and the associated functional equation with $\beta_\psi(\chi_s)$ from the theory
of the local zeta integrals associated to prehomogeneous vector space (\cite{Ki03}) that consists of a pair
$$
(\GL_{2n+1}(F),\Sym^2(F^{2n+1})),
$$
and the corresponding functional equation. This theory was developed through the work of  G. Shimura (\cite{S97}), J. Igusa (\cite{Ig78} and \cite{Ig00}), Piatetski-Shapiro and Rallis (\cite{PSR87}),
and T. Ikeda (\cite{Ik17}). By using the theory of fiber
integrations as developed by A. Weil (\cite{W65}), we push the theory on $\Sym^2(F^{2n+1})$ down to a theory on $\BG_m(F)=F^\times$ through the algebraic morphism
$$
\det\ :\ \Sym^2(F^{2n+1})\rightarrow\BG_a(F)=F.
$$
We review the theory on $\Sym^2(F^{2n+1})$ in Section \ref{ssec-zi-sm} and develop harmonic analysis on $F^\times$ for the $\gamma$-functions $\beta_\psi(\chi_s)$ through Sections \ref{ssec-bbf} and \ref{ssec-ssgft}.

\subsection{Zeta integrals associated to symmetric matrices}\label{ssec-zi-sm}

It is well known that the space of $m\times m$-symmetric matrices with the natural $\GL_m$-action is a prehomogeneous vector space defined over the algebraic closure $\ovl{F}$ of $F$ (\cite{Ki03}).
Denote by $S_{m}(F):=\Sym^2(F^m)$ the space of all $m\times m$-symmetric matrices with entries in $F$.
The general linear group $\GL_{m}(F)$ acts on $S_{m}(F)$ by $g\cdot X:=g\cdot X\cdot{^t\!g}$.
Let $S_{m}^\circ$ be the set of all $\GL_{m}(F)$-orbits of $S_{m}(F)$ of maximal dimension (i.e. the semistable orbits).
For $X\in S_{m}(F)$ with $\det(X)\ne 0$, write $S_{m}^X$ to be the $\GL_{m}(F)$-orbit containing $X$. It is clear that $S_{m}^X\in S_{m}^\circ$, if $\det(X)\ne 0$. With fixed representatives, there exist
finitely many elements $X_{1},\cdots,X_{t}\in S^{\circ}_{m}$ such that $S^{\circ}_{m} = \bigsqcup_{i=1}^{t} S^{X_{i}}_{m}$.

For any $\Phi\in \CC^\infty_c(S_{m})$, the space of smooth, compactly supported functions on $S_m(F)$,
the local zeta integral $\CZ_{S_m}(s,\Phi,\chi)$ is defined, as in \cite{Ik17}, by
\begin{equation}\label{lzi-sm}
\CZ_{S_m}(s,\Phi,\chi)=\int_{S_m(F)}\Phi(X)\chi(\det(X))|\det(X)|^{s-\frac{m+1}{2}}\ud X,
\end{equation}
for any unitary character $\chi$ of $F^\times$, where the measure ${\rm d} X$ is normalized so that $\vol(S_{m}(\CO_{F}),\ud X)=1$.
Note that this normalization differs from the one used in \cite[p. 53]{Ik17}
by a constant $|2|^{\frac{m(m-1)}{4}}$.
It is known that the integral defining $\CZ_{S_m}(s,\Phi,\chi)$ converges absolutely for $\Re(s)>\frac{m-1}{2}$, and has a meromorphic continuation to $s\in\BC$.
For any $X\in S^{\circ}_{m}$, we may also define a zeta integral on the $F$-rational orbit $S^{X}_m(F)$ by
$$
\CZ_{S^{X}_{m}}(s,\Phi,\chi) =
\int_{S^{X}_{m}}\Phi(X)
\chi(\det(X))
|\det(X)|^{s-\frac{m+1}{2}}\ud X.
$$
For convenience, we fix an additive character $\psi$ of $F$ with conductor $\CO_{F}$. The Fourier transform is defined by
$$
\wh{\Phi}(X)=
\int_{S_{m}(F)}
\Phi(Y)\psi(\tr(XY))\ud Y,
$$
for any $\Phi\in \CC^\infty_c(S_{m})$.
The following proposition is a reformulation of the main relevant results from \cite[Appendix~1 to Chapter~3]{PSR87} and \cite{Ik17}.

\begin{pro}\label{pro:zeta-pvs}
Assume that the character $\chi$ is unitary.
The local zeta integrals $\CZ_{S_m}(s,\Phi,\chi)$ and $\CZ_{S^{X}_{m}}(s,\Phi,\chi)$ converge absolutely for $\Re(s)>\frac{m-1}{2}$ and admit meromorphic continuations to $s\in \BC$, and enjoy the following
functional equation:
$$
\CZ_{S_m}(s,\Phi,\chi) =
\sum_{i=1}^{t}
c_{S^{X_{i}}_{m}}(\chi,s)
\CZ_{S^{X_{i}}_{n}}
(\frac{m+1}{2}-s,\wh{\Phi},\chi^{-1})
$$
where
\begin{enumerate}
\item if $m=2n+1$, the coefficients have the following expression:
\begin{align*}
c_{S^{X_{i}}_{m}}(\chi,s)
&=
\gam(s-n,\chi,\psi)^{-1}
\prod_{r=1}^{n}
\gam(2s-2n-1+2r,\chi^{2},\psi)^{-1}
\\
&\qquad\times
|2|^{-2ns}
\chi^{-n}(4)\eta_{X_{i}};
\end{align*}
\item if $m=2n$, the coefficients have the following expression:
\begin{align*}
c_{S^{X_{i}}_{m}}(\chi,s)
&=
\gam(s-n+\frac{1}{2},\chi,\psi)^{-1}
\prod_{r=1}^{n}
\gam(2s-2n+2r,\chi^{2},\psi)^{-1}
\\
&\qquad\times
|2|^{-2ns}
\chi^{-n}(4)
\frac{\alp(D_{X_{i}})}{\alp(1)}
\gam(s+\frac{1}{2},\chi\cdot \chi_{D_{X_{i}}},\psi)^{-1}.
\end{align*}
\end{enumerate}
Here $\gam(s,\chi,\psi)$ is the abelian $\gam$-function associated to character $\chi$, which can be defined via Tate integrals. 
\end{pro}
The other unexplained notations associated to  $X_{i}$ for the $m=2n$ even case can be found in 
\cite[Theorem 2.1]{Ik17}. 
Since we only need the odd case in our calculation, we will omit the notations for the even case here. 

The following lemma can be proved via straightforward computation.

\begin{lem}\label{lem:homogenityofzeta}
The distribution $\CZ_{S_m}(s,\cdot,\chi)$ satisfies the homogeneity, with respect to the action of $g\in\GL_m(F)$,
$$
\CZ_{S_m}(s,g^{-1}\cdot\Phi,\chi)  = \chi^{-2}(\det(g))|\det(g)|^{-2s}\CZ_{S_m}(s,\Phi,\chi).
$$
The same property holds for $\CZ_{S^{X}_{m}}(s,\cdot,\chi)$.
\end{lem}

Define a product of abelian $L$-functions by
\begin{align}\label{am}
a_{m}(s,\chi) = L(s-\frac{m-1}{2},\chi)
\prod_{r=1}^{\lfloor{\frac{m}{2}\rfloor}}
L(2s-m+2r,\chi^{2}).
\end{align}

We have the following theorem generalizing \cite[Appendix~1 to Chapter~3, Theorem]{PSR87}. The proof is almost parallel and for the convenience of reader we include the proof here.
\begin{thm}\label{thm:gcd-zi}
Assume that the character $\chi$ is unitary.
For any $\Phi\in \CC^{\infty}_{c}(S_{m})$, the zeta integral $\CZ_{S_m}(s,\Phi,\chi)$ can be expressed as a product of $a_{m}(s,\chi)$ and a polynomial in $\BC[q^{s},q^{-s}]$. In other words, the fractional ideal
$\CI_\CZ$ generated by $\CZ_{S_m}(s,\Phi,\chi)$ is given by
$$
\CI_{\CZ} = \{ \CZ_{S_m}(s,\Phi,\chi)\ |\ \Phi\in \CC^{\infty}_{c}(S_{m}) \} =
a_{m}(s,\chi)\cdot \BC[q^{s},q^{-s}].
$$
\end{thm}

\begin{proof}
For any $\Phi\in \CC^{\infty}_{c}(S_{m})$, the zeta integral $\CZ_{S_m}(s,\Phi,\chi)$ can be expressed as a rational function in $q^{-s}$.
From Lemma \ref{lem:homogenityofzeta}, the set $\CI_{\CZ}$ is a $\BC[q^{s},q^{-s}]$-module, and hence a fractional ideal in $\BC(q^s,q^{-s})$. It is enough to show that the quotient
$$
\frac{\CZ_{S_m}(s,\cdot,\chi)}{a_{m}(s,\chi)}
$$
is entire, and for any $s=s_{0}\in \BC$, there exists $\Phi_{0}\in \CC^{\infty}_{c}(S_{m})$ such that
$$
\frac{\CZ_{S_m}(s,\Phi,\chi)}{a_{m}(s,\chi)}
$$
is nonzero at $s=s_{0}$.

The proof of this statement goes via induction on $m$. When $m=1$ it follows from Tate's thesis (\cite{Tt50} and also \cite{CF67}). For a general $m>1$, we consider the variety
$$
S_{m}^{m-1}(F):
=\{
\left(
\begin{smallmatrix}
X & Y\\
{}^t\!Y & x
\end{smallmatrix}
\right)
\in S_{m}(F)
\ |\ X\in S_{m-1}(F), Y\in F^{m-1}, x\in F^{\times}
 \}
$$
which is open in $S_{m}(F)$. It is clear that $S_{m}^{m-1}(F)$ is isomorphic to $F^{m-1}\times S_{n-1}(F)\times F^{\times}$ via the following map
$$
\left(\begin{smallmatrix}
\RI & U\\
0 & 1
\end{smallmatrix}\right)
\times
\left(\begin{smallmatrix}
A 	& 0\\
0	& b
\end{smallmatrix}\right)
\to
\left(\begin{smallmatrix}
\RI & U\\
0 & 1
\end{smallmatrix}\right)
\left(\begin{smallmatrix}
A 	& 0\\
0	& b
\end{smallmatrix}\right)
\left(\begin{smallmatrix}
\RI & 0\\
{}^t\!U & 1
\end{smallmatrix}\right)
=
\left(\begin{smallmatrix}
A+Ub\cdot {}^t\!U 	& bU\\
{}^t\!Ub 	&	b
\end{smallmatrix}\right)
$$
where $U\in F^{m-1}, A\in S_{m-1}(F)$, and $b\in F^{\times}$.

Consider the zeta integral $\CZ_{S_m}(s,\Phi,\chi)$ for $\Phi\in\CC^{\infty}_{c}(S^{m-1}_{m})$. For any $\Phi\in \CC^{\infty}_c(S^{m-1}_{m})$, the space of smooth, compactly supported functions on $S^{m-1}_{m}(F)$,
there exists $\Phi_{1}\in \CC^{\infty}_{c}(F^{m-1}),
\Phi_{2}\in \CC^{\infty}_{c}(S_{m-1})$ and $\Phi_{3}\in \CC^{\infty}_{c}(F^{\times})$ such that
$$
\Phi(\left(\begin{smallmatrix}
A+Ub\cdot {}^t\!U 	& bU\\
{}^t\!Ub 	&	b
\end{smallmatrix}\right))
=\Phi_1(U)\Phi_2(A)\Phi_3(b).
$$
Due to the choice of the relevant measures, up to a constant, we have
\begin{align*}
\CZ_{S_m}(s,\Phi,\chi)
&=
\int_{F^{m-1}}
\Phi_{1}(U)\ud U\\
&\qquad\times
\int_{S_{m-1}(F)}
\Phi_{2}(A)\chi(\det (A))
|\det (A)|^{s-\frac{m+1}{2}}\ud A
\\
&\qquad\times
\int_{F^\times}
\Phi_{3}(b)\chi(b)|b|^{s-\frac{m+1}{2}+1}\ud^\times b.
\end{align*}
Note that the integral associated to $\Phi_{3}$ is entire since $\Phi_{3}\in \CC^{\infty}_{c}(F^{\times})$, while the integral $\int_{F^{m-1}}\Phi_{1}(U)\ud U$ is a constant.
By the assumption of the induction, we have that
$$
\{ \frac{\CZ_{S_m}(s,\Phi,\chi)}{a_{m-1}(s-\frac{1}{2},\chi)}\ |\ \Phi\in \CC^{\infty}_{c}(S^{m-1}_{m}) \} = \BC[q^{s},q^{-s}].
$$

Next, we consider the poles of the local zeta integral $\CZ_{S_m}(s,\Phi,\chi)$ for general $\Phi\in \CC^{\infty}_{c}(S_{m})$. To this end, we consider the following Laurent expansion of
$\CZ_{S_m}(s,\Phi,\chi)$ at $s=s_0$:
$$
\CZ_{S_m}(s,\Phi,\chi) =
\sum_{k\geq a}\ell_{k}(\Phi,\chi)(s-s_{0})^{k}.
$$
It is clear that $\Phi\to \ell_{k}(\Phi,\chi)$ defines a distribution on $\CC^{\infty}_{c}(S_{m})$. We are interested in the case that $s=s_0$ is a pole of $\CZ_{S_m}(s,\Phi,\chi)$, i.e., $a=a(s_0)$ is negative.
In such a situation, it is easy to check that the leading term (or the Cauchy principal value) $\ell_{a}(\Phi,\chi)$ is a distribution with the quasi-invariant property with respect to the action of $\GL_m(F)$ as
given in Lemma \ref{lem:homogenityofzeta}. By the definition of $\CZ_{S_m}(s,\Phi,\chi)$ as in \eqref{lzi-sm}, for any $\Phi\in\CC_c^\infty(S_m^\circ)$, the local zeta integral $\CZ_{S_m}(s,\Phi,\chi)$ is entire
as function in $s$. Hence when $a=a(s_0)<0$, the support of the distribution $\ell_{a}(\Phi,\chi)$ is contained in
$$
S_m(F)\bs S_m^\circ(F)=\{X\in S_m(F)\ |\ \det(X)=0\}.
$$

Assume that $s=s_0$ is not a pole of $\CZ_{S_m}(s,\Phi,\chi)$ for $\Phi$ running through $\CC^{\infty}_{c}(S^{m-1}_{m})$. Then the support of the distribution $\ell_{a}(\Phi,\chi)$ with
$\Phi\in\CC_c^\infty(S_m)$ is contained in the following subset
\[
(S_{n}(F)\bs S^{m-1}_{m}(F))\cap (S_m(F)\bs S_m^\circ(F))=S_m(F)\bs (S^{m-1}_{m}(F)\cup S_m^\circ(F)),
\]
and stable under the action of $\GL_m(F)$.
Over the $p$-adic local field $F$, it is easy to check that the set
$$
S_m(F)\bs \GL_m(F)\cdot(S^{m-1}_{m}(F)\cup S_m^\circ(F))
$$
consists only of the unique closed $\GL_m(F)$-orbit $\{0\}$. Hence, we have that $\ell_{a}(\Phi,\chi) = \alpha_{s_0,\chi}\cdot\Phi(0)$ for some constant $\alpha_{s_0,\chi}$ that depends on $\chi$ and $s_0$,
as a distribution on $\CC^{\infty}_{c}(S_{m})$.
It follows from Lemma \ref{lem:homogenityofzeta} that the character $\chi^{-2}\cdot |\cdot|^{-2s_0}$ must be the trivial character.
Note that such a situation never occurs when $\chi^{2}$ is a ramified unitary character of $F^{\times}$, in which case $a_{n}(s,\chi)  =1$ identically. Therefore, in order to understand the situation that
$s=s_0$ is a pole of $\CZ_{S_m}(s,\Phi,\chi)$ for $\Phi\in \CC^{\infty}_{c}(S_{m})$, but is not a pole for $\Phi\in\CC^{\infty}_{c}(S^{m-1}_{m})$, we only need to consider the case when the character
$\chi^{2}$ is unramified.

In such a situation,
 $|\cdot|^{2s_0}$ must be a unitary character, which happens only when $\Re(s_0)=0$.
From the functional equation in Proposition \ref{pro:zeta-pvs}, if $\CZ_{S_m}(s,\Phi,\chi)$ has such a pole at $s=s_{0}$, then the right-hand side of the functional equation must
have the same pole at $s=s_0$. Since $\Re(\frac{m+1}{2}-s_{0}) = \frac{m+1}{2}>\frac{m-1}{2}$, the local zeta integral $\CZ_{S^{X_{i}}_{m}}(\frac{m+1}{2}-s,\wh{\Phi},\chi^{-1})$ is absolutely convergent at $s=s_0$, and
hence is holomorphic at $s=s_0$ for every $i$. Hence there exists at least one $S^{X_{i}}_{m}$ such that the coefficient $c_{S^{X_{i}}_{m}}(s,\chi)$ must have a pole at such $s_{0}$.
According to the explicit expression for the coefficients $c_{S^{X_{i}}_{m}}(s,\chi)$ in Proposition \ref{pro:zeta-pvs}, the coefficient $c_{S^{X_{i}}_{m}}(s,\chi)$ has a pole at such  $s_0$
only when $m$ is even, which is a simple pole.
In this situation, the distribution $\ell_{a}(\cdot,\chi)$ gives the extra simple pole for $\CZ_{S_m}(s,\Phi,\chi)$ with $\Phi\in \CC^{\infty}_{c}(S_{m})$,
which can exactly be captured by the standard $L$-factor $L(2s,\chi^{2})$.

By summarizing the above discussion, we deduce that the poles of the local zeta integral $\CZ_{S_m}(s,\Phi,\chi)$ with $\Phi\in \CC^{\infty}_{c}(S_{m})$ are exactly the same as those of
$$
a_{m}(s,\chi) =
\bigg\{
\begin{matrix}
a_{m-1}(s-\frac{1}{2},\chi) & \text{ if $m$ is odd},\\
L(2s,\chi^{2})a_{m-1}(s-\frac{1}{2},\chi) & \text{ if $m$ is even}.
\end{matrix}
$$
We are done.
\end{proof}


\subsection{Zeta integrals and intertwining operators}\label{ssec-zi-io}
As discussed in \cite[Appendix~1 to Chapter~3]{PSR87}, we are going to take a close look at the relation between the poles of the local zeta integrals $\CZ_{S_m}(s,\Phi,\chi)$ and those of the local intertwining operator associated to the
standard Siegel parabolic subgroup $P=MN$ of $\Sp_{2m}$.
Following the notation in Section \ref{sec-LT-DZI}, we define
$\RI(s,\chi) = \Ind^{\Sp_{2m}(F)}_{P(F)}(\chi_{s})$ to be the normalized induced representation $\Sp_{2m}(F)$ from $P$, which consists of holomorphic sections. As in \eqref{chis}, the character $\chi_s$ is defined to be
\[
\chi_s(m(A)):=\chi(\det (A))|\det (A)|^s
\]
where $m(A)=\left(\begin{smallmatrix}A&0\\0 &A^*\end{smallmatrix}\right)\in M(F)\subset\Sp_{2m}(F)$ with $A\in\GL_m(F)$.
Take the Weyl element $w_{m}=
\left(
\begin{smallmatrix}
0 & -\RI_{m} \\
\RI_{m} & 0
\end{smallmatrix}
\right)
$ of $\Sp_{2m}$, which takes $P$ to its opposite $P^-$.  The intertwining operator $\RM_{w_{m}}(s,\chi)$ from $\RI(s,\chi)$ to $\RI(-s,\chi^{-1})$, as in \eqref{lio-0}, is defined by
$$
\RM_{w_{m}}(s,\chi)(f_{\chi_s})(g) =
\int_{N(F)}
f_{\chi_s}(w_{m}ng)\ud n,
$$
with $f_{\chi_s}\in \RI(s,\chi)$, which converges absolutely for $\Re(s)$ sufficiently positive and admits meromorphic continuation to $s\in \BC$ (\cite{Wal03}, for instance).

\begin{pro}\label{pro:pole-Mw}
For any holomorphic section $f_{\chi_s}\in \RI(s,\chi)$, the quotient
$$
\frac{\RM_{w_m}(s,\chi)(f_{\chi_s})}{a_{m}(s,\chi)}
$$
is a holomorphic section in $\RI(-s,\chi^{-1})$. Moreover, for any $s=s_{0}$, there exists a holomorphic section $f_{\chi_{s}}\in \RI(s,\chi)$ such that
the quotient $\frac{\RM_{w_m}(s,\chi)(f_{\chi_s})}{a_{m}(s,\chi)}$
is non-zero at $s=s_0$.
\end{pro}

\begin{proof}
By the Rallis Lemma \cite[Lemma 4.1]{PSR87}, the analytical properties of
$\RM_{w_m}(s,\chi)(f_{\chi_s})$ with $f_{\chi_s}$ running through the holomorphic sections in $\RI(s,\chi)$
coincide with those of
$\RM_{w_m}(s,\chi)(f_{\chi_s})(w_m)$ with $f_{\chi_s}$ running through the holomorphic sections in $\RI(s,\chi)$ that have their support $\supp(f_{\chi_s})$ contained in
$Pw_{m}N$.

For any $\Phi\in \CC^{\infty}_{c}(\CS_{m})$, one may construct a holomorphic section $f\in \RI(s,\chi)$ by defining $f(w_{m}n(X)) = \Phi(X)$. It is clear that such a constructed holomorphic section $f$
has its support contained in $Pw_{m}N$. For such a holomorphic section $f$, the intertwining operator can be calculated as follows:
\begin{align*}
\RM_{w_m}(s,\chi)(f)(w_m)
&=
\int_{S_{m}(F)}
f(w_mn(X)w_m)\ud X\\
&=
\int_{S_{m}(F)}
\chi(\det (X^{-1}))
|\det (X^{-1})|^{s+\frac{m+1}{2}}
\Phi(-X^{-1})\ud X.
\end{align*}
Here $\del_{P}(m(A)) = |\det (A)|^{m+1}$ for any $A\in \GL_{m}(F)$.  For any $X\in S_{m}^{0}(F)$, one has that
$$
w_{m}n(X)w_{m} = n(-X^{-1})
m({}^t\!X^{-1})w_{m}n(-{}^t\!X^{-1}).
$$
By \cite[p.57]{M97}, one has that ${\rm d}(X^{-1}) = |\det (X)|^{-m-1}\ud X$. By changing variables, we obtain that
\begin{align*}
\RM_{w_m}(s,\chi)(f)(w_{m})
=
&\chi((-1)^{m})
\int_{S_{m}(F)}
\chi(\det (X))|\det (X)|^{s-\frac{m+1}{2}}\Phi(X)\ud X
\\
=
&
\chi((-1)^{m})\CZ_{S_m}(s,\Phi,\chi).
\end{align*}
By Theorem \ref{thm:gcd-zi}, we have that
\[
\{ \RM_{w_m}(s,\chi)(f_{\chi_s})(w_{m})\ |\ f_{\chi_s}\in \RI(s,\chi) \text{ with } \supp(f_{\chi_s})\subset Pw_mN \}
\]
contains the fractional ideal $ a_{m}(s,\chi)\cdot \BC[q^{s},q^{-s}]$.
On the other hand, according to \cite[Lemma 3.1]{Y14}, the quotient $\frac{\RM_{w_m}(s,\chi)}{a_{m}(s,\chi)}$ is entire as function in $s$.  Therefore, we obtain that the set
\[
\{ \RM_{w_m}(s,\chi)(f_{\chi_s})(w_m)\ |\ f\in \RI(s,\chi) \text{ with } \supp(f_{\chi_s})\subset Pw_{m}N \} \cdot \BC[q^{s},q^{-s}]
\]
is equal to the fractional ideal $ a_{m}(s,\chi)\cdot \BC[q^{s},q^{-s}]$.
We are done.
\end{proof}

Let $\RM^{\dag}_{w_m}(s,\chi,\psi)$ be the normalized intertwining operator as defined in \eqref{eq:IT-eta-rho} (see also \cite[p.668]{Y14}). Note that up to a nonzero holomorphic factor given by constant multiple of $\veps$-factors,
\begin{align}\label{nM-nM}
\frac{1}{b_{m}(-s,\chi^{-1})}\RM^{\dag}_{w_m}(s,\chi,\psi) \text{ ``=" }
\frac{1}{a_{m}(s,\chi)}\RM_{w_m}(s,\chi)
\end{align}
where
$$
b_{m}(s,\chi) = L(s+\frac{m+1}{2},\chi)
\prod_{r=1}^{\lfloor{\frac{m}{2}\rfloor}}L(2s+2r-1,\chi^{2}).
$$
In the local theory of the zeta integrals from the Piatetski-Shapiro and Rallis doubling method, the notion of {\sl good sections} plays an important role in the determination of the analytic properties of the
local zeta integrals. We recall from \cite{Ik92} and \cite[Definition 3.1]{Y14} the space of good sections of $\Sp_{2m}(F)$ associated to the standard Siegel parabolic subgroup $P$, which is denoted by
$\RI^{\dag}(s,\chi)$.

\begin{dfn}[Good Sections]\label{defn:good-f}
Let $\chi$ be any unitary character of $F^\times$. For $s_0\in\BC$,  a meromorphic section $f_{\chi_s}$ of $\RI(s,\chi)$ is called {\bf good} at $s=s_0$ if when $\Re(s_0)>-\frac{1}{2}$,
$f_{\chi_s}$ is holomorphic at $s=s_0$;  and when $\Re(s_0)<0$,
$\RM^{\dag}_{w_m}(s,\chi,\psi)(f_{\chi_s})$ is holomorphic at $s=s_0$. In general, a section $f_{\chi_s}\in\RI(s,\chi)$ is called a {\bf good section} if $f_{\chi_s}$ is good at every $s_0\in\BC$.
\end{dfn}

\begin{pro}\label{pro:analyticalgoodsection}
For any good section $f_{\chi_s}\in \RI^{\dag}(s,\chi)$, the quotient $\frac{f_{\chi_s}(g)}{b_{m}(s,\chi)}$ is a holomorphic section in $\RI(s,\chi)$. Moreover, for any $s=s_{0}$,
there exists a good section $f_{\chi_s}\in \RI^{\dag}(s,\chi)$ such that the quotient $\frac{f_{\chi_s}(g)}{b_{m}(s,\chi)}$ is non-zero at $s=s_0$.
\end{pro}

\begin{proof}
According to \cite[Proposition~3.1]{Y14}, for any good section $f_{\chi_s}\in \RI^{\dag}(s,\chi)$, there exist holomorphic sections $f_{1}\in \RI(s,\chi)$ and $f_{2}\in \RI(-s,\chi^{-1})$, respectively, such that
$$
f_{\chi_s} = f_{1}+\RM^{\dag}_{w_m}(-s,\chi^{-1},\psi^{-1})(f_{2}).
$$
Note that up to a nonzero entire factor,
$$
\RM^{\dag}_{w_m}(-s,\chi^{-1},\psi^{-1})(f_{2}) ``="
\frac{b_{m}(s,\chi)}{a_{m}(-s,\chi^{-1})}\RM_{w_m}(-s,\chi^{-1})(f_{2}).
$$
From Proposition \ref{pro:pole-Mw}, the quotient $\frac{\RM_{w_m}(-s,\chi^{-1})(f_{2})}{a_{m}(-s,\chi^{-1})}$ is entire as function in $s$. Clearly, $\frac{f_1}{b_{m}(s,\chi)}$ is holomorphic since $\frac{1}{b_m(s,\chi)}$ is holomorphic. Hence for any good section $f_{\chi_s}\in \RI^{\dag}(s,\chi)$, the
quotient $\frac{f_{\chi_s}(g)}{b_{m}(s,\chi)}$ is entire as function in $s$. The non-vanishing of the quotient at $s=s_0$ also follows directly from Proposition \ref{pro:pole-Mw}.
\end{proof}

We will come back to have more conceptual discussions on the space of good sections (Proposition \ref{pro:FTX}) in Section \ref{ssec-FoSs}.


\subsection{Fiber integrations and functional equations}\label{ssec-fife}
We are going to specialize the general discussion in Section \ref{ssec-zi-sm} to the case of $m=2n+1$ and deduce a functional equation for $\beta_{\psi}(\chi_s)$ by considering the fiber integration
through the fibration morphism:
\begin{equation}\label{det-map}
\det\ :\ S_{2n+1}(F)\rightarrow F
\end{equation}
following \cite{W65} and \cite[Section 8.3]{Ig00}. From \cite[Lemma 2.1]{Ik17}, the Clifford invariant $\varrho(X)$ of $X\in S_{2n+1}(F)$ is given by
\[
\varrho(X)=\apair{-1,-1}^{\frac{n(n+1)}{2}}\apair{(-1)^n,\det (X)}\eps_X,	
\]
where $\apair{\cdot,\cdot}$ is the Hilbert symbol and $\eps_X$ is the Hasse invariant of $X$.
In this case, we renormalize the local zeta integral by defining
\begin{equation*} 
\CZ_\Phi(s,\chi)=\int_{S_{2n+1}(F)}\Phi(X)\chi(\det (X))|\det (X)|^{s}\ud X,
\end{equation*}
The results developed in Section \ref{ssec-zi-sm} will have a shift from $s$ there to $s+(n+1)$ here, since $\frac{m+1}{2}=n+1$ with $m=2n+1$.
Note here that we use the Haar measure ${\rm d} X$ on $S_{2n+1}(F)$ that is normalized with $\vol(S_{2n+1}(\CO),\ud X)=1$.
It is important to note that the measure ${\rm d} X$ chosen in \cite{Ik17} is self-dual with respect to the Fourier transform, and is defined to be
${\rm d} X=|2|^{\frac{n(2n+1)}{2}} \prod_{i\leq j}\ud X_{ij}$, assuming that the additive character $\psi$ has conductor $\CO$. This implies that the formulas from
\cite{Ik17} appear here with the constants different by $|2|^{\frac{n(2n+1)}{2}}$!

Because the Clifford invariant $\varrho(X)$ of $X\in S_{2n+1}(F)$ is constant on the $\GL_{2n+1}(F)$-orbit of $X$, the functional equation in Proposition \ref{pro:zeta-pvs} can be reformulated as

\begin{pro}[{\cite[Theorem 2.1]{Ik17}}]\label{lm:Ikeda}
For $\Phi\in \CC^\infty_c(S_{2n+1})$, and for
any unitary character $\chi$ of $F^\times$, the local zeta integrals $\CZ_{\Phi}(s, \chi)$ and $\CZ_{\varrho\cdot\wh{\Phi}}(s,\chi^{-1})$ satisfy the following functional equation:
\begin{equation}\label{fe-lzism}
 |2|^{-2ns+2n^2}\chi^{-2n}(2)\CZ_{\varrho\cdot\wh{\Phi}}(-s-\frac{1}{2},\chi^{-1})
= \beta_\psi(\chi_{s})\CZ_{\Phi}(s+\frac{1}{2}-(n+1), \chi),
\end{equation}
where $\beta_\psi(\chi_s)$ is defined in \eqref{eq:eta-intertwining} and $\wh{\Phi}$ is the Fourier transform of $\Phi$.
\end{pro}
Note that the local zeta integral $\CZ_{\Phi}(s+\frac{1}{2}-(n+1), \chi)$ converges absolutely for $\Re(s)>n-\frac{1}{2}$ and
$\CZ_{\varrho\cdot\wh{\Phi}}(-s-\frac{1}{2},\chi^{-1})$ converges absolutely for $\Re(s)<\frac{1}{2}$, and the functional equation \eqref{fe-lzism} holds by meromorphic continuation of the local zeta functions.

In order to understand the factor $\beta_\psi(\chi_s)$ in terms of the harmonic analysis on $\GL_1(F)$,
we consider,
for $\Phi\in \CC^\infty_c(S_{2n+1})$ and $t\in F^\times$, by \cite[Proposition 1]{W65} and \cite[Section 6]{W65}, and \cite[Lemma 8.3.2]{Ig00}, there exists a unique measure $\mu_t$ on the fiber of $t$
\[
{\det}^{-1}(t):=\{X\in S_{2n+1}(F)\mid \det(X)=t\},
\]
which is a $p$-adic manifold, such that the following fiber integration:
\begin{equation}\label{radon-1}
f_\Phi(t)=\int_{{\det}^{-1}(t)}\Phi(X)\mu_t(X)
\end{equation}
has the property that
\begin{equation}\label{mut}
 \int_{S_{2n+1}(F)}\Phi(X)\ud X=\int_{F}\int_{{\det}^{-1}(t)}\Phi(X) \mu_t(X)\ud t=\int_Ff_\Phi(t)\ud t,
\end{equation}
where ${\rm d} t$ is the Haar measure on $F$ with $\vol(\CO,\ud t)=1$.
Similarly, one defines the following fiber integration:
\begin{equation}\label{radon-2}
f_{\varrho\cdot\wh{\Phi}}(t)=\int_{{\det}^{-1}(t)}(\varrho\cdot\widehat{\Phi})(X)\mu_t(X)=\int_{{\det}^{-1}(t)}\varrho(X)\widehat{\Phi}(X)\mu_t(X).
\end{equation}
Since $|\widehat{\Phi}(X)\varrho(X)|=|\widehat{\Phi}(X)|$, the integral defining $f_{\varrho\cdot\wh{\Phi}}$ is absolutely convergent.
Following \cite[Theorem 8.4.1]{Ig00}, for any unitary character $\chi$ of $F^\times$ and for any $\Phi\in \CC^\infty_c(S_{2n+1})$,
we have that
\begin{equation}\label{eq:Igusa-fiber-1}
\CZ_{\Phi}(s,\chi)=
(1-q^{-1})\int_{F^\times}f_{\Phi}(t)\chi(t)|t|^{s+1}\ud^*t,
\end{equation}
where ${\rm d}^* t:=(1-q^{-1})^{-1}|t|^{-1}\ud t$ with  $\vol(\CO^\times, \ud^* t)=1$.
Define the Mellin transform of $f_\Phi$ by
\begin{equation}\label{mellin}
\CM(f_\Phi)(s,\chi):=\int_{F^\times}f_\Phi(t)\chi(t)|t|^{s}\ud^*t.
\end{equation}
It converges absolutely for $\Re(s)>1$ and has meromorphic continuation to $s\in\BC$.
By Proposition \ref{lm:Ikeda}, we obtain the following corollary.

\begin{cor}\label{fe-Radon}
For $\Phi\in \CC^\infty_c(S_{2n+1})$,  the following functional equation:
\begin{equation}\label{eq:FE-F-0}
|2|^{-2ns+2n^2}\chi^{-2n}(2) \CM(f_{\varrho\cdot\wh{\Phi}})(-s+\frac{1}{2},\chi^{-1})
=\beta_{\psi}(\chi_{s})\CM(f_{\Phi})(s-n+\frac{1}{2},\chi)	
\end{equation}
holds for all unitary characters $\chi$, with $\beta_\psi(\chi_s)$ as in \eqref{eq:eta-intertwining}.
\end{cor}
It is clear that $\CM(f_{\Phi})(s-n+\frac{1}{2},\chi)$ converges absolutely for $\Re(s)>n+\frac{1}{2}$ and $\CM(f_{\varrho\cdot\wh{\Phi}})(-s+\frac{1}{2},\chi^{-1})$ converges absolutely
for $\Re(s)<-\frac{1}{2}$, and the functional equation holds by meromorphic continuations of the both sides.


\section{Harmonic Analysis for $\beta_\psi(\chi_s)$}\label{sec-FTbeta}


In Corollary \ref{fe-Radon}, we obtain a functional equation with $\beta_\psi(\chi_s)$ occurring from the functional equation for local zeta integrals associated to a prehomogeneous vector space
$(\GL_{2n+1},S_{2n+1})$. In this section, we intend to understand the functional equation in Corollary \ref{fe-Radon} in terms of harmonic analysis on $\GL_1(F)$, following the idea of the Tate thesis (\cite{Tt50}) or
more generally the suggestion from the Braverman-Kazhdan proposal (\cite{BK00}). When $n=0$, it is the $\GL_1$-theory from the Tate thesis (\cite{Tt50}). We will assume that $n>0$ in this section, although most of the
discussions automatically cover the case of $n=0$.


\subsection{Behavior of both $f_{\Phi}$ and $f_{\varrho\cdot\wh{\Phi}}$}\label{ssec-bbf}
In order to understand the space of all functions $f_{\Phi}$, the fiber integration as defined in \eqref{radon-1}, and the space of all functions $f_{\varrho\cdot\wh{\Phi}}$, as in \eqref{radon-2},
we introduce the following spaces $\CS^{+}_{n}(F^\times)$ and $\CS^{-}_{n}(F^\times)$, as in \cite[Chapter 1, \S 5.2]{Ig78}.

\begin{dfn}[Space $\CS^{+}_{n}(F^\times)$]\label{CS+}
Define $\CS^{+}_{n}(F^\times)$ to be the subspace of $\CC^\infty(F^\times)$ consisting of functions $f_{+}$ with property:
\begin{itemize}
	\item[(1)] ${\rm supp}(f_+)$ is bounded on $F^\times$, i.e., $f_+(x)=0$ for $|x|\gg 1$;
	\item[(2)] for $|x|\ll 1$, $f_+$ has the following asymptotic behavior:
	\begin{equation} \label{eq:asymptotic-1}
	f_+(x)=a^{+}_{0}({\rm ac}(x)) +\sum_{i=0}^{n-1}\bigg(a^{+}_{i,+}({\rm ac}(x))|x|^{i+\frac{1}{2}}+a^{+}_{i,-}({\rm ac}(x))|x|^{i+\frac{1}{2}}(-1)^{{\rm ord}(x)}\bigg)
	\end{equation}
	for some locally constant functions $a^+_{0}({\rm ac}(x))$ and $a^+_{i,\pm}({\rm ac}(x))$ on $\CO^\times$.
\end{itemize}
Here $(-1)^{{\rm ord}(x)}=|x|^{\pi i/\log(q)}$, and
${\rm ac}(x):=x\varpi^{-{\rm ord}(x)}$ is the angular component of $x\in F^\times$, depending on the choice of the uniformizer $\varpi$.
\end{dfn}

\begin{dfn}[Space $\CS^{-}_{n}(F^\times)$]\label{CS-}
Define $\CS^{-}_{n}(F^\times)$ to be the subspace of $\CC^\infty(F^\times)$ consisting of functions $f_{-}$ with property:
\begin{itemize}
	\item[(1)] ${\rm supp}(f_-)$ is bounded on $F^\times$, i.e., $f_-(x)=0$ for $|x|\gg 1$;
	\item[(2)] for $|x|\ll 1$, $f_-$ has the following asymptotic behavior:
	\begin{equation}  \label{eq:asymptotic-2}
	f_-(x)=a^{-}_{0}({\rm ac}(x))|x|^{n} +\sum_{i=0}^{n-1}\bigg(a^{-}_{i,+}({\rm ac}(x))|x|^{i}+a^{-}_{i,-}({\rm ac}(x))|x|^{i}(-1)^{{\rm ord}(x)}\bigg),
	\end{equation}
	for some locally constant functions $a^-_{0}({\rm ac}(x))$ and $a^-_{i,\pm}({\rm ac}(x))$ on $\CO^\times$.
\end{itemize}
Here ${\rm ac}(x)$ and $(-1)^{{\rm ord}(x)}$ are the same as in Definition \ref{CS+}.
\end{dfn}

It is important to note that the functions in $\CS_n^{\pm}(F^\times)$ are $\CO^\times$-finite and hence are uniformly locally constant. More precisely
we have the following lemma.

\begin{lem}\label{K-finite}
For any $f$ in $\CS^{+}_{n}(F^{\times})$ or $\CS^{-}_{n}(F^{\times})$, there exists $N\in \BN$ such that $f$ is $(1+\vpi^{N}\CO)$-invariant.
\end{lem}
\begin{proof}
For an $f$ in $\CS^{+}_{n}(F^{\times})$, there exists $u_{1}>0$ such that $f(x) = 0$ for any $x\in F^{\times}$ satisfying $|x|>u_{1}$.
There also exists $u_{2}>0$ such that $f(x)$ has an asymptotic expression as in \eqref{eq:asymptotic-1} for any $|x|<u_{2}$.
Let $\mathrm{ch}_{u_{1},u_{2}}$ be the characteristic function of the open compact set $$\{ x\in F^{\times}\ |\ u_{2}\leq |x|\leq u_{1}\}.$$
It is clear that $f\cdot \mathrm{ch}_{u_{1},u_{2}}\in \CC^{\infty}_{c}(F^{\times})$. It follows that in the range $|x|\geq u_{2}$, there exists an integer $N_{1}$ such that $f$ is $(1+\vpi^{N_{1}}\CO)$-invariant.
On the other hand, for $|x|\leq u_{2}$, from the asymptotic expression in \eqref{eq:asymptotic-1}, there exists an integer $N_{2}$ such that $f$ is $(1+\vpi^{N_{2}}\CO)$-invariant. Taking $N= \max\{ N_{1},N_{2}\}$,
we know that $f$ is $(1+\vpi^{N}\CO)$-invariant.

For any $f$ in $\CS^{-}_{n}(F^{\times})$, the proof is the same. We omit the details.
\end{proof}

In order to understand the spaces $\CS^{+}_{n}(F^\times)$ and $\CS^{-}_{n}(F^\times)$ via the Mellin transform and its inversion on $F^\times$, we introduce the following two spaces
$\CZ^{+}_n(\widehat{F^\times})$ and $\CZ^{-}_n(\widehat{F^\times})$, as in
Theorem 5.3 of \cite[Chapter 1]{Ig78}.

\begin{dfn}\label{Z+-}
Let $\wh{\CO^\times}$ be the Pontryagin dual of $\CO^\times$ and write $\widehat{F^\times}=\BC^\times\times \wh{\CO^\times}$ via $\chi=(\chi(\varpi),\chi\vert_{\CO^\times})$.
Define $\CZ^{\pm}_n(\widehat{F^\times})$ to be the space of all complex-valued functions $Z_\pm(z,\chi)$ on $\widehat{F^\times}$, respectively, with properties:
\begin{itemize}
	\item[(1)] for almost all $\chi\in \wh{\CO^\times}$, $Z_{\pm}(z,\chi)$ are identically zero; and
	\item[(2)] for every $\chi\in \wh{\CO^\times}$, there exist constants $b^{\pm}_{0,\chi}$ and $b^{\pm}_{i,\pm,\chi}$  such that
	\begin{equation}\label{eq:Z-asymptotic-1}
	Z_+(z,\chi)- \frac{b_{0,\chi}^+}{1- z}-\sum^{n-1}_{i=0}\frac{b_{i,+,\chi}^+}{1-q^{-(i+\frac{1}{2})} z}+\frac{b_{i,-,\chi}^+}{1+q^{-(i+\frac{1}{2})} z}	
	\end{equation}
	and
	\begin{equation}\label{eq:Z-asymptotic-2}
	Z_{-}(z,\chi)- \frac{b_{0,\chi}^-}{1-q^{-n} z}-\sum^{n-1}_{i=0}\frac{b_{i,+,\chi}^-}{1-q^{-i} z}+\frac{b_{i,-,\chi}^-}{1+q^{-i} z}
	\end{equation}
	are polynomials in $z$ and $z^{-1}$ with complex coefficients.
\end{itemize}
\end{dfn}

The Mellin transform as defined in \eqref{mellin} enjoys the following properties, according to Theorem 5.3 of \cite[Chapter 1]{Ig78}.

\begin{pro}[Theorem 5.3, Chapter 1 of \cite{Ig78}]\label{Igusa78-53}
For $\chi\in \widehat{F^\times}$ with $\chi(\varpi)=1$, the Mellin transform $\CM(s,\chi)=\CM(z,\chi)$, with $z=q^{-s}$, is defined in \eqref{mellin} for $\Re(s)>1$,
admits meromorphic continuation to $s\in\BC$, and
provides  isomorphisms from $\CS^{\pm}_{n}(F^\times)$ onto $\CZ^{\pm}_n(\widehat{F^\times})$, respectively.
\end{pro}
We denote by $\CM^{-1}$ the inverse of Mellin transform, which takes any function $Z\in \CZ^{\pm}_n(\widehat{F^\times})$ to $\CM^{-1}(Z)\in \CS^{\pm}_{n}(F^\times)$.
For a $p$-adic local field $F$, it is more precisely given by
\begin{equation*} 
\CM^{-1}(Z)(x)=	\sum_{\chi\in \widehat{\CO^\times}} {\rm Res}_{z=0}(Z(z,\chi)z^{-{\rm ord}(x)-1})\chi({\rm ac}(x))^{-1},
\end{equation*}
according to \cite[Chapter 1, Theorem 5.3]{Ig78}.

\begin{pro}\label{pro-2fi}
For any $\Phi\in \CC^\infty_c(S_{2n+1})$,
the fiber integration $f_\Phi$ belongs to $\CS^+_{n}(F^\times)$ and similarly, $f_{\varrho\cdot\wh{\Phi}}$ belongs to $\CS^-_{n}(F^\times)$.
\end{pro}

\begin{proof}
First, we show that $f_\Phi$ belongs to $\CS^+_{n}(F^\times)$.
By \cite[Theorem  8.4.1]{Ig00}, we have that
\[
f_{\Phi}=\CM^{-1}(\CM(f_\Phi)(z,\chi))	
\]
for any $\Phi\in \CC^\infty_c(S_{2n+1})$.
Hence it is sufficient to show that the Mellin transform $\CM(f_\Phi)(z,\chi)$ is in $\CZ^{+}_n(\widehat{F^\times})$ as $\CM$ is an isomorphism from $\CS^{+}_n(F^\times)$ to $\CZ^{+}_n(\widehat{F^\times})$
(Proposition \ref{Igusa78-53}).

By \cite[Theorems 8.2.1 and 8.4.1]{Ig00}, $\CM(f_\Phi)$ admits a meromorphic continuation to a rational function in $\BC(z)$ and there is a positive integer $e(\Phi)$ such that
$\CM(f_\Phi)(z,\chi)$ is identically zero unless the conductor $e(\chi)\leq e(\Phi)$.
It suffices to prove that $\CM(f_\Phi)(z,\chi)$ satisfies \eqref{eq:Z-asymptotic-1}. To this end, we need the results on the local intertwining operator $\RM_{w_{2n+1}}(s,\chi)$ of $\Sp_{4n+2}$
from \cite{Ik92}.

Let $P=MN$ be the standard Siegel parabolic subgroup of $\Sp_{4n+2}$, with $M\cong\GL_{2n+1}$ and $N(F)\cong S_{2n+1}(F)$.
Consider the normalized induced representation
$$
\RI^{4n+2}(s,\chi)=\Ind^{\Sp_{4n+2}(F)}_{P(F)}(\chi(\det (A))|\det (A)|^s),
$$
with $m(A)=\left(\begin{smallmatrix}A&0\\0&A^*\end{smallmatrix}\right)\in M(F)$. Take the Weyl element
$w_{2n+1}=\left(\begin{smallmatrix}0&-\RI\\ \RI&0\end{smallmatrix}\right)$ with $\RI=\RI_{2n+1}$.
The associated
local intertwining operator $\RM_{w_{2n+1}}(s,\chi)$ takes $\RI^{4n+2}(s,\chi)$ to $\RI^{4n+2}(-s,\chi^{-1})$. Note that the modular character $\delta_P$ is given by $\delta_P(m(A))=|\det(A)|^{2n+2}$.

According to  \cite{PSR87}, \cite[p.~206]{Ik92} and Proposition \ref{pro:pole-Mw}, the following operator
\[
\frac{\RM_{w_{2n+1}}(s,\chi)}{L(s-n,\chi)\prod^{n}_{i=1}L(2s-2n-1+2i,\chi^2)}
\]
is holomorphic in $s$. It follows that
\begin{equation}\label{eq:CM-z-s}
\frac{\CM(f_\Phi)(z,\chi)}{L(s,\chi)\prod_{i=0}^{n-1}L(2s+2i+1,\chi^2)}	
\end{equation}
has no pole for all $\chi\in\widehat{\CO^{\times}}$, and hence is a polynomial of $z=q^{-s}$.
Note that we mix the use of $z=q^{-s}$ and $s$ in \eqref{eq:CM-z-s} to emphasize the different roles they play.
In fact, $\CM(f_\Phi)(z,\chi)$ in the numerator is considered as a function of variable $(z,\chi)\in \widehat{F^\times}$ for Mellin transform, studied by Igusa in \cite{Ig78}. 
We follow the notation there  as in Definition~\ref{Z+-}.
The normalizer in the denominator is a rational function of $z$, also a product of $L$-factors. 
To trace the involved $L$-factors, we follow usual convention to consider it as a function of variable $(s,\chi)$ instead of $(z,\chi)$.  

In order to show that $\CM(f_\Phi)(z,\chi)$ has the asymptotic behavior \eqref{eq:Z-asymptotic-1}, we do the following calculations.
Whenever $\chi$ is unramified (i.e. $\chi$ is trivial), we have
$$
L(s,\chi) = \frac{1}{1-\chi(\vpi)q^{-s}} = \frac{1}{1-z}
$$
and
$$
L(2s+2i+1,\chi^{2}) = \frac{1}{1-\chi^{2}(\vpi)q^{-2s-2i-1}} =
\frac{1}{1-q^{-2i-1}z^{2}},
$$
where the last identity, up to constant, is equal to
$$
\frac{1}{1-q^{-i-\frac{1}{2}}z}+\frac{1}{1+q^{-i-\frac{1}{2}}z}.
$$
Using the fact that the following polynomials in $z$
$$
\{ 1-z, 1-q^{-i-\frac{1}{2}}z, 1+q^{-i-\frac{1}{2}}z\mid i=0,...,n-1\}
$$
are coprime to each other, we get that
$
\CM(f_{\Phi})(z,\chi)
$
has the desired asymptotic behavior in \eqref{eq:Z-asymptotic-1}.

Now, let us show that $f_{\varrho\cdot\wh{\Phi}}$ belongs to $\CS^{-}_n(F^\times)$.
By definition, $f_{\varrho\cdot\wh{\Phi}}$ is locally constant and $|f_{\varrho\cdot\wh{\Phi}}|\leq f_{|\widehat{\Phi}|}$.
It follows that $\CM(f_{\varrho\cdot\wh{\Phi}})(z,\chi)$ admits a meromorphic continuation to a rational function in $\BC(z)$
and is identically zero for almost all $\chi\in\widehat{\CO^\times}$.
Since $\CM(f_{\varrho\cdot\wh{\Phi}})$ is absolutely convergent for $\Re(s)>0$,
we may apply  \cite[Proposition 7.2.2]{Ig00}  and the Mellin inversion formula in \cite[\S 5.3, Remark 2]{Ig78}, and get
$$
f_{\varrho\cdot\wh{\Phi}}(x)=\CM^{-1}(\CM(f_{\varrho\cdot\wh{\Phi}})(z,\chi))(x)
$$
with $x\in\varpi^n\CO^\times$ for all $n$, and hence for all $x\in F^\times$.

In order to prove that $f_{\varrho\cdot\wh{\Phi}} \in \CS^{-}_n(F^\times)$, or equivalently that $\CM(f_{\varrho\cdot\wh{\Phi}})\in \CZ^{-}_{n}(\wh{F^\times})$,
it is sufficient to show that the Mellin transform $\CM(f_{\varrho\cdot\wh{\Phi}})(z,\chi)$ has the asymptotic
behavior \eqref{eq:Z-asymptotic-2}. This can be carried out by using the functional equation \eqref{eq:FE-F-0} and then the same arguments we used to deal with $\CM(f_\Phi)$.

First, we write
\begin{eqnarray*}
\CM(f_\Phi\cdot|\cdot|^{-n+\frac{1}{2}})(z,\chi)&=&\CM(f_{\Phi})(s-n+\frac{1}{2},\chi)\\
\CM(f_{\varrho\cdot\wh{\Phi}}\cdot|\cdot|^{\frac{1}{2}})(z^{-1},\chi^{-1})&=&\CM(f_{\varrho\cdot\wh{\Phi}})(-s+\frac{1}{2},\chi^{-1})
\end{eqnarray*}
with $z=q^{-s}$,
to emphasize the Mellin transforms as the functions of variable $(z,\chi)\in \widehat{F^\times}$ instead of $(s,\chi)$.
Then the functional equation in \eqref{eq:FE-F-0} turns to be
\begin{equation}\label{fe-f}
\CM(f_{\varrho\cdot\wh{\Phi}}|\cdot|^{\frac{1}{2}})(z,\chi)=|2|^{-2n^2}\chi_s(2)^{-2n}\beta_{\psi}(\chi_s^{-1})\CM(f_\Phi|\cdot|^{-(n-\frac{1}{2})})(z^{-1},\chi^{-1}),
\end{equation}
i.e.,
\begin{equation*} 
\CM(f_{\varrho\cdot\wh{\Phi}}|\cdot|^{\frac{1}{2}})(q^{-s},\chi)=|2|^{-2n^2}\chi_s(2)^{-2n}\beta_{\psi}(\chi_s^{-1})\CM(f_\Phi|\cdot|^{-(n-\frac{1}{2})})(q^{s},\chi^{-1}),
\end{equation*}
In order to understand the asymptotic behavior of the right-hand side in the above functional equation, we write out the local gamma factors in $\beta_{\psi}(\chi_s^{-1})$ in terms of the corresponding
local $L$-factors and local $\varepsilon$-factors.
Hence we are able to write the left-hand side of the above functional equation as follows:
\begin{eqnarray}
 &&\CM(f_{\varrho\cdot\wh{\Phi}}|\cdot|^{\frac{1}{2}})(z,\chi)=
 \CM(f_{\varrho\cdot\wh{\Phi}}|\cdot|^{\frac{1}{2}})(q^{-s},\chi)\label{eq:M-f-hat}\nonumber\\
=&&|2|^{-2n^2}\chi_s(2)^{-2n}\varepsilon(-s-\frac{2n-1}{2},\chi^{-1},\psi)
\prod^{n-1}_{i=0}\varepsilon(-2s-2i,\chi^{-2},\psi) \label{eq:M-f-hat-eps}\nonumber\\
&&\qquad\times L(s+\frac{2n+1}{2},\chi)\prod_{i=0}^{n-1}L(2s+2i+1,\chi^2) \label{eq:M-f-hat-L}\\
&&\qquad\qquad\times\frac{\CM(f_\Phi|\cdot|^{-(n-\frac{1}{2})})(z^{-1},\chi^{-1})}
{L(-s-\frac{2n-1}{2},\chi)\prod_{i=0}^{n-1}L(-s-2i,\chi^2)}, \label{eq:M-f-hat-hol} \nonumber
\end{eqnarray}
with $z=q^{-s}$.
Like \eqref{eq:CM-z-s}, we mix the use of $z=q^{-s}$ and $s$ in \eqref{eq:M-f-hat-L} to 
study the numerator as a function of variable $(z,\chi)\in \widehat{F^\times}$ and 
trace the $L$-factors in the denominator.
In the remaining of this proof, we follow in this manner.. 

Recall that under the assumption that the additive character $\psi$ has conductor $\CO$  (i.e., $\psi\vert_{\CO}=1$ and $\psi\vert_{\varpi^{-1}\CO}\ne 1$) one has
$$\varepsilon(s,\chi,\psi)=q^{e(\chi)(\frac{1}{2}-s)}\varepsilon(\frac{1}{2},\chi,\psi),
$$
where $e(\chi)$ is the conductor of $\chi$ (i.e., the smallest positive integer such that $\chi\vert_{1+\varpi^m\CO}=1$). It follows that $\veps(s,\chi,\psi)$ and $\veps(2s,\chi^{2},\psi)$ are always units in
$
\BC[q^{-s},q^{s}] = \BC[z,z^{-1}].
$
Hence up to a unit in $\BC[z,z^{-1}]$, we obtain the following identity:
\begin{align*}
&\frac{\CM(f_{\varrho\cdot\wh{\Phi}}|\cdot|^{1/2})(z,\chi)}{L(s+\frac{2n+1}{2},\chi)
\prod_{i=0}^{n-1}L(2s+2i+1,\chi^{2})
}
\\
&\qquad\qquad\qquad \text{ ``=" }
\frac{\CM(f_{\Phi}|\cdot|^{-(n-\frac{1}{2})})(z^{-1},\chi^{-1})}{
L(-s-\frac{2n-1}{2},\chi^{-1})\prod_{i=0}^{n-1}
L(-s-2i, \chi^{-2})
}.
\end{align*}
Using the result we just proved for $f_{\Phi}$, we obtain immediately  that
$$
\frac{\CM(f_{\varrho\cdot\wh{\Phi}}|\cdot|^{1/2})(z,\chi)}{L(s+\frac{2n+1}{2},\chi)
\prod_{i=0}^{n-1}L(2s+2i+1,\chi^{2})
}\in \BC[z,z^{-1}].
$$
By a shift from $s+\frac{1}{2}$ to $s$, we obtain that
$$
\frac{\CM(f_{\varrho\cdot\wh{\Phi}})(z,\chi)}{L(s+n,\chi)
\prod_{i=0}^{n-1}L(2s+2i,\chi^{2})
}\in \BC[z,z^{-1}].
$$
Finally by using the same argument for $f_{\Phi}$ as given above, we obtain indeed the desired leading terms for $\CM(f_{\varrho\cdot\wh{\Phi}})(z,\chi)$ as in \eqref{eq:Z-asymptotic-2} and deduce that
$
\CM(\wh{f}_{\Phi})(z,\chi)\in \CZ^{-}_{n}(\wh{F}^{\times}).
$
\end{proof}

\begin{rmk}
When $\Phi$ is taken to be the characteristic function of $S_{2n+1}(\CO)$, by \cite[Chapter III]{S97} and \cite[Proposition 10.3.1]{Ig00}, one has
$$
\CM(f_{\Phi})(z,1) =
\frac{1}{(1-z)\prod_{i=0}^{n-1}(1-q^{-(2i+1)}z^{2})}.
$$
This formula is more precise, but more special than what we got in Proposition \ref{pro-2fi}. Due to our choice of the Haar measure ${\rm d}^*t$ on $F^\times$,
we need a shift from $s$ to $s-1$ in the formula in \cite[Proposition 10.3.1]{Ig00}.
\end{rmk}

\subsection{Schwartz spaces and Fourier operator}\label{ssec-ssgft}

We are ready to introduce certain spaces of Schwartz type and define a Fourier (convolution) operator over $F^\times$ with the generalized function $\eta_{\pvs,\psi}(x)$ on $F^\times$ as the kernel function,
in order to understand the $\gamma$-function $\beta_\psi(\chi_s)$ as in \eqref{eq:eta-intertwining}, by means of harmonic analysis on $F^\times$.

First, we re-arrange the functional equation in Corollary \ref{fe-Radon} as
\begin{align}
 & |2|^{2n^2}\chi_s(2)^{-2n}\int_{F^\times}f_{\varrho\cdot\wh{\Phi}}(t)|t|^{n+1}\chi^{-1}(t)|t|^{-s-\frac{2n+1}{2}}\ud^* t\nonumber \\
&\qquad\qquad\qquad\qquad =  \beta_{\psi}(\chi_s)\int_{F^\times}f_\Phi(t)|t|^{-2n}\chi(t)|t|^{s+\frac{2n+1}{2}}\ud^*t. \label{eq:FE-F-1}	
\end{align}
By changing variable $t\mapsto 2^{-2n}t$ on the left-hand side of \eqref{eq:FE-F-1}, we obtain a new version of the functional equation that relates $f_\Phi$ to
$f_{\varrho\cdot\wh{\Phi}}$:
\begin{align}
 & |2|^{2n^2-n}\int_{F^\times}f_{\varrho\cdot\wh{\Phi}}(2^{-2n}t)|t|^{n+1}\chi_{s+\frac{2n+1}{2}}(t)^{-1}\ud^*t\nonumber \\
&\qquad\qquad\qquad\qquad\qquad = \beta_{\psi}(\chi_s)\int_{F^\times}f_\Phi(t)|t|^{-2n}\chi_{s+\frac{2n+1}{2}}(t)\ud^* t. \label{eq:FE-F}
\end{align}

\begin{dfn}[Spaces $\CS^\pm_{\pvs}(F^\times)$]\label{schwartzF}
Define $\CS^+_{\pvs}(F^\times)$ to be the subspace of $\CC^\infty(F^\times)$ that consists of all functions of the form: $|t|^{-2n}f_\Phi(t)$ with all $\Phi\in\CC^\infty_c(S_{2n+1})$; and define
$\CS^-_{\pvs}(F^\times)$ to be the subspace of $\CC^\infty(F^\times)$ that consists of all functions of the form: $|t|^{n+1}f_{\varrho\cdot\wh{\Phi}}(t)$ with all $\Phi\in\CC^\infty_c(S_{2n+1})$.
\end{dfn}

The notation $\CS^\pm_{\pvs}(F^\times)$ indicates that they are the certain spaces of the smooth functions on $F^\times$ that come from the prehomogeneous vector space ($\pvs$) $S_{2n+1}(F)$.
By Proposition \ref{pro-2fi}, it is clear from Definition \ref{schwartzF} that
\[
\CC^\infty_{c}(F^\times)\subset	\CS^+_{\pvs}(F^\times)\subset |\cdot|^{-2n}\CS^{+}_n(F^\times)
\]
and
\[
\CC^\infty_{c}(F^\times)\subset\CS^-_{\pvs}(F^\times)\subset |\cdot|^{n+1}\CS^{-}_n(F^\times),	
\]
where $|\cdot|^m \CS^{\pm}_n(F^\times):=\{|t|^mf(t) \mid f(t)\in \CS^{\pm}_n(F^\times)\}$.  For any $f\in \CS^+_{\pvs}(F^\times)$, we write $f(t)=|t|^{-2n}f_\Phi$ for some $\Phi\in\CC_c^\infty(S_{2n+1})$.
We define
\begin{equation}\label{whf}
\wh{f|\cdot|^{2n}}(t):=f_{\varrho\cdot\wh{\Phi}}(t).
\end{equation}
Define a linear transform $\FL$
from  $\CS^+_{\pvs}(F^\times)$ to $\CS^-_{\pvs}(F^\times)$  by
\begin{equation}\label{LT}
 \FL(f)(t):=|2|^{2n^2-n}|t|^{n+1} \widehat{f|\cdot|^{2n}}(2^{-2n}t)
\in \CS^-_{\pvs}(F^\times)
\end{equation}
for any $f(t)\in \CS^+_{\pvs}(F^\times)$. Finally, we are able to establish a functional equation with the $\gamma$-factor $\beta_\psi(\chi_s)$.

\begin{thm}[Functional Equation for $\beta_\psi(\chi_s)$]\label{thm:LT}
The linear transform $\FL$ as in \eqref{LT} is well-defined, and for any $f\in \CS^+_{\pvs}(F^\times)$, the following functional equation
\begin{equation}\label{eq:FE-eta}
\int_{F^\times}\FL(f)(t)\chi_{s+\frac{2n+1}{2}}(t)^{-1}\ud^* t
=\beta_{\psi}(\chi_s)\int_{F^\times}f(t)\chi_{s+\frac{2n+1}{2}}(t)\ud^* t
\end{equation}
holds, where $\beta_{\psi}(\chi_s)$ is defined as in \eqref{eq:eta-intertwining}.
\end{thm}

\begin{proof}
In fact, it suffices to show that $f_\Phi\mapsto f_{\varrho\cdot\wh{\Phi}}$ is a well-defined mapping.
That is,
if there exist $\Phi_1$ and $\Phi_2$ in $\CC^\infty_c(S_{2n+1})$ such that $f_{\Phi_1}(t)=f_{\Phi_2}(t)$,
we need to show that  $f_{\varrho\cdot\wh{\Phi}_1}(t)=f_{\varrho\cdot\wh{\Phi}_2}(t)$.
In fact,  take $\Phi=\Phi_1-\Phi_2$.
Then  $f_{\varrho\cdot\wh{\Phi}}=f_{\varrho\cdot\wh{\Phi}_1}-f_{\varrho\cdot\wh{\Phi}_2}\in  \CS^{-}_n(F^\times)$. If assume that $f_\Phi(t)=f_{\Phi_1}(t)-f_{\Phi_2}(t)\equiv 0$,
then by the functional equation in \eqref{eq:FE-F}, one must have that
$\CM(f_{\varrho\cdot\wh{\Phi}})\equiv 0$ for all  $\chi$. By \cite[Chapter 1, Theorem 5.3]{Ig78} (Proposition \ref{Igusa78-53}), $\CM$ is injective on $\CS^{-}_n(F^\times)$. Hence we must have that
$f_{\varrho\cdot\wh{\Phi}}(t)=0$. This implies that $f_{\varrho\cdot\wh{\Phi}_1}(t)=f_{\varrho\cdot\wh{\Phi}_2}(t)$.

The functional equation relating $f\in\CS^+_\pvs(F^\times)$ and $\FL(f)\in\CS^-_\pvs(F^\times)$ follows from the functional equation in \eqref{eq:FE-F}.
\end{proof}

It is important to note that the integral on the left-hand side of the functional equation in \eqref{eq:FE-eta} is absolutely convergent for $\Re(s)$ sufficiently negative,
while the integral on the right-hand side is absolutely convergent for $\Re(s)$ sufficiently positive.
However, the functional equation  \eqref{eq:FE-eta} holds by meromorphic continuation as functionals with the parameter $s\in\BC$ on the both sides.
This functional equation relating $f\in\CS^+_\pvs(F^\times)$ to $\FL(f)\in\CS^-_\pvs(F^\times)$ is essential to our construction of the generalized function $\eta_{\pvs,\psi}$ (Theorem \ref{thm:eta}),
and to the proof of the key technical result of the paper (Theorem \ref{thm:CFP-M}).


\subsection{Generalized function $\eta_{\pvs,\psi}$}\label{ssec-etapvs}
We are going to construct the generalized function $\eta_{\pvs,\psi}(x)$ on $F^\times$ in an analytic way, so that the generalized Fourier transform from
$f\in\CS^+_{\pvs}(F^\times)$ to $\FL(f)\in\CS^-_{\pvs}(F^\times)$ as defined in \eqref{LT} becomes a Fourier (convolution) operator $\FL_{\eta_{\pvs,\psi}}$ with $\eta_{\pvs,\psi}(x)$ as the kernel function
(Theorem \ref{thm:eta}).

For any generalized function $\eta(t)$ on $F^\times$, the convolution of $\eta$ with a proper function $f$ on $F^\times$ is defined to be the {\sl principle value integral}:
\begin{align}\label{eta-conv-1}
\eta(f)(t)=(\eta*f)(t)
&:=\int_{F^\times}^\pv \eta(x)f(x^{-1}t)\ud^*x\nonumber\\
&=\lim_{k\to \infty}\int\limits_{q^{-k}\leq |x|\leq q^{k}}\eta(x)f(x^{-1}t)\ud^* x,
\end{align}
if the limit exists.

\begin{thm} \label{thm:eta}
There exists a locally constant function $\eta_{\pvs,\psi}$ on $F^\times$, which will be explicitly defined in \eqref{def:eta-rho-psi} below,
such that the generalized Fourier transform $\FL(f)$, as defined in \eqref{LT}, can be expressed as the principal value integral of the form:
\begin{equation}\label{eq:tilde-f-convolution}
\FL(f)(t)=\FL_{\eta_{\pvs,\psi}}(f)(t):=((\eta_{\pvs,\psi}|\cdot|^{\frac{2n+1}{2}})*f^\vee)(t)
\end{equation}
for all $f\in\CS^{+}_\pvs(F^\times)$, where $f^\vee(t)=f(t^{-1})$.
Moreover, for any character $\chi_s$ of $F^\times$, the $\chi_s$-Fourier coefficient of the generalized function $\eta_{\pvs,\psi}$ is given by the following formula:
\begin{equation}\label{eta-beta}
\eta_{\pvs,\psi}(\chi_s)=\beta_\psi(\chi_s),
\end{equation}
where $\beta_{\psi}(\chi_s)$ is defined as in \eqref{eq:eta-intertwining}.
\end{thm}

Note that when $n=0$, the generalized function can be written as $\eta_{\pvs,\psi}(t)=\psi(t)|t|^{\frac{1}{2}}\zeta(1)^{-1}$. It defines the Fourier operator over $F^\times$, which is the restriction to $F^\times$
of the the usual Fourier transform on $F$.

\begin{proof}
Define $\mathds{1}_{k}=\frac{1}{{\rm vol}(1+\varpi^k\CO,\mathrm{d}^* x)}\mathrm{ch}_{1+\varpi^k\CO}$, where $\mathrm{ch}_{1+\varpi^k\CO}$ is  the characteristic function of $1+\varpi^k\CO$ for $k\geq 1$.
For $\chi\in\widehat{\CO^{\times}}$, denote $Z_k=\CM(\FL(\mathds{1}_{k})|\cdot|^{-\frac{2n+1}{2}})$.
By Proposition \ref{pro-2fi} and the functional equation in \eqref{eq:FE-eta}, we have $Z_k(z,\chi)\in \CZ_{n}^{-}(\widehat{F}^\times)$ and
\begin{align*}
 Z_k(z,\chi)=&  \beta_{\psi}(\chi_s^{-1})
\CM(\mathds{1}_{k}|\cdot|^{\frac{2n+1}{2}})(z^{-1},\chi^{-1})\\
 =& \beta_{\psi}(\chi_s^{-1})\times
\begin{cases}
	1 &\text{ if } e(\chi)\leq k\\
	0 &\text{ if } e(\chi)> k.\\
\end{cases}
\end{align*}
Recall that $e(\chi)$ is the smallest positive integer $m$ such that $\chi\vert_{1+\varpi^m\CO}=1$ and   $z=q^{-s}$.
By \cite[Theorem 5.3]{Ig78} (Proposition \ref{Igusa78-53}), the Mellin inversion $\CM^{-1}(Z_k(z,\chi))$  is given by
\begin{align}
 \CM^{-1}(Z_k(z,\chi))(x)=&\sum_{\chi\in \wh{\CO^\times}}
 \chi({\rm ac}(x))^{-1} \Res_{z=0}Z_k(z,\chi)z^{-{\rm ord}(x)-1}\nonumber\\
=&\sum_{\substack{\chi\in \wh{\CO^\times},\ e(\chi)\leq k}}\chi({\rm ac}(x))^{-1}\Res_{z=0}\beta_{\psi}(\chi^{-1}_s)z^{-{\rm ord}(x)-1}, \label{eq:Mellin-inversion}
\end{align}
which is a finite sum.

From \cite[p.143]{BH06}, we have
\begin{equation}\label{epsilon}
\veps(s,\chi,\psi)  =q^{e(\chi)(\frac{1}{2}-s)} \veps(\frac{1}{2},\chi,\psi) =
q^{\frac{1}{2}e(\chi)}\veps(\frac{1}{2},\chi,\psi)
z^{e(\chi)}.
\end{equation}
Hence we deduce that
\begin{equation}
{\rm Res}_{z=0}~\varepsilon(s,\chi,\psi)z^{-m-1}=0, \text{ unless } m=e(\chi).		
\end{equation}
Whenever $\chi^{2}$ is ramified, we have, from \eqref{eq:eta-intertwining},
$$
\beta_{\psi}(\chi^{-1}_{s}) =
\veps(-s-\frac{(2n-1)}{2},\chi^{-1},\psi)
\prod_{i=0}^{n-1}\veps(-2s-2i, \chi^{-2},\psi),
$$
which, up to a nonzero constant, is equal to
$$
q^{s\cdot e(\chi^{-1})+2ns\cdot e(\chi^{-2})} = z^{-(e(\chi)+2n\cdot e(\chi^{2}))}.
$$
Hence we obtain that if $\chi^{2}$ is ramified, then
\begin{equation}\label{4.18}
\Res_{z=0}\beta_{\psi}(\chi^{-1}_{s})z^{-m-1} = 0, \text{ unless $m=-e(\chi)-2n\cdot e(\chi^{2})$}.
\end{equation}
If $\chi^2$ is unramified, then
\begin{equation}
{\rm Res}_{z=0}\beta_{\psi}(\chi_s^{-1})z^{-m-1}=0, 	
\text{ unless } m\geq -(2n+1+e(\chi)).	
\end{equation}

Let $N\geq 1$. For any $\ell>k\gg \frac{N+\ord(2)n}{2n+1}$,
from the Mellin inversion formula in \eqref{eq:Mellin-inversion}, we deduce that
\begin{align*}
&(\FL(\mathds{1}_{\ell})(x)-\FL(\mathds{1}_{k})(x))|x|^{-\frac{2n+1}{2}}\\
&\qquad\qquad=\quad\CM^{-1}(Z_\ell(z,\chi))(x)-\CM^{-1}(Z_k(z,\chi))(x)\\
&\qquad\qquad= \sum_{k< e(\chi)\leq \ell }	\chi({\rm ac}(x))^{-1} {\rm Res}_{z=0}\beta_{\psi}(\chi_s^{-1})z^{-{\rm ord }(x)-1}.
\end{align*}
When $k$ is sufficiently positive and $e(\chi)>k$, we must have that $\chi$ is ramified. Through straightforward calculation we obtain that $e(\chi) = e(\chi^{2})+\ord(2)$.

For any $x$ with ${\rm ord}(x)\geq -N$, we have that
$$
-{\rm ord}(x)\leq N\ll (2n+1)k+2n\cdot\ord(2) \leq  e(\chi)+2n\cdot e(\chi^{2}),
$$
 and by \eqref{4.18}, we obtain that
 \[
 \FL(\mathds{1}_{\ell})(x)=\FL(\mathds{1}_{k})(x).
 \]
Hence $\lim_{k\to \infty}\FL(\mathds{1}_{k})(x)$ is stably convergent, that is, for any $x$, there exists $N$ such that $\FL(\mathds{1}_{k})(x)=\FL(\mathds{1}_{N})(x)$ for all $k\geq N$.

Define the generalized function $\eta_{\pvs,\psi}(x)$ by
\begin{equation}\label{def:eta-rho-psi}
\eta_{\pvs,\psi}(x)=|x|^{-\frac{2n+1}{2}}\lim_{k\to \infty}\FL(\mathds{1}_{k})(x)
\end{equation}
which is a locally constant function on $F^\times$. By the definition of $Z_k(z,\chi)$ at the beginning of this proof, and the Mellin inversion of $Z_k(z,\chi)$ in \eqref{eq:Mellin-inversion}, together with the
definition of the generalized function $\eta_{\pvs,\psi}$ in \eqref{def:eta-rho-psi}, we are able to obtain a formula for $\eta_{\pvs,\psi}$ below, which is essential to the rest of the proof. More precisely,
for a fixed $k$, we obtain from \eqref{eq:Mellin-inversion} that
\begin{align*}
|x|^{-\frac{2n+1}{2}}\FL(\mathds{1}_{k})(x) &= \CM^{-1}(Z_{k}(z,\chi))(x)
\\
&=
\sum_{\chi\in \wh{\CO^{\times}}, e(\chi)\leq k}
\chi(\ac(x))^{-1}\Res_{z=0}\bet_{\psi}(\chi^{-1}_{s})z^{-\ord(x)-1}.
\end{align*}
After passing the limit $k\to \infty$, we obtain the following formula:
\begin{equation}\label{eq:eta-MI}
\eta_{\pvs,\psi}(x)= \sum_{\chi\in \wh{\CO^\times}}
 \chi({\rm ac}(x))^{-1}\Res_{z=0}\beta_{\psi}(\chi_s^{-1})z^{-{\rm ord}(x)-1}.	
\end{equation}
Note that the sum over $\chi$ is a finite sum for any $x$ belonging to any given open compact subgroup of $F^\times$.

In order to establish the identity in \eqref{eq:tilde-f-convolution} for the generalized function $\eta_{\pvs,\psi}(x)$,
we start with its left-hand side. For $f\in\CS^+_\pvs(F^\times)$, we may assume that $f$ is $(1+\varpi^N\CO)$-invariant, according to Lemma \ref{K-finite}.
Applying the Mellin inversion on the both sides of \eqref{eq:FE-eta},
we have
\[
\FL(f)(t)=|t|^{\frac{2n+1}{2}}\CM^{-1}\left(\beta_{\psi}(\chi_s^{-1}) \CM(f|\cdot|^{\frac{2n+1}{2}})(z^{-1},\chi^{-1})\right)(t),
\]
for $t\in F^\times$,
where
$$
\CM(f|\cdot|^{\frac{2n+1}{2}})(z^{-1},\chi^{-1})=\int_{F^\times}f(x)|x|^{-s+\frac{2n+1}{2}}\chi^{-1}(x)\ud x.
$$
Since $f$ is $(1+\varpi^N\CO)$-invariant, $\CM(f|\cdot|^{\frac{2n+1}{2}})(z^{-1},\chi^{-1})=0$ for any $\chi\in\wh{\CO^\times}$ with $e(\chi)=e(\chi^{-1})>N$.
By the Mellin inversion formula in \eqref{eq:Mellin-inversion}, we obtain that $\FL(f)(t)$ equals
\begin{align}
&|t|^{\frac{2n+1}{2}}\sum_{\substack{\chi\in \wh{\CO^\times}\\ e(\chi)\leq N}} {\rm Res}_{z=0}\left(\beta_{\psi}(\chi_s^{-1})\CM(f|\cdot|^{\frac{2n+1}{2}})(z^{-1},\chi^{-1})z^{-{\rm ord}(t)-1}\right)\chi({\rm ac}(t))^{-1}\nonumber\\
&\qquad\qquad =|t|^{\frac{2n+1}{2}}\sum_{\substack{\chi\in \wh{\CO^\times}\\  e(\chi)\leq N}}\chi({\rm ac}(t))^{-1}
\sum_{i\in \BZ}{\rm Res}_{z=0} \beta_{\psi}(\chi_s^{-1})z^{-i-1} \nonumber\\
&\qquad\qquad\qquad\qquad\qquad\times
 \Res_{z=0} \CM(f|\cdot|^{\frac{2n+1}{2}})(z^{-1},\chi^{-1})z^{-({\rm ord}(t)-i)-1}.\label{eq:f-residue}
\end{align}
Note that the second sum over $i$ is a finite sum. Here we use the fact that for two meromorphic functions $g,h$ near $z=0$ with poles at $z=0$,
$$
\mathrm{Res}_{z=0}(g\cdot h) = \sum_{i}\mathrm{Res}_{z=0}fz^{i}\cdot \mathrm{Res}_{z=0}gz^{-i-1}
$$
with $i$ running over a finite set.

First, we compute the residual term
$$
\Res_{z=0} \CM(f|\cdot|^{\frac{2n+1}{2}})(z^{-1},\chi^{-1})z^{-({\rm ord}(t)-i)-1}
$$
from \eqref{eq:f-residue}. For $\Re(s)$ sufficiently negative, we have that
\begin{align*}
\CM(f|\cdot|^{\frac{2n+1}{2}})(z^{-1},\chi^{-1})
&=\int_{F^{\times}}
f(x)|x|^{\frac{2n+1}{2}}\chi^{-1}(x)|x|^{-s}\ud^*x
\\
&=
\sum_{k\in \BZ}
\int_{\ord(x) = k}
f(x)|x|^{\frac{2n+1}{2}}\chi^{-1}(x)|x|^{-s}\ud^*x
\\
&=
\sum_{k\in \BZ}
z^{-k}
\int_{\ord(x) = k}
f(x)|x|^{\frac{2n+1}{2}}
\chi^{-1}(x)\ud^*x.
\end{align*}
Here we abuse the notation and restrict the measure ${\rm d}^*x$ to the open compact subset $\{ x\in F^{\times}\mid \ord(x) = k \}$.
It follows that
\begin{align*}
&\CM(f|\cdot|^{\frac{2n+1}{2}})(z^{-1},\chi^{-1})
z^{-(\ord(t) - i)-1}
\\
&\qquad\qquad\qquad =\sum_{k\in \BZ}
z^{-k-(\ord(t)-i)-1}
\int_{\ord(x) = k}f(x)|x|^{\frac{2n+1}{2}}\chi^{-1}(x)\ud^*x.
\end{align*}
Hence we obtain the following residue formula:
\begin{align*}
&\mathrm{Res}_{z=0}
\CM(f|\cdot|^{\frac{2n+1}{2}})(z^{-1},\chi^{-1})
z^{-(\ord(t)-i)-1}
\\
&\qquad\qquad=
\int_{\ord(x) = -\ord(t)+i}
f(x)|x|^{\frac{2n+1}{2}}\chi^{-1}(x)\ud^*x.
\end{align*}
Then by plugging the residue formula back into \eqref{eq:f-residue} we obtain the following identity for $\FL(f)(t)$,
\begin{align}\label{4.25}
\nonumber&
|t|^{\frac{2n+1}{2}}
\sum_{\chi\in \wh{\CO^{\times}}, e(\chi)\leq N}
\chi(\ac(t))^{-1}
\sum_{i}
\mathrm{Res}_{z=0}
\beta_{\psi}(\chi^{-1}_{s})z^{-i-1}
\\
&\qquad\qquad\qquad\times
\int_{\ord(x) = -\ord(t)+i}
f(x)|x|^{\frac{2n+1}{2}}\chi^{-1}(x)\ud^*x.
\end{align}
Note that the summations over both $\chi\in \wh{\CO^\times}$ and $i\in \BZ$ are finite.

In order to compute the right-hand side of \eqref{eq:tilde-f-convolution}, we are going to use the formula for the generalized function $\eta_{\pvs,\psi}(x)$ in \eqref{eq:eta-MI}.
By definition, the right-hand side of \eqref{eq:tilde-f-convolution} is equal to
$$
\lim_{k\to \infty}
\sum_{i=-k}^{k}
\int_{\ord(m) = i}
\eta_{\pvs,\psi}(m)|m|^{\frac{2n+1}{2}}f(mt^{-1})\ud^*m.
$$
We are going to show that the series is convergent, and is equal to $\FL(f)(t)$ as given in \eqref{4.25}.

By using \eqref{eq:eta-MI},  we obtain the following expression for the compact integration $\int_{\ord(m) = i}$ with a fixed $i$
\begin{align}\label{ord=i-1}
&\int_{\ord(m) = i}
\eta_{\pvs,\psi}(m)|m|^{\frac{2n+1}{2}}f(mt^{-1})\ud^*m\nonumber\\
&\qquad\qquad=
\int_{\ord(m) = i}
(\sum_{\chi\in \wh{\CO^{\times}}}\chi(\ac(m))^{-1}
\Res_{z=0}\beta_{\psi}(\chi^{-1}_{s})z^{-i-1}
)
\nonumber\\
&\qquad\qquad\qquad\qquad\qquad\qquad\times |m|^{\frac{2n+1}{2}}
f(mt^{-1})\ud^*m.
\end{align}
From (\ref{4.18}) we know that the summation over $\chi\in \wh{\CO^{\times}}$ is finite. Hence we can exchange the order of the summation over $\chi\in \wh{\CO^{\times}}$ and
the compact integration $\int_{\ord(m) = i}$, and obtain that the right-hand side of \eqref{ord=i-1} is equal to
\begin{equation}\label{ord=i-2}
\sum_{\chi\in \wh{\CO^{\times}}}
\Res_{z=0}\beta_{\psi}(\chi^{-1}_{s})z^{-i-1}
\int_{\ord(m)=i}
\chi(\ac(m))^{-1}|m|^{\frac{2n+1}{2}}f(mt^{-1})\ud^*m.
\end{equation}
By changing variable $m\to mt$, we get that \eqref{ord=i-2} is equal to
\begin{align}\label{ord=i-3}
&
|t|^{\frac{2n+1}{2}}
\sum_{\chi\in \wh{\CO^{\times}}}
\chi(\ac(t))^{-1}
\Res_{z=0}\beta_{\psi}(\chi^{-1}_{s})z^{-i-1}\nonumber
\\
&\qquad\qquad\qquad\times \int_{\ord(m)=-\ord(t)+i}
\chi(\ac(m))^{-1}|m|^{\frac{2n+1}{2}}f(m)\ud^*m.
\end{align}
Since $\chi\in \wh{\CO^{\times}}$, one must have that $\chi(\ac(m)) = \chi(m)$.  Hence \eqref{ord=i-3} can be written as
\begin{align}\label{ord=i-4}
&|t|^{\frac{2n+1}{2}}
\sum_{\chi\in \wh{\CO^{\times}}}
\chi(\ac(t))^{-1}
\Res_{z=0}\beta_{\psi}(\chi^{-1}_{s})z^{-i-1}\nonumber
\\
&\qquad\qquad\qquad\times
\int_{\ord(m)=-\ord(t)+i}
\chi(m)^{-1}|m|^{\frac{2n+1}{2}}f(m)\ud^*m.
\end{align}
Since $f$ is $(1+\vpi^{N}\CO)$-invariant, the integral
$$
\int_{\ord(m)=-\ord(t)+i}
\chi(m)^{-1}|m|^{\frac{2n+1}{2}}f(m)\ud^*m
$$
is possibly non-zero only for $\chi\in \wh{\CO^{\times}}$ with $e(\chi)\leq N$. Therefore we deduce that
\begin{equation}\label{ord=i-5}
\sum_{i=-k}^{k}
\int_{\ord(m) = i}
\eta_{\pvs,\psi}(m)|m|^{\frac{2n+1}{2}}f(mt^{-1})\ud^*m
\end{equation}
is equal to
\begin{align}\label{ord=i-6}
&|t|^{\frac{2n+1}{2}}
\sum_{i=-k}^{k}
\sum_{\chi\in \wh{\CO^{\times}},e(\chi)\leq N}
\chi(\ac(t))^{-1}
\Res_{z=0}\beta_{\psi}(\chi^{-1}_{s})z^{-i-1}\nonumber
\\
&\qquad\qquad\qquad\times \int_{\ord(m)=-\ord(t)+i}
\chi(m)^{-1}|m|^{\frac{2n+1}{2}}f(m)\ud^*m.
\end{align}
By exchanging the order of the summation $\sum_{i=-k}^{k}$ and the summation $\sum_{\chi\in \wh{\CO^{\times}},e(\chi)\leq N}$, we obtain that \eqref{ord=i-5} is equal to
\begin{align}\label{ord=i-7}
&
|t|^{\frac{2n+1}{2}}
\sum_{\chi\in \wh{\CO^{\times}},e(\chi)\leq N}
\sum_{i=-k}^{k}
\chi(\ac(t))^{-1}
\Res_{z=0}\beta_{\psi}(\chi^{-1}_{s})z^{-i-1}\nonumber
\\
&\qquad\qquad\qquad\times
\int_{\ord(m)=-\ord(t)+i}
\chi(m)^{-1}|m|^{\frac{2n+1}{2}}f(m)\ud^*m.
\end{align}
Finally, the right-hand side of \eqref{eq:tilde-f-convolution} is the limit by taking $k\to \infty$ of \eqref{ord=i-7}. By taking the limit $\lim_{k\to\infty}$ of the expression
in \eqref{ord=i-7}, we arrive at (\ref{4.25}). This establishes the identity in \eqref{eq:tilde-f-convolution}.

It remains to show that for any character $\chi_s$ of $F^\times$, the $\chi_s$-Fourier coefficient $\eta_{\pvs,\psi}(\chi_s)$ is equal to $\beta_\psi(\chi_s)$. To do so,
we need the functional equation \eqref{eq:FE-eta} in Theorem \ref{thm:LT}.
Recall that for a fixed $\chi\in \wh{\CO^{\times}}$ with fixed integer $N>e(\chi)$, and $\Re(s)$ sufficiently positive, the convolution action of the generalized function $\eta_{\pvs,\psi}(x)$ on $\chi_{s}^{-1}$ is defined,
as in \eqref{eta-conv-1}, to be
\begin{align}\label{eta4-1}
\eta_{\pvs,\psi}(\chi_s^{-1})
=\eta_{\pvs,\psi}*\chi_s^{-1}(1)
=\lim_{k\to \infty}
\int_{q^{-k}\leq |t|\leq q^{k}}
\eta_{\pvs,\psi}(t)\chi_{s}(t)\ud^*t.
\end{align}
By Lemma \ref{eta=} below, we get that
\begin{equation}\label{eta4-2}
\eta_{\pvs,\psi}(\chi_s^{-1})=\lim_{k\to \infty}
\int_{q^{-k}\leq |t|\leq q^{k}}
(\eta_{\pvs,\psi}*\one^{\vee}_{N})(t)\chi_{s}(t)\ud^*t.
\end{equation}
Since $\mathds{1}_{N}=\frac{1}{{\rm vol}(1+\varpi^N\CO,\mathrm{d}^* x)}\mathrm{ch}_{1+\varpi^N\CO}$, the normalized characteristic function of $1+\varpi^N\CO$, we have that $\one_N(mt^{-1})=\one_N(mt^{-1})|mt^{-1}|^{\frac{2n+1}{2}}$.
Then the convolution $\eta_{\pvs,\psi}*\one^{\vee}_{N}(t)$ can be calculated as follows.
\begin{align*}
\eta_{\pvs,\psi}*\one^{\vee}_{N}(t)
&= \int_{F^{\times}}
\eta_{\pvs,\psi}(m)\one_{N}(mt^{-1})\ud^*m\\
&=
\int_{F^{\times}}
\eta_{\pvs,\psi}(m)\one_{N}(mt^{-1})|mt^{-1}|^{\frac{2n+1}{2}}\ud^* m\\
&=
|t|^{-\frac{2n+1}{2}}
\int_{F^{\times}}
\eta_{\pvs,\psi}(m)
|m|^{\frac{2n+1}{2}}
\one_{N}(mt^{-1})
\ud^*m,
\end{align*}
which, from \eqref{eq:tilde-f-convolution} that we just proved, is equal to
\begin{align}\label{eta4-3}
|t|^{-\frac{2n+1}{2}}
\FL({\one}_{N})(t).
\end{align}
Therefore we obtain the following identity
\begin{equation}\label{eta4-4}
\eta_{\pvs,\psi}(\chi_s^{-1})
=\int_{F^{\times}}
\FL({\one}_{N})(t)|t|^{-\frac{2n+1}{2}}
\chi_{s}(t)
\ud^*t.
\end{equation}
By the functional equation \eqref{eq:FE-eta}, we obtain that
\begin{equation}\label{eta4-5}
\eta_{\pvs,\psi}(\chi_s^{-1})
=
\beta_{\psi}(\chi_{s}^{-1})
\int_{F^{\times}}
\one_{N}(t)|t|^{\frac{2n+1}{2}}
\chi_{s}^{-1}(t)\ud^*t = \beta_{\psi}(\chi^{-1}_{s}),
\end{equation}
because $\int_{F^{\times}}
\one_{N}(t)|t|^{\frac{2n+1}{2}}
\chi_{s}^{-1}(t)\ud^*t=1$ when $N>e(\chi)$.
Therefore, we have proved that
$$
\eta_{\pvs,\psi}(\chi_s)=\beta_{\psi}(\chi_{s})
$$
for all characters $\chi_s$.
\end{proof}

We are now coming back to prove the following formula that was used in the proof of Theorem \ref{thm:eta}.

\begin{lem}\label{eta=}
Given $\chi\in \wh{\CO^{\times}}$, if $N>e(\chi)$, then the following identity holds
\begin{align*}
\eta_{\pvs,\psi}(\chi_{s}^{-1})
=
\lim_{k\to \infty}
\int_{q^{-k}\leq |t|\leq q^{k}}
(\eta_{\pvs,\psi}*\one_{N}^{\vee})(t)\chi_{s}(t)\ud^*t.
\end{align*}
\end{lem}
\begin{proof}
For any $k\in \BN$, both the convolution $\eta_{\pvs,\psi}*\one^{\vee}_{N}$ and the integral
$$
\int_{q^{-k}\leq |t|\leq q^{k}}
(\eta_{\pvs,\psi}*\one^{\vee}_{N})(t)\chi_{s}(t)\ud^*t
$$
are absolutely convergent when $\Re(s)$ sufficiently positive. By the definition of the convolution $(\eta_{\pvs,\psi}*\one^{\vee}_{N})(t)$ and then by changing the variable, we write
\begin{align*}
&\int_{q^{-k}\leq |t|\leq q^{k}}
(\eta_{\pvs,\psi}*\one^{\vee}_{N})(t)\chi_{s}(t)\ud^*t
\\
&\qquad\qquad\qquad=
\int_{q^{-k}\leq |t|\leq q^{k}}
\int_{m\in F^{\times}}
\eta_{\pvs,\psi}(tm^{-1})\one_{N}(m)\ud^*m
\chi_{s}(t)\ud^*t,
\end{align*}
which is equal to
\begin{align}\label{eta4-7}
\frac{1}{\vol_N}
\int_{q^{-k}\leq |t|\leq q^{k}}
\int_{m\in (1+\vpi^{N}\CO)}
\eta_{\pvs,\psi}(tm^{-1})\ud^*m
\chi_{s}(t)\ud^*t,
\end{align}
where $\vol_N=\vol(1+\vpi^{N}\CO,\ud^*x)$.
By changing variable $t\to tm$, we deduce that \eqref{eta4-7} is equal to
\begin{align*}
&\frac{1}{\vol_N}
\int_{q^{-k}\leq |t|\leq q^{k}}
\int_{m\in 1+\vpi^{N}\CO}
\eta_{\pvs,\psi}(t)\ud^*m \chi_{s}(t)\chi_{s}(m)\ud^*t
\\
&\qquad\qquad=\int_{q^{-k}\leq |t|\leq q^{k}}
\eta_{\pvs,\psi}(t)\chi_{s}(t)\ud^*t
\int_{m\in F^{\times}}\one_{N}(m)\chi_{s}(m)\ud^*m\\
&\qquad\qquad\qquad\qquad=
\int_{q^{-k}\leq |t|\leq q^{k}}
\eta_{\pvs,\psi}(t)\chi_{s}(t)\ud^*t.
\end{align*}
The last identity holds because $N>e(\chi)$. Finally, by taking the limit with $k\to \infty$, we obtain the desired identity.
\end{proof}

\subsection{A Paley-Wiener type theorem for $\CS^{\pm}_\pvs(F^\times)$}\label{ssec-gcd}
To develop harmonic analysis over $X_{P_\Del}(F)$ in Section \ref{sec-etapvs-FT}, we need the asymptotic behavior of the image of the functions in $\CS^{\pm}_\pvs(F^\times)$ under the Mellin transform, which is a
version of the classical Paley-Wiener theorem for Fourier transform.

\begin{dfn}[Spaces $\CZ^{\pm}_{n,\bet}(\wh{F^\times})$]\label{defin:Mellin-gcd}
Denote $\CZ^{\pm}_{n,\bet}(\wh{F^\times})$ to be the subspace of $\CZ^{\pm}_{n}(\wh{F^\times})$ consisting of all complex-valued functions $Z_{\pm}(z,\chi)$ on $\wh{F^\times}$ with properties:
\begin{enumerate}
\item For every $\chi\in \widehat{\CO^{\times}}$, there exist constants $b^{\pm}_{0,\chi}$ and $b^{\pm}_{i,\pm,\chi}$  such that
	\begin{equation}\label{eq:Z-asymptotic-11}
	Z_+(z,\chi)- \frac{b_{0,\chi}^+}{1- z}-\sum^{n-1}_{i=0}\frac{b_{i,+,\chi}^+}{1-q^{-(i+\frac{1}{2})} z}+\frac{b_{i,-,\chi}^+}{1+q^{-(i+\frac{1}{2})} z}	
	\end{equation}
	and
	\begin{equation}\label{eq:Z-asymptotic-22}
	Z_{-}(z,\chi)- \frac{b_{0,\chi}^-}{1-q^{-n} z}-\sum^{n-1}_{i=0}\frac{b_{i,+,\chi}^-}{1-q^{-i} z}+\frac{b_{i,-,\chi}^-}{1+q^{-i} z}
	\end{equation}
	are polynomials in $z$ and $z^{-1}$ with complex coefficients; and
\item The coefficients satisfy the conditions:  $b^{\pm}_{0,\chi}= 0$ unless $\chi$ is the trivial character, and for each $i$, $b^{\pm}_{i,\pm,\chi}= 0$ unless $\chi^{2}$ is the trivial character.
\end{enumerate}
\end{dfn}

By Proposition \ref{Igusa78-53}, we may define the subspaces $\CS^{\pm}_{n,\bet}(F^{\times})$ of $\CS_n^{\pm}(F^\times)$ associated to $\CZ^{\pm}_{n,\bet}(\wh{F^\times})$
via the Mellin inversion $\CM^{-1}$. From Definition \ref{defin:Mellin-gcd}, we have the following short exact sequence for $\CS^{\pm}_{n,\bet}(F^{\times})$,
$$
0\to \CC^{\infty}_{c}(F^{\times})\to \CS^{\pm}_{n,\bet}(F^{\times})\to \FA_{\{ 0\}}^{\pm} \to 0
$$
where $\FA_{\{0\}}^{\pm}$ is a $1+2n\cdot |\CO^{\times}/\CO^{\times 2}| = 1+\frac{4n}{|2|}$ dimensional vector space over $\BC$ capturing the asymptotic behavior of functions in $\CS^{\pm}_{n,\bet}(F^{\times})$ whenever $|x|$ tends to $0$.

\begin{thm}[Paley-Wiener Theorem]\label{thm:PWM}
The spaces $\CS_\pvs^{\pm}(F^\times)$ as defined in Definition \ref{schwartzF} and spaces $\CS^{\pm}_{n,\bet}(F^{\times})$ as defined above share the following properties:
$$
\CS_\pvs^{+}(F^\times)=|\cdot|^{-2n}\CS^{+}_{n,\bet}(F^{\times});
$$
and
$$
\CS_\pvs^{-}(F^\times)=|\cdot|^{n+1}\CS^{-}_{n,\bet}(F^{\times}).
$$
\end{thm}

\begin{proof}
We only prove the equality between $|\cdot|^{2n}\CS^{+}_{\mathrm{pvs}}(F^{\times})$ and $\CS^{+}_{n,\bet}(F^{\times})$.

From Proposition \ref{Igusa78-53}, the Mellin transform $\CM(z,\chi)$ is an isomorphism from $\CS_n^+(F^\times)$ onto $\CZ_n^+(\wh{F^\times})$. By definition, the subspace $\CS^{+}_{n,\bet}(F^{\times})$ of
$\CS_n^+(F^\times)$ is the preimage of the subspace $\CZ^{+}_{n,\bet}(\wh{F^{\times}})$ of $\CZ_n^+(\wh{F^\times})$ with respect to the isomorphism $\CM(z,\chi)$.
Recall from Definition~\ref{defin:Mellin-gcd} that the subspace $\CZ^{+}_{n,\bet}(\wh{F^{\times}})$ consists of all complex-valued functions $\CZ_{+}(z,\chi)$ on $\wh{F^{\times}}$ with properties:
\begin{enumerate}
\item For every $\chi\in \widehat{\CO^{\times}}$, there exist constants $b^{+}_{0,\chi}$ and $b^{+}_{i,\pm,\chi}$ such that
$$
Z_{+}(z,\chi)-
\frac{b^{+}_{0,\chi}}{1-z}
-\sum_{i=0}^{n-1}
\frac{b^{+}_{i,+,\chi}}{1-q^{-(i+\frac{1}{2})}z}+
\frac{b^{+}_{i,-,\chi}}{1+q^{-(i+\frac{1}{2})}z}
$$
is a polynomial in $z$ and $z^{-1}$ with complex coefficients; and
\item
The coefficients satisfy the conditions: $b^{+}_{0,\chi}\neq 0$ only when $\chi$ is the trivial character, and for each $i$, $b^{+}_{i,\pm,\chi}\neq 0$ only when $\chi^{2}$ is the trivial character.
\end{enumerate}
It is equivalent to say that for a given $f\in \CS^{+}_{n,\bet}(F^{\times})$, the Mellin transform $\CM(f)(z,\chi)$ is absolutely convergent for $\Re(s)$ sufficiently positive, admits meromorphic continuation to $s\in \BC$, and
$$
\frac{\CM(f)(z,\chi)}{L(s,\chi)\prod_{i=0}^{n-1}L(2s+2i+1,\chi^{2})}
$$
is a polynomial in $z$ and $z^{-1}$. In particular, the set
$$
\{\CM(f)(z,\chi)\ |\ f\in \CS^{+}_{n,\bet}(F^{\times}) \}
$$
forms a fractional ideal
\[
L(s,\chi)\prod_{i=0}^{n-1}L(2s+2i+1,\chi^{2})\BC[q^{s},q^{-s}].
\]

On the other hand, from Theorem \ref{thm:gcd-zi} and the discussion in Section \ref{ssec-fife}, the following set
$$
\{\CM(f)(z,\chi)\ |\ f\in |\cdot|^{2n}\CS^{+}_{\mathrm{pvs}}(F^{\times})\}
$$
forms the same fractional ideal
\[
L(s,\chi)\prod_{i=0}^{n-1}L(2s+2i+1,\chi^{2})\BC[q^{s},q^{-s}].
\]
By the Mellin inversion (Proposition \ref{Igusa78-53}), we must have the equality of the two subspaces:
$$
|\cdot|^{2n}\CS^{+}_{\mathrm{pvs}}(F^{\times})
=\CS^{+}_{n,\bet}(F^{\times}).
$$
It is clear that the same proof works for $\CS_\pvs^{-}(F^\times)=|\cdot|^{n+1}\CS^{-}_{n,\bet}(F^{\times})$. We omit the details here.
\end{proof}

We end up this section with a discussion on the relation of the space $\CZ^{\pm}_{n,\bet}(\wh{F^\times})$ with the abelian $\gamma$-factor $\beta_\psi(\chi_s)$, in order to justify the notation with $\beta$.
Because the polynomials
$$
\{1-z, 1-q^{-(i+\frac{1}{2})}z, 1+q^{-(i+\frac{1}{2})}z\mid i=0,1,...,n-1 \}
$$
are coprime to each other, and so are
$$
\{ 1-q^{-n}z, 1-q^{-i}z, 1+q^{-i}z\mid i=0,1,...,n-1 \},
$$
by using the same arguments as in the proof of Proposition \ref{pro-2fi}, we deduce immediately the following result.

\begin{pro}\label{pro:gcd}
The spaces $\CZ^{\pm}_{n,\bet}(\wh{F^\times})$ consist of functions $Z_{\pm}(z,\chi)$ belonging to $\CZ^{\pm}_{n}(\wh{F^\times})$ with the properties:
$$
\frac{Z_{+}(z,\chi)}{L(s,\chi)\prod_{i=0}^{n-1}L(2s+2i+1,\chi^{2})}
$$
and
$$
\frac{Z_{-}(z,\chi)}{L(s+n,\chi)\prod_{i=0}^{n-1}L(2s+2i,\chi^{2})}
$$
are polynomials in $z=q^{-s}$ and $z^{-1}=q^s$ with complex coefficients. Here $L(s,\chi)$ is the standard $L$-factor defined via the corresponding Tate integral.
\end{pro}
In other words, for a fixed character $\chi$, the spaces $\CZ^{\pm}_{n,\bet}(\wh{F^{\times}})$ are the fractional ideals
$$
L(s,\chi)\prod_{i=0}^{n-1}L(2s+2i+1,\chi^{2})\cdot \BC[z,z^{-1}]
$$
and
$$
L(s+n,\chi)\prod_{i=0}^{n-1}L(2s+2i,\chi^{2})\cdot \BC[z,z^{-1}].
$$


\section{$\eta_{\pvs,\psi}$-Fourier Transform on $X_{P_\Del}$}\label{sec-etapvs-FT}


In this section, we are ready to define the $\eta_{\pvs,\psi}$-Fourier transform $\CF_{X,\psi}$ over $X_{P_\Del}(F)$ based on the long study in Sections \ref{sec-FEbeta} and \ref{sec-FTbeta}
on the generalized function $\eta_{\pvs,\psi}$ on $F^\times$ and the $\gamma$-function $\beta_\psi(\chi_s)$, in particular, in Theorem \ref{thm:LT} and Theorem \ref{thm:eta}.
The goal here is to establish a relation between the $\eta_{\pvs,\psi}$-Fourier transform $\CF_{X,\psi}$ and
the local intertwining operator $\beta_\psi(\chi_s)\cdot \RM_{w_\Del}(s,\chi)$, as given in \eqref{eq:IT-eta-rho}, in Theorem \ref{thm:CFP-M}. With the precise analytic information of
the local intertwining operator $\RM_{w_\Del}(s,\chi)$ as obtained in Proposition \ref{pro:pole-Mw}, we are able to develop basic results about
the  $\eta_{\pvs,\psi}$-Fourier transform $\CF_{X,\psi}$ over $X_{P_\Del}(F)$, including the extension of the $\eta_{\pvs,\psi}$-Fourier transform to a unitary operator on the space $L^2(X_{P_\Del})$ of
square-integrable functions on $X_{P_\Del}(F)$ (Proposition \ref{pro:Fx-norm} and Remark \ref{FXunitary}), a relation between the Schwartz functions on $X_{P_\Del}(F)$ and the good sections on $\Sp_{4n}(F)$
(Proposition \ref{pro:FTX}), and a characterization of Schwartz functions in $\CS_\pvs(X_{P_\Del})$ by their asymptotic behavior on the $M_\Del^\ab$-part (Theorem \ref{thm:asymSf}), where $M^\ab_\Del$ is defined in \eqref{M-Mab}.

Without lose of generality, we assume that $n>0$ in this section.


\subsection{Fourier transform and $L^2$-space on $X_{P_\Del}$}\label{ssec-FoL2}

Let $M^{\ab}_\Del(F)$ and $ \Sp_{4n}(F)$ act on  $X_{P_\Del}(F)$   via the left and right translations respectively.
The action of $M^{\ab}_\Del(F)$ on $\CC^\infty_c(X_{P_\Del})$, the space of smooth, compactly supported functions on $X_{P_\Del}(F)$, is defined to be the normalized left translation by means of the section $\Fs$:
\begin{equation}\label{La}
\Fl_a(f)(x) := f(\Fs_{a}^{-1}x) \del_{P_\Del}^{\frac{1}{2}}(\Fs_{a}),
\end{equation}
for any $a\in F^\times=\BG_m(F)$ and any $f\in \CC^\infty_c(X_{P_\Del})$, where the section $a\mapsto \Fs_a$ associated to the abelianization morphism $\Fa$ is defined in \eqref{section:Fs} and
$\delta_{P_\Del}$ is the modular character of $P_{\Del}$ as in \eqref{dPDel}. In particular, we have that
\begin{align}\label{deltaa}
\del_{P_\Del}(\Fs_a)=|\Fa(\Fs_{a})|^{2n+1}=|a|^{2n+1}.
\end{align}
It is easy to verify that the definition of the action in \eqref{La} is independent of the choice of the section $\Fs$. Note that we consider here the normalized left-translation following \cite{BK02}.

Define $\CP_{\chi_s}\ :\ \CC^{\infty}_{c}(X_{P_{\Del}}) \to {\rm I}(s,\chi)$ to be the projection:
\begin{equation}\label{proj-X-I}
\CP_{\chi_s}(f)(g):=\int_{F^\times}f(\Fs_a^{-1}g)|a|^{\frac{2n+1}{2}}\chi_s(a)\ud^*a,
\end{equation}
for all characters $\chi_s$ of $F^\times$ and with $f\in\CC^\infty_c(X_{P_\Del})$.
The integral is independent of the choice of the section $\Fs$ because of \eqref{La}, and converges absolutely for all $\chi_s$.
It is clear that the projection $\CP_{\chi_s}$ defines a surjective $\Sp_{4n}(F)$-equivariant linear morphism.
Recall the unnormalized intertwining operator
$\RM_{w_\Del}(s,\chi): {\rm I}(s,\chi) \to {\rm I}(-s,\chi^{-1})$ from \eqref{lio-0},
which is absolutely convergent for $\Re(s)$ sufficiently positive and has meromorphic continuation to $s\in \BC$.

\begin{dfn}[Fourier Transform over $X_{P_\Del}$]\label{dfn:FTX}
The $\eta_{\pvs,\psi}$-Fourier transform on $\CC^\infty_c(X_{P_\Del})$ is defined by
\begin{equation}
\CF_{X,\psi}(f)(g):=\int_{F^\times}^\pv\eta_{\pvs,\psi}(x)|x|^{-\frac{2n+1}{2}}\int_{N_\Del(F)}f(w_\Del n \Fs_x g)\ud n\ud^* x,
\end{equation}
for $f\in\CC^\infty_c(X_{P_\Del})$ and the generalized function $\eta_{\pvs,\psi}(x)$ defined in \eqref{def:eta-rho-psi}.
\end{dfn}

The issue of convergence of the integral in \eqref{FTX} will be addressed in Corollary \ref{FX-conv}. We first study the following Radon transform:
\begin{align}\label{Radon}
\RR_{X}(f)(g): =
\int_{N_\Del(F)}f(w_{\Del}ng)\ud n
\end{align}
for $f\in\CC^\infty_c(X_{P_\Del})$. By using the same argument as in the proof of \cite[Th. IV. 1.1]{Wal03}, one can easily prove that the integral defining the Radon transform $\RR_X$ is absolutely convergent.

\begin{lem}\label{lem:Rx}
Let $\RR_{X}$ be the Radon transform defined on $\CC^{\infty}_{c}(X_{P_{\Del}})$ via the absolutely convergent integral in \eqref{Radon}.  The following identity
$$
\CP_{\chi^{-1}_{s}}
\circ \RR_{X}(f) (g) =
\RM_{w_\Del}(s,\chi)\circ \CP_{\chi_{s}}(f)
$$
holds for any $f\in \CC^{\infty}_{c}(X_{P_{\Del}})$, after meromorphic continuation in $s\in\BC$.
Actually both sides of the identity are absolutely convergent for $\Re(s)$ sufficiently positive.
\end{lem}

\begin{proof}
For any $f\in \CC^{\infty}_{c}(X_{P_{\Del}})$, the projection  as defined in \eqref{proj-X-I}:
$$
\CP_{\chi_s}(f)(g):=\int_{F^\times}f(\Fs_a^{-1}g)|a|^{\frac{2n+1}{2}}\chi_s(a)\ud^*a,
$$
converges absolutely for any $s\in \BC$ and defines a holomorphic section belonging to $\RI(s,\chi)$. The composition $\RM_{w_\Del}(s,\chi)\circ \CP_{\chi_{s}}(f)(g)$ is given by
the double integral
\begin{align}\label{R-1}
\int_{N_{\Del}(F)}
\int_{F^{\times}}
f(\Fs_a^{-1}w_{\Del}ng)|a|^{\frac{2n+1}{2}}\chi_s(a)\ud^{*}a\ud n,
\end{align}
which converges absolutely for $\Re(s)$ sufficiently positive, according to \cite[Th. IV.1.1.]{Wal03}, for instance.
We may switch the order of the two integrations and obtain that \eqref{R-1} is equal to
\begin{align}\label{R-2}
\int_{F^{\times}}
|a|^{\frac{2n+1}{2}}\chi_s(a)
\int_{N_{\Del}(F)}f(\Fs_a^{-1}w_{\Del}ng)\ud n
\ud^{*}a.
\end{align}
Because
\[
f(\Fs_a^{-1}w_{\Del}ng)=f(w_{\Del}\Fs_ang)=f(w_{\Del}\Fs_an\Fs^{-1}_{a}\Fs_ag),
\]
by changing the variable $n\to \Fs_a^{-1}n\Fs_a$, we deduce that \eqref{R-2} is equal to
\begin{align}\label{R-3}
\int_{F^{\times}}
|a|^{-\frac{2n+1}{2}}\chi_s(a)
\int_{N_{\Del}(F)}f(w_{\Del}n\Fs_ag)\ud n \ud^{*}a.
\end{align}
By changing the variable $a\to a^{-1}$, it is clear that \eqref{R-3} is equal to
\begin{align}\label{R-4}
\int_{F^{\times}}
|a|^{\frac{2n+1}{2}}\chi_s(a)^{-1}
\int_{N_{\Del}(F)}f(w_{\Del}n\Fs_{a^{-1}}g)\ud n \ud^{*}a.
\end{align}
For $f\in \CC^{\infty}_{c}(X_{P_{\Del}})$, by a simple calculation according to the definition of the section $\Fs$, one finds that
\begin{align}\label{R-5}
f(w_{\Del}n\Fs_{a^{-1}}g)=f(w_{\Del}n\Fs_a^{-1}g).
\end{align}
Hence \eqref{R-4} is exactly equal to the composition
$\CP_{\chi^{-1}_{s}}\circ \RR_{X}(f)(g)$ when $\Re(s)$ is sufficiently positive. Therefore, for $\Re(s)$ sufficiently positive, we have the desired identity.
\end{proof}

Now we treat the convergence issue of the integral in (\ref{FTX}) that defines the Fourier transform $\CF_{X,\psi}$.

\begin{pro}\label{pro:Fg}
For any $f\in \CC^{\infty}_{c}(X_{P_{\Del}})$, and $a\in F^\times$, the function in $a$ as defined by
\begin{align}\label{53-1}
F_g(a):=|a|^{-(2n+1)}\RR_X(f)(\Fs_ag)
\end{align}
belongs to the space $\CS^+_\pvs(F^\times)$, and as functions in $a\in F^\times$, the following identity
\begin{align}\label{FX-Fg}
\FL(F_g)(a)=|a|^{2n+1}\CF_{X,\psi}(f)(\Fs_a^{-1}g)
\end{align}
holds, where the linear transform $\FL$ is defined in \eqref{LT}.
\end{pro}

\begin{proof}
We consider the Mellin transform of the function $F_g(a)$, and claim that
\begin{align}\label{53-2}
\CM(F_g)(s+\frac{2n+1}{2},\chi)=\CP_{\chi^{-1}_{s}}\circ\RR_{X}(f) (g).
\end{align}
In fact, for a fixed $g\in \Sp_{4n}(F)$ and an $f\in \CC^{\infty}_{c}(X_{P_{\Del}})$, when $\Re(s)$ is sufficiently positive, it is easy to deduce that
\begin{align*}
\CM(F_g)(s+\frac{2n+1}{2},\chi)
&=
\int_{F^{\times}}
\chi_{s}(a)|a|^{\frac{2n+1}{2}}
|a|^{-(2n+1)}
\RR_X(f)(\Fs_ag)\ud^{*}a\\
&=
\int_{F^{\times}}
\chi_{s}(a)|a|^{-\frac{2n+1}{2}}
\RR_X(f)(\Fs_ag)\ud^{*}a\\
&=
\int_{F^{\times}}
\chi^{-1}_{s}(a)|a|^{\frac{2n+1}{2}}
\RR_X(f)(\Fs_a^{-1}g)\ud^{*}a
\\
&=\CP_{\chi^{-1}_{s}}\circ\RR_{X}(f)(g).
\end{align*}
By Proposition \ref{pro:pole-Mw}, $a_{2n}(s,\chi)^{-1}\cdot \RM_{w_\Del}(s,\chi)$ is holomorphic. By using Lemma \ref{lem:Rx}, we obtain that
$$
a_{2n}(s,\chi)^{-1}\cdot
\CP_{\chi^{-1}_{s}}
\circ \RR_{X}(f)
$$
is a holomorphic section in $\RI(-s,\chi^{-1})$ for any $f\in \CC^{\infty}_{c}(X_{P_{\Del}})$. It follows that
$\CM(F_g)(s+\frac{2n+1}{2},\chi)$, with a fixed $g\in\Sp_{4n}(F)$, belongs to the fractional ideal $a_{2n}(s,\chi)\cdot\BC[q^{s},q^{-s}]$. Hence $\CM(F_g)(s,\chi)$ belongs to the fractional ideal
$$
a_{2n}(s-\frac{2n+1}{2},\chi)\cdot\BC[q^{s},q^{-s}].
$$
From the proof of Theorem \ref{thm:PWM}, it is easy to deduce that a function $h$ in the space $|\cdot|^{-2n}\CS_n^+(F^\times)$ belongs to the subspace $\CS^{+}_{\pvs}(F^{\times})$ if and only if
its Mellin transform $\CM(h)(s,\chi)$ belongs to the fractional ideal
$$
a_{2n}(s-\frac{2n+1}{2},\chi)\cdot\BC[q^{s},q^{-s}].
$$
Therefore, for any fixed $g\in\Sp_{4n}(F)$, as a function in $a\in F^{\times}$, the function $F_g(a)$ belongs to the space $\CS^{+}_{\pvs}(F^{\times})$.

Now we write $|t|^{2n+1}\CF_{X,\psi}(f)(\Fs_t^{-1}g)$ as
\begin{align}\label{53-3}
\int_{F^{\times}}^\pv
\eta_{\pvs,\psi}(x)|x|^{\frac{2n+1}{2}}
|xt^{-1}|^{-(2n+1)}\RR_X(\Fs_x\Fs_t^{-1}g)\ud^{*}x.
\end{align}
Hence it can be further written as
\[
\int_{F^{\times}}^\pv
\eta_{\pvs,\psi}(x)|x|^{\frac{2n+1}{2}}F_g(xt^{-1})\ud^{*}x=\FL(F_g)(t),
\]
according to Theorem \ref{thm:eta}. We are done.
\end{proof}

\begin{cor}\label{FX-conv}
For any $f\in \CC^{\infty}_{c}(X_{P_{\Del}})$, and $a\in F^\times$,
the function $|a|^{2n+1}\CF_{X,\psi}(f)(\Fs_a^{-1}g)$ belongs to the space $\CS^{-}_{\pvs}(F^{\times})$, as a function in $a\in F^{\times}$.
In particular the integral in \eqref{FTX} that defines the Fourier transform $\CF_{X,\psi}$ is convergent.
\end{cor}

\begin{proof}
By Proposition \ref{pro:Fg}, we have
$$
|a|^{2n+1}\CF_{X,\psi}(f)(\Fs_a^{-1}g)=\FL(F_g)(a)
$$
as functions in $a\in F^\times$. Since $F_g(a)$ belongs to the space $\CS^+_\pvs(F^\times)$, it follows that the function $|a|^{2n+1}\CF_{X,\psi}(f)(\Fs_a^{-1}g)$ belongs to the space $\CS^{-}_{\pvs}(F^{\times})$
because the definition of the linear transform $\FL$ in \eqref{LT} and Theorem \ref{thm:eta}. In particular, by taking $a=1$, we obtain that the integral in \eqref{FTX} that defines the Fourier transform $\CF_{X,\psi}$
is convergent.
\end{proof}

Now we are ready to establish the compatibility of the Fourier transform $\CF_{X,\psi}$ with intertwining operator $\RM_{w_\Del}(s,\chi)$.

\begin{thm}\label{thm:CFP-M}
For any $f\in  \CC^\infty_c(X_{P_{\Del}})$, the composition $\CP_{\chi^{-1}_s}\circ\CF_{X,\psi}(f)(g)$ converges absolutely for $\Re(s)$ sufficiently negative, and the following identity
\begin{equation}\label{eq:wtF-Inter}
	(\CP_{\chi^{-1}_s}\circ\CF_{X,\psi})(f)(g)=\beta_\psi(\chi_s)(\RM_{w_\Del}(s,\chi)\circ\CP_{\chi_s})(f)(g)
\end{equation}
holds as meromorphic functions in $s\in\BC$, with $g\in\Sp_{4n}(F)$.
\end{thm}

\begin{proof}
First, we show that $\CP_{\chi^{-1}_{s}}\circ \CF_{X,\psi}(f)(g)$ is absolutely convergent for $\Re(s)$ sufficiently negative, for any $f\in  \CC^\infty_c(X_{P_{\Del}})$.

By \eqref{FTX}, the Fourier transform $\CF_{X,\psi}(f)$ is defined as
\begin{align*}
\CF_{X,\psi}(f)(\Fs_a^{-1}g):=\int_{F^\times}^\pv\eta_{\pvs,\psi}(x)|x|^{-\frac{2n+1}{2}}\int_{N_\Del(F)}f(w_\Del n \Fs_x \Fs_a^{-1}g)\ud n\ud^* x.
\end{align*}
By Proposition \ref{pro:Fg} and Corollary \ref{FX-conv}, for any fixed $g\in\Sp_{4n}(F)$, the function $\FL(F_g)(a)=|a|^{2n+1}\CF_{X,\psi}(f)(\Fs_a^{-1}g)$,
as a function in $a\in F^{\times}$, belongs to the space  $\CS^{-}_{\pvs}(F^{\times})\subset|\cdot|^{n+1}\CS_n^-(F^\times)$.
On the other hand,  as in (\ref{proj-X-I}), we have that
\begin{align}\label{54-1}
\CP_{\chi^{-1}_{s}}(\CF_{X,\psi}(f))(g)=\int_{F^{\times}}
\chi_{s}^{-1}(a)|a|^{\frac{2n+1}{2}}\CF_{X,\psi}(f)(\Fs^{-1}_{a}g)\ud^{*} a,
\end{align}
which is the Mellin transform of the function $|a|^{2n+1}\CF_{X,\psi}(f)(\Fs_a^{-1}g)$, as a function in $a$, evaluated at $\chi^{-1}_{s}$, and hence
is absolutely convergent for $\Re(s)$ sufficiently negative by Proposition \ref{Igusa78-53}. This verifies the first part of the theorem.

Next we prove the identity in \eqref{eq:wtF-Inter}. The left-hand side of \eqref{eq:wtF-Inter}, $(\CP_{\chi^{-1}_{s}}\circ\CF_{X,\psi})(f)(g)$, via \eqref{54-1}, can be written as
\begin{align}\label{54-2}
(\CP_{\chi^{-1}_{s}}\circ\CF_{X,\psi})(f)(g)=&\int_{F^{\times}}
\chi_{s}^{-1}(a)|a|^{-\frac{2n+1}{2}}
|a|^{2n+1}
\CF_{X,\psi}(f)
(\Fs_{a}^{-1}g)\ud^{*} a\nonumber\\
&
=
\int_{F^{\times}}
\chi_{s}^{-1}(a)|a|^{-\frac{2n+1}{2}}\FL(F_{g})(a)\ud^* a,
\end{align}
according to Proposition \ref{pro:Fg}. Then by the functional equation in Theorem \ref{thm:LT}, we deduce that
\begin{align}\label{54-3-0}
(\CP_{\chi^{-1}_{s}}\circ\CF_{X,\psi})(f)(g)=
&
\bet_{\psi}(\chi_{s})
\int_{F^{\times}}
\chi_{s}(a)|a|^{\frac{2n+1}{2}}
F_{g}(a)\ud^{*}a.
\end{align}
By Theorem \ref{thm:LT}, the left-hand side converges absolutely for $\Re(s)$ sufficiently negative, while the right-hand side converges absolutely for $\Re(s)$ sufficiently positive, and the identity
holds by meromorphic continuation as functions in $s\in\BC$.

By Proposition \ref{pro:Fg}, we have that $F_g(a)=|a|^{-(2n+1)}\RR_X(f)(\Fs_ag)$. Whenever $\Re(s)$ is sufficiently positive, we may write the integral in the right-hand side of \eqref{54-3-0} as
\begin{align}\label{54-3}
\int_{F^{\times}}
\chi_{s}(a)|a|^{\frac{2n+1}{2}}
F_{g}(a)\ud^{*}a
&=
\int_{F^{\times}}
\chi_{s}(a)|a|^{-\frac{2n+1}{2}}\RR_X(f)(\Fs_ag)\ud^{*} a\nonumber\\
&=
\int_{F^{\times}}
\chi_{s}^{-1}(a)|a|^{\frac{2n+1}{2}}\RR_X(f)(\Fs_a^{-1}g)\ud^{*} a\nonumber\\
&=
\CP_{\chi^{-1}_{s}}\circ \mathrm{R}_{X}(f)(g).
\end{align}
By Lemma \ref{lem:Rx}, it is equal to $\RM_{w_\Del}(s,\chi)\circ\CP_{\chi_s}(f)(g)$. Therefore we obtain the identity:
$$
(\CP_{\chi^{-1}_s}\circ\CF_{X,\psi})(f)(g)=\beta_\psi(\chi_s)(\RM_{w_\Del}(s,\chi)\circ\CP_{\chi_s})(f)(g),
$$
which holds as meromorphic functions in $s\in\BC$.
We are done.
\end{proof}

We are going to show that the $\eta_{\pvs,\psi}$-Fourier transform $\CF_{X,\psi}$, extends to a unitary operator from $\CC^{\infty}_{c}(X_{P_{\Del}})$ to $L^2(X_{{P_\Del}})$, the space of square-integrable functions
on $X_{P_\Del}(F)$. Here the inner product on $L^2(X_{{P_\Del}})$ is defined by
\begin{equation}\label{eq:L2onXp}
\apair{f_1,f_2}_{X_{P_{\Del}}}
=\int_{K_\Del}\int_{F^\times}f_1(\Fs_{x}k)\ol{f_2(\Fs_{x}k)}\delta_{P_{\Del}}(\Fs_{x})^{-1}\ud^{*} x\ud k	
\end{equation}
for $f_1,f_2\in L^{2}(X_{P_\Del})$. Here we use the measure ${\rm d}^{*} x\ud k$ on $X_{P_\Del}(F)$, which is different from the induced measure from ${\rm d} g$ on $\Sp_{4n}(F)$ by a nonzero constant.

\begin{pro}\label{pro:Fx-norm}
For $f\in\CC_c^\infty(X_{P_\Del})$, the function $\CF_{X,\psi}(f)$ is square-integrable on $X_{P_\Del}(F)$ and
the following identity
$$\|\CF_{X,\psi}(f)\|=c\|f\|$$
holds, where $\|\cdot\|$ is the $L^2$-norm on $L^2(X_{{P_\Del}})$ and the constant $c=|2|^{-n(2n+1)}$.
\end{pro}

\begin{proof}
For a given unitary character $\chi$,
define the Hermitian bilinear pairing between $\RI(s,\chi)$ and $\RI(-\bar{s},\chi)$ by
\begin{align}\label{hbp}
\apair{f_{\chi,s},f_{\chi,{-\ol{s}}}}_{K_\Del}:=\int_{K_\Del}f_{\chi,s}(k)\overline{f_{\chi,{-\ol{s}}}(k)}\ud k
\end{align}
with $f_{\chi,s}\in \RI(s,\chi)$ and $f_{\chi,{-\ol{s}}}\in \RI(-\bar{s},\chi)$. It follows that for $f_{\chi,s}\in \RI(s,\chi)$, its image under the intertwining operator $\RM_{w_\Del}(s,\chi)$ belongs to
$\RI(-s,\chi^{-1})$, which is paired with $\RI(\ol{s},\chi^{-1})$, according to \eqref{hbp}. By \cite[p.287]{Wal03},
the adjoint operator $\RM_{w_\Del}(s,\chi)^*$ of $\RM_{w_\Del}(s,\chi)$ is the intertwining operator $\RM_{w_{\Del}^{-1}}(\bar{s},\chi^{-1})$ with the identity
\begin{align}\label{56-1}
\apair{ \RM_{w_\Del}(s,\chi)(f_{\chi,s}),f_{\chi^{-1},\ol{s}} }_{K_\Del}&=\apair{ f_{\chi,s},\RM_{w_\Del}(s,\chi)^*(f_{\chi^{-1},\ol{s}}) }_{K_\Del}\nonumber\\
&=\apair{ f_{\chi,s},\RM_{w_{\Del}}(\bar{s},\chi^{-1})(f_{\chi^{-1},\ol{s}}) }_{K_\Del},
\end{align}
for $f_{\chi,s}\in \RI(s,\chi)$ and $f_{\chi^{-1},\ol{s}}\in\RI(\bar{s},\chi^{-1})$.  Note that $w_{\Del}=w_\Del^{-1}$.

By \eqref{eq:IT-eta-rho} and simple calculations, we can deduce that the adjoint operator of the normalized intertwining operator $\RM^{\dagger}_{w_{\Del}}(s,\chi,\psi)$ is
\begin{equation}\label{56-2}
\RM_{w_{\Del}}^{\dagger}(s,\chi,\psi)^*=
\ol{\eta_{\pvs,\psi}(\chi_s)}\cdot\eta_{\pvs,\psi}(\chi^{-1}|\cdot|^{\bar{s}})^{-1}	\cdot\RM^{\dagger}_{w_{\Del}}(\bar{s},\chi^{-1},\psi).
\end{equation}
From \cite[Corollary~2,~p.142]{BH06}, one deduces that
$$
\ol{\varepsilon(s,\chi,\psi)}=\chi(-1)\varepsilon(\bar{s},\chi^{-1},\psi).
$$
By Theorem \ref{thm:eta}, together with the expression in \eqref{eq:eta-intertwining}, we obtain that the following identity:
\[
\ol{\eta_{\pvs,\psi}(\chi_s)}=\chi(-1)\eta_{\pvs,\psi}(\chi^{-1}|\cdot|^{\bar{s}})
\]
holds for any unitary character $\chi$. It follows that
\begin{equation}\label{56-3}
\RM^{\dagger}_{w_{\Del}}(s,\chi,\psi)^*=\chi(-1)	\RM_{w_{\Del}}^{\dagger}(\bar{s},\chi^{-1},\psi).
\end{equation}

Recall from Definition \ref{defn:good-f} the space $\RI^\dagger (s,\chi)$ of good sections. By \cite[Proposition~3.1]{Y14},
for holomorphic sections $f_{\chi,s}\in \RI(s,\chi)$ and $f_{\chi,-\ol{s}}\in \RI(-\bar{s},\chi)$, $\RM^{\dagger}_{w_{\Del}}(s,\chi,\psi)(f_{\chi,s})$
and $\RM^{\dagger}_{w_{\Del}}(-\bar{s},\chi,\psi)(f_{\chi,-\ol{s}})$ are good sections in $\RI^\dagger (-s,\chi^{-1})$  and $\RI^\dagger(\bar{s},\chi^{-1})$, respectively.
We consider the Hermitian bilinear pairing given by
\begin{equation}\label{56-4}
\apair{\RM_{w_{\Del}}^{\dagger}(s,\chi,\psi)(f_{\chi,s}),\RM_{w_{\Del}}^{\dagger}(-\bar{s},\chi,\psi)(f_{\chi,-\wb{s}})}_{K_\Del}.	
\end{equation}
By\eqref{56-1}, the Hermitian inner product in \eqref{56-4} is equal to
\begin{align}\label{56-5}
\apair{f_{\chi,s},\RM_{w_{\Del}}^{\dagger}(s,\chi,\psi)^*\circ\RM_{w_{\Del}}^{\dagger}(-\bar{s},\chi,\psi)(f_{\chi,-\wb{s}})}_{K_\Del},
\end{align}
which, by \eqref{56-3},  is equal to
\begin{align}\label{56-6}
\apair{f_{\chi,s},\chi(-1)	\RM_{w_{\Del}}^{\dagger}(\bar{s},\chi^{-1},\psi)\circ\RM_{w_{\Del}}^{\dagger}(-\bar{s},\chi,\psi)(f_{\chi,-\wb{s}})}_{K_\Del}.
\end{align}
By \eqref{eq:M-FE} and the fact from \cite[Lemma~1,~p.143]{BH06} that $\veps(s,\chi,\psi) = \chi(-1)\veps(s,\chi,\psi^{-1})$,
we obtain that
$$
\RM_{w_{\Del}}^{\dagger}(\bar{s},\chi^{-1},\psi)\circ \RM_{w_{\Del}}^{\dagger}(-\bar{s},\chi,\psi)=\chi(-1).
$$
Hence \eqref{56-6} is equal to
\begin{align}\label{56-7}
\apair{f_{\chi,s},\chi(-1)^2f_{\chi,-\wb{s}}}_{K_\Del}=\apair{f_{\chi,s},f_{\chi,-\wb{s}}}_{K_\Del}.
\end{align}
Therefore we obtain the following equality:
\begin{equation}\label{eq:M=f}
\apair{\RM_{w_{\Del}}^{\dagger}(s,\chi,\psi)(f_{\chi,s}),\RM_{w_{\Del}}^{\dagger}(-\bar{s},\chi,\psi)(f_{\chi,-\wb{s}})}_{K_\Del}
=\apair{f_{\chi,s},f_{\chi,-\wb{s}}}_{K_\Del}.
\end{equation}

Now we are going to use the identity in \eqref{eq:M=f} to finish the proof of the proposition.

For any $f,f'\in\CC^\infty_c(X_{P_\Del})$, take $f_{\chi,s}=\CP_{\chi_s}(f)$ and $f'_{\chi,-\wb{s}}=\CP_{\chi_{-\bar{s}}}(f')$,
which are holomorphic sections in $\RI(s,\chi)$ and $\RI(-\bar{s},\chi)$, respectively.
Consider the Hermitian inner product of $f_{\chi,s}$ and $f'_{\chi,-\wb{s}}$
\begin{align}\label{56-8}
\apair{f_{\chi,s},f'_{\chi,-\wb{s}}}_{K_\Del}
&=
\int_{K_\Del}
\int_{F^{\times}}
\int_{F^{\times}}
\chi_{s}(a)|a|^{\frac{2n+1}{2}}
f(\Fs^{-1}_{a}k)
\nonumber
\\
&\qquad\qquad\qquad\qquad
\wb{\chi_{-\wb{s}}(x)}|x|^{\frac{2n+1}{2}}
\wb{f'(\Fs^{-1}_{x}k)}
\ud^{*} a \ud^{*} x \ud k.
\end{align}
Since $\wb{\chi_{-\wb{s}}(x)} = \chi_{s}(x)^{-1}$, by changing the variable $x\to xa$, we get that the right-hand side of \eqref{56-8} is equal to
\begin{align}\label{56-9}
\int_{F^\times} \chi_{s}(x)^{-1}|x|^{\frac{2n+1}{2}}
\int_{K_\Del}\int_{F^\times}
f(\Fs_{a}^{-1}k)\ol{f'(\Fs_{x}^{-1}\Fs_{a}^{-1}k)}|a|^{2n+1}\ud^{*} a\ud k\ud^{*} x.
\end{align}
By \eqref{La} and \eqref{deltaa}, for $x\in F^\times$, we write $\Fl_a(f')(g)=|x|^{\frac{2n+1}{2}}f'(\Fs_x^{-1}g)$, and hence we write
\begin{align*}
&|x|^{\frac{2n+1}{2}}
\int_{K_\Del}\int_{F^\times}
f(\Fs_{a}^{-1}k)\ol{f'(\Fs_{x}^{-1}\Fs_{a}^{-1}k)}|a|^{2n+1}\ud^{*} a\ud k
\\
&\qquad\qquad=
\int_{K_\Del}\int_{F^\times}
f(\Fs_{a}^{-1}k)\ol{\Fl_x(f')(\Fs_{a}^{-1}k)}|a|^{2n+1}\ud^{*} a\ud k
\\
&\qquad\qquad=
\int_{K_\Del}\int_{F^\times}
f(\Fs_{a}k)\ol{\Fl_x(f')(\Fs_{a}k)}|a|^{-(2n+1)}\ud^{*} a\ud k=
\apair{f,\Fl_xf'}_{X_{P_\Del}},
\end{align*}
because $\del_{P_{\Del}}(\Fs_{a})=|a|^{2n+1}$ and definition of the hermitian inner product $\apair{f_1,f_2}_{X_{P_{\Del}}}$ in \eqref{eq:L2onXp}.
Therefore we arrive at the following identity
\begin{align}
\apair{f_{\chi,s},f'_{\chi,-\wb{s}}}_{K_\Del}= \int_{F^\times} \chi_{s}(x)^{-1}\apair{f,\Fl_xf'}_{X_{P_\Del}}\ud^{*} x.
\label{eq:M=f-1}
\end{align}
Note that all the integrations above are absolutely convergent for any $s\in \BC$.

On the other hand,
by \eqref{eq:wtF-Inter} and \eqref{eq:IT-eta-rho}, we have
\begin{align*}
 & \apair{\RM_{w_{\Del}}^{\dagger}(s,\chi,\psi)(f_{\chi,s}),\RM_{w_{\Del}}^{\dagger}(-\bar{s},\chi,\psi)(f'_{\chi,-\wb{s}})}_{K_\Del}\\
&\qquad\qquad\qquad=c^{-2}\apair{\CP_{\chi_s}\circ\CF_{X,\psi}(f),\CP_{\chi_{-\bar{s}}}\circ\CF_{X,\psi}(f')}_{K_\Del}.
\end{align*}

From Definition \ref{CS-}, for any $f_{-}\in \CS^{-}_{n}(F^{\times})$, $\supp(f_{-})$ is bounded on $F^{\times}$, and for $|x|\ll 1$, there exist locally constant functions $a^{-}_{0}(\ac(x))$ and $a^{-}_{i,\pm}(\ac(x))$ on $\CO^{\times}$ such that
$$
f_{-}(x) =
a^{-}_{0}(\ac(x))
|x|^{n}+\sum_{i=0}^{n-1}
(
(a^{-}_{i,+}(\ac(x)))|x|^{i}
+a^{-}_{i,-}(\ac(x))|x|^{i}(-1)^{\ord(x)}).
$$
Therefore for any $f_{-}\in \CS^{-}_{n}(F^{\times})$, the Mellin transform
$$
\CM(f_{-})(s,\chi)
=\int_{F^{\times}}
\chi_{s}(x)f_{-}(x)\ud^{*}x
$$
is absolutely convergent for $\Re(s)>0$.

From Corollary \ref{FX-conv}, for any $\phi\in \CC^{\infty}_{c}(X_{P_{\Del}}(F))$, as a function of $a\in F^{\times}$, $|a|^{2n+1}\CF_{X,\psi}(\phi)(\Fs^{-1}_{a}g)$ lies in $\CS^{-}_{\pvs}(F^{\times})\subset |\cdot|^{n+1}\CS^{-}_{n}(F^{\times})$. Therefore
the following integral formula
$$
\CP_{\chi_{s}}\circ \CF_{X,\psi}(\phi) =
\int_{F^{\times}}
\chi_{s}(a)|a|^{-\frac{2n+1}{2}}
|a|^{2n+1}
\CF_{X,\psi}(\phi)(\Fs_{a}^{-1}g)\ud^{*}a
$$
is absolutely convergent for $\Re(s)>\frac{2n+1}{2}-(n+1)=-\frac{1}{2}$. Similarly the integral defining $\CP_{\chi^{-1}_{\wb{s}}}\circ \CF_{X,\psi}(f^{\p})$ is absolutely convergent for $\Re(s)<\frac{1}{2}$.
Hence whenever $-\frac{1}{2}<\Re(s)<\frac{1}{2}$, we have the following absolutely convergent integral
\begin{align*}
&\apair{\RM_{w_{\Del}}^{\dagger}(s,\chi,\psi)(f_{\chi,s}),\RM_{w_{\Del}}^{\dagger}(-\bar{s},\chi,\psi)(f'_{\chi,-\wb{s}})}_{K_\Del}
\\
&\qquad\qquad=c^{-2}\apair{\CP_{\chi_s}\circ\CF_{X,\psi}(f),\CP_{\chi_{-\bar{s}}}\circ\CF_{X,\psi}(f')}_{K_\Del}\\
&\qquad\qquad=c^{-2}\int_{K_\Del}\int_{F^\times}
\chi_s(a)|a|^{\frac{2n+1}{2}}
\CF_{X,\psi}(f)(\Fs_{a}^{-1}k)\ud^{*} a\\
&\qquad\qquad\qquad\qquad\cdot \int_{F^\times} \ol{\chi_{-\bar{s}}(x)}|x|^{\frac{2n+1}{2}}
\ol{\CF_{X,\psi}(f')(\Fs_{x}^{-1}k)}
\ud^{*} x\ud k.
\end{align*}
With the same argument as the proof of the identity \eqref{eq:M=f-1}, we obtain the following equality whenever $-\frac{1}{2}<\Re(s)<\frac{1}{2}$,
\begin{align*}
 & c^2\apair{\RM^{\dagger}_{w_{\Del}}(s,\chi,\psi)(f_{\chi,s}),\RM^{\dagger}_{w_{\Del}}(-\bar{s},\chi,\psi)(f_{\chi,-\wb{s}})}_{K_\Del}\\
&\qquad\qquad=\int_{F^\times}\chi_s(x)^{-1}
\apair{\CF_{X,\psi}(f),\Fl_x\CF_{X,\psi}(f')}_{X_{P_\Del}}
\ud^{*} x.
\end{align*}

Combining with \eqref{eq:M=f-1}, we obtain that for any unitary character $\chi$ and $-\frac{1}{2}<\Re(s)<\frac{1}{2}$,
\begin{align}
&\int_{F^\times}\chi_s(x)^{-1}
\apair{\CF_{X,\psi}(f),\Fl_x\CF_{X,\psi}(f')}_{X_{P_\Del}}
\ud^{*} x\nonumber \\
&\qquad\qquad\qquad=c^2\int_{F^\times} \chi_{s}(x)^{-1}\apair{f,\Fl_xf'}_{X_{P_\Del}}\ud^{*} x.
 \label{eq:I=f} 	
\end{align}
Now both sides of \eqref{eq:I=f} admit meromorphic continuation to $s\in \BC$.
Therefore by the Mellin inversion formula, we have
\[
\apair{\CF_{X,\psi}(f), \CF_{X,\psi}(f')}_{X_{P_\Del}}=c^{2}\apair{f, f'}_{X_{P_\Del}}.
\]
This completes the proof of this proposition.
\end{proof}

It is natural to have the following
\begin{cor}\label{Fx-invs}
The Fourier operators enjoy the following property:
$$
\CF_{X,\psi^{-1}}\circ\CF_{X,\psi}=c^{2}\cdot \Id
$$
as operators on $L^{2}(X_{P_{\Del}})$, where $c=|2|^{-n(2n+1)}$.
\end{cor}

\begin{proof}
From \eqref{eq:M-FE} we have that
$$
\RM^{\dag}_{w_{\Del}}(-s,\chi^{-1},\psi^{-1})
\circ \RM^{\dag}_{w_{\Del}}(s,\chi,\psi) = \Id.
$$
By using \eqref{eq:IT-eta-rho}, we have that
$$
\bet_{\psi^{-1}}(\chi_{s}^{-1})
\RM_{w_{\Del}}(-s,\chi^{-1})
\circ \bet_{\psi}(\chi_{s})\RM_{w_{\Del}}(s,\chi)
=c^{2}\cdot \Id.
$$
For any $f\in \CC^{\infty}_{c}(X_{P_{\Del}})$, Theorem \ref{thm:CFP-M} implies that
\begin{align*}
&\bet_{\psi^{-1}}(\chi_{s}^{-1})
\RM_{w_{\Del}}(-s,\chi^{-1})
\circ \bet_{\psi}(\chi_{s})\RM_{w_{\Del}}(s,\chi)
\circ \CP_{\chi_{s}}(f)
\\
&\qquad\qquad=\bet_{\psi^{-1}}(\chi_{s}^{-1})
\RM_{w_{\Del}}(-s,\chi^{-1})
\circ \CP_{\chi_{s}^{-1}}\circ \CF_{X,\psi}(f)
\\
&\qquad\qquad=
\CP_{\chi_{s}}
\circ \CF_{X,\psi^{-1}}
\circ \CF_{X,\psi}(f).
\end{align*}
Note that $\CF_{X,\psi^{-1}}\circ \CF_{X,\psi}(f)$ is defined by Proposition \ref{pro:Fx-norm}.
Hence
$$
\CP_{\chi_{s}}
\circ \CF_{X,\psi^{-1}}
\circ \CF_{X,\psi}(f)=
c^{2}\cdot \CP_{\chi_{s}}(f).
$$
Applying the Mellin inversion formula on both sides we get
$$
\CF_{X,\psi^{-1}}\circ \CF_{X,\psi} = c^{2}\cdot \Id
$$
on $\CC^{\infty}_{c}(X_{P_{\Del}})$, which is a dense subspace of $L^{2}(X_{P_{\Del}})$. We are done.
\end{proof}

\begin{rmk}\label{FXunitary}
If we re-scale the Fourier transform $\CF_{X,\psi}$ and define it to be $c^{-1}\cdot\CF_{X,\psi}$ with $c=|2|^{-n(2n+1)}$,
it is clear that the new Fourier transform $c^{-1}\cdot\CF_{X,\psi}$ extends a unitary operator of the space $L^2(X_{P_\Del})$ and enjoys the property:
$(c^{-1}\CF_{X,\psi^{-1}})\circ(c^{-1}\CF_{X,\psi}) = \Id$. However, such a normalization is not necessary for the rest of this paper.
\end{rmk}

\subsection{Fourier transform and Schwartz space on $X_{P_\Del}$}\label{ssec-FoSs}

With the Fourier transform $\CF_{X,\psi}$ as studied in Section \ref{ssec-FoL2}, we are able to define the Schwartz space on $X_{P_\Del}(F)$ as suggested in \cite{BK02}.

\begin{dfn}\label{dfn:SSFX}
The Schwartz space on $X_{P_\Del}(F)$ is defined to be
\[
\CS_\pvs(X_{P_{\Del}})=\CC^\infty_c(X_{P_\Del})+\CF_{X,\psi}(\CC^\infty_c(X_{P_\Del})).
\]
\end{dfn}
This space of Schwartz functions on $X_{P_\Del}(F)$ enjoys the following properties, which give a more conceptual understanding of the {\sl good sections} associated to the normalized induced representation $\RI(s,\chi)$
of $\Sp_{4n}(F)$ (Definition \ref{defn:good-f}).

\begin{pro}\label{pro:FTX}
The Schwartz space $\CS_{\pvs}(X_{P_{\Del}})$ is stable under the action of the Fourier transform $\CF_{X,\psi}$, and can recover the space of good sections $\RI^\dagger(s,\chi)$ via the projection $\CP_{\chi_s}$, i.e.
the projection $\CP_{\chi_s}$ is surjective from $\CS_{\pvs}(X_{P_{\Del}})$ to $\RI^\dagger(s,\chi)$, for all characters $\chi_s$.
\end{pro}

\begin{proof}
From \cite[Proposition~3.1.4~(c)]{Y14}, for any good section $f_{\chi_{s}}\in \RI^{\dag}(s,\chi)$, there exist holomorphic sections $f_{1,\chi_s}\in \RI(s,\chi)$ and $f_{2,\chi_s^{-1}}\in \RI(-s,\chi^{-1})$ respectively, such that
$$
f_{\chi_{s}} = f_{1,\chi_s}+\RM^{\dag}_{w_{\Del}}(-s,\chi^{-1},\psi)(f_{2,\chi_s^{-1}}).
$$
Since $\CP_{\chi_{s}}$ is surjective from $\CC^{\infty}_{c}(X_{P_{\Del}})$ to the space of holomorphic sections in $\RI(s,\chi)$, there exist $f_{1},f_{2}\in \CC^{\infty}_{c}(X_{P_{\Del}})$, such that
$f_{1,\chi_s} = \CP_{\chi_{s}}(f_{1})$ and $f_{2,\chi_s^{-1}} = \CP_{\chi^{-1}_{s}}(f_{2})$. Therefore
$$
f_{\chi_{s}} =
\CP_{\chi_{s}}
(f_{1})+\RM^{\dag}_{w_{\Del}}(-s,\chi^{-1},\psi)\circ \CP_{\chi^{-1}_{s}}(f_{2}),
$$
which, from \eqref{eq:wtF-Inter} and \eqref{eq:M-FE}, can be written as
$$
f_{\chi_{s}}=\CP_{\chi_{s}}
(f_{1})+\CP_{\chi_{s}}\circ \CF_{X,\psi}(f_{3})
=\CP_{\chi_{s}}(f_1+\CF_{X,\psi}(f_{3})),
$$
for some $f_3\in\CC^{\infty}_{c}(X_{P_{\Del}})$.
Combining with Definition \ref{dfn:SSFX} we obtain that the projection $\CP_{\chi_s}$ is surjective from $\CS_{\pvs}(X_{P_{\Del}})$ to $\RI^\dagger(s,\chi)$, for all characters $\chi_s$.

Now we are going to show that the Schwartz space $\CS_{\pvs}(X_{P_{\Del}})$ is stable under the action of the Fourier transform $\CF_{X,\psi}$. This is done by using Corollary \ref{Fx-invs}.
We first find the relation between $\CF_{X,\psi}$ and $\CF_{X,\psi^{-1}}$. By Theorem \ref{thm:CFP-M}, we have
\begin{align}\label{59-1}
\CP_{\chi^{-1}_{s}}\circ \CF_{X,\psi^{-1}}(f)(g)
=
\bet_{\psi^{-1}}(\chi_{s})
(\RM_{w_{\Del}}(s,\chi)\circ \CP_{\chi_{s}})(f)(g),
\end{align}
for any $f\in\CS_{\pvs}(X_{P_{\Del}})$.
By \cite[Lemma~1,~p.143]{BH06}, one has that $\veps(s,\chi,\psi) = \chi(-1)\veps(s,\chi,\psi^{-1})$. Hence we obtain that $\bet_{\psi^{-1}}(\chi_{s})=\chi(-1)\bet_{\psi}(\chi_{s})$. It follows that
the right-hand side of \eqref{59-1} is equal to $\chi(-1)\bet_{\psi}(\chi_{s})(\RM_{w_{\Del}}(s,\chi)\circ \CP_{\chi_{s}})(f)(g)$, which, by Theorem \ref{thm:CFP-M} again, is equal to
$\chi(-1)\CP_{\chi^{-1}_{s}}\circ \CF_{X,\psi}(f)(g)$. From the definition of the projection $\CP_{\chi_{s}}$, we must have
\begin{align}\label{59-2}
\chi(-1)
\CP_{\chi^{-1}_{s}}\circ \CF_{X,\psi}(f)(g)
=\CP_{\chi^{-1}_{s}}\circ \CF_{X,\psi}(f)(\Fs_{-1} g)
\end{align}
for any $f\in \CC^{\infty}_{c}(X_{P_{\Del}})$. Note that the homeomorphism $\iota:g\to \Fs_{-1}\cdot g$ preserves $\CC^{\infty}_{c}(X_{P_{\Del}})$.
From Definition \ref{dfn:FTX} the identity $\CF_{X,\psi}(\iota\circ f)(g) = \iota\circ \CF_{X,\psi}(f)(g)$ holds for any $f\in \CC^{\infty}_{c}(X_{P_{\Del}})$. Hence we obtain the following identity:
\begin{align*} 
\CP_{\chi^{-1}_{s}}\circ \CF_{X,\psi^{-1}}(f)(g)
=
\CP_{\chi^{-1}_{s}}\circ(\iota\circ \CF_{X,\psi}(f))(g).
\end{align*}
Therefore after applying the Mellin inversion formula, we get
\begin{align}\label{59-3}
\CF_{X,\psi^{-1}}= \iota\circ\CF_{X,\psi}.
\end{align}
By Definition \ref{dfn:SSFX}, we get that
\[
\CF_{X,\psi}(\CS_{\pvs}(X_{P_{\Del}})) =
\CF_{X,\psi}(\CC^{\infty}_{c}(X_{P_{\Del}}))+
\CF_{X,\psi}(\CF_{X,\psi}(\CC^{\infty}_{c}(X_{P_{\Del}}))).
\]
Since $\iota$ commutes with $\CF_{X,\psi}$ and preserves $\CC^\infty_c(X_{P_\Del})$, we have 
\[
\CF_{X,\psi}(\CC^{\infty}_{c}(X_{P_{\Del}}))=
\CF_{X,\psi}(\iota\circ \CC^{\infty}_{c}(X_{P_{\Del}}))
=
(\iota\circ\CF_{X,\psi})(\CC^{\infty}_{c}(X_{P_{\Del}})).
\]
Hence we obtain from \eqref{59-3} that
\[
\CF_{X,\psi}(\CC^{\infty}_{c}(X_{P_{\Del}}))=
\CF_{X,\psi^{-1}}(\CC^{\infty}_{c}(X_{P_{\Del}})),
\]
and
\[
\CF_{X,\psi}(\CF_{X,\psi}(\CC^{\infty}_{c}(X_{P_{\Del}})))
=
\CF_{X,\psi}(\CF_{X,\psi^{-1}}(\CC^{\infty}_{c}(X_{P_{\Del}})))
=
\CC^{\infty}_{c}(X_{P_{\Del}}).
\]
Therefore, we obtain the following:
\begin{align*}
\CF_{X,\psi}(\CS_{\pvs}(X_{P_{\Del}}))
=
\CF_{X,\psi}(\CC^{\infty}_{c}(X_{P_{\Del}}))+\CC^{\infty}_{c}(X_{P_{\Del}})=\CS_{\pvs}(X_{P_{\Del}}).
\end{align*}
We are done.
\end{proof}

Let $K_{\Del}$ be the maximal open compact subgroup of $\Sp_{4n}(F)$, as chosen in Theorem \ref{thm:PSR86}, such that $P_{\Del}(F)\cdot K_{\Del}=\Sp_{4n}(F)$ is the Iwasawa decomposition.
We have the following characterization of the functions in $\CS_{\pvs}(X_{P_{\Del}})$.

\begin{thm}\label{thm:asymSf}
A function $f\in \CC^{\infty}(X_{P_{\Del}})$ belongs to $\CS_{\pvs}(X_{P_{\Del}})$ if and only if $f$ is right $K_{\Del}$-finite and as a function in $a\in F^{\times}$,
$$
|a|^{2n+1}f(\Fs^{-1}_ak)
$$
belongs to $\CS^{-}_{\pvs}(F^{\times})$ for any fixed $k\in K_{\Del}$. In particular, a function $f\in \CC^{\infty}(X_{P_{\Del}})$ belongs to $\CC^\infty_c(X_{P_\Del})$ if and only if
$f$ is right $K_{\Del}$-finite and as a function in $a\in F^{\times}$, $f(\Fs^{-1}_ak)$ belongs to $\CC^\infty_c(F^\times)$ for any fixed $k\in K_{\Del}$.
\end{thm}

\begin{proof}
We first show the \emph{only if} part. For any $f\in \CS_{\pvs}(X_{P_{\Del}})$, from Definition \ref{dfn:SSFX}, there exist  $g_{1},g_{2}\in \CC^{\infty}_{c}(X_{P_{\Del}})$ such that
$$
f = g_{1}+\CF_{X,\psi}(g_{2}).
$$
Since both $g_{1}$ and $g_{2}$ are right $K_{\Del}$-finite, Definition \ref{dfn:FTX} shows that $f$ is also right $K_{\Del}$-finite. Fix $k\in K_{\Del}$. Since $g_{1}\in \CC^{\infty}_{c}(X_{P_{\Del}})$,
as a function of $a\in F^{\times}$, $|a|^{2n+1}g_{1}(\Fs^{-1}_a k)$ lies in $\CC^{\infty}_{c}(F^{\times})$. From Corollary \ref{FX-conv},
as a function of $a\in F^{\times}$, $|a|^{2n+1}\CF_{X,\psi}(g_{2})(\Fs^{-1}_a k)$ lies in $\CS^{-}_{\pvs}(F^{\times})$. This establishes the \emph{only if} part.

In order to show the \emph{if} part, by Proposition \ref{pro:Fg} and by  applying $\FL^{-1}$ to the function $f$ in the statement, it is equivalent to show the following statement:
\begin{itemize}
\item Any  $f\in \CC^{\infty}(X_{P_{\Del}})$ lies in the space $\CC^{\infty}_{c}(X_{P_{\Del}})+\mathrm{R}_{X}(\CC^{\infty}_{c}(X_{P_{\Del}}))$ if it is right $K_{\Del}$-finite and the following function in $a\in F^{\times}$
$$
|a|^{-(2n+1)}f(\Fs_a k)
$$
lies in $\CS^{+}_{\pvs}(F^{\times})$ for any $k\in K_{\Del}$.
\end{itemize}
By the Iwasawa decomposition, the space $\CC^{\infty}_{c}(X_{P_{\Del}})$ consists of functions $f\in \CC^{\infty}(X_{P_{\Del}})$ that are right $K_{\Del}$-finite and as a function in $a\in F^{\times}$,
$$
f(\Fs_a k)
$$
lies in $\CC^{\infty}_{c}(F^{\times})$ for any $k\in K_{\Del}$. Therefore, following the notation in Proposition \ref{pro:Fg}, we only need to show that for fixed $k\in K_{\Del}$, up to unramified shift, the set of functions
$$
\FF_k=\{ F_{k}\ |\ F_{k}(a) = |a|^{-(2n+1)}\mathrm{R}_{X}(h)(\Fs_a k), h\in \CC^{\infty}_{c}(X_{P_{\Del}})\}
$$
has Mellin transform given by $\CZ^{+}_{n,\bet}(\wh{F^{\times}})$, which is introduced in Definition \ref{defin:Mellin-gcd}.
Equivalently, from Proposition \ref{pro:gcd}, we only need to show the following equality for any fixed $k\in K_{\Del}$,
\begin{align*}
\{\CM(F_{k})(s+\frac{2n+1}{2},\chi)\ |\ F_{k}\in\FF_k\}
=a_{2n}(s,\chi)\cdot\BC[q^{s},q^{-s}].
\end{align*}
Indeed, this follows directly from the proof of Proposition \ref{pro:Fg}. Note that from Proposition \ref{pro:pole-Mw} the factor $a_{2n}(s,\chi)$ can indeed be achieved. Hence we also complete the proof for the \emph{if} part. It is clear that the above proof also proves the characterization for $f\in\CC^\infty(X_{P_\Del})$ to belong to $\CC^\infty_c(X_{P_\Del})$. We are done.
\end{proof}

\section{Harmonic Analysis on $\BG_m\times\Sp_{2n}$}\label{sec-FOG}


After establishing the basic properties of the Fourier transform $\CF_{X,\psi}$ on $X_{P_\Del}$ in Section \ref{sec-etapvs-FT}, we are ready to move to the right-hand side of the Diagram in \eqref{2OpenX},
which is recalled below,  in order to develop the basic properties in harmonic analysis related to the Fourier operator $\CF_{\rho,\psi}$ on $G=\BG_m\times\Sp_{2n}$ (Definitions \ref{dfn-SDFT} and \ref{dfn-SDFTG}).
The main goal of this section is to
prove Theorem \ref{thm:L2}, which is a combination of Theorems \ref{thm:Pl-FOG} and \ref{thm:FOG-Ss}. We assume as well that $n>0$ in this section.

\subsection{Cayley transform and the Jacobian}\label{ssec-CT-J}
We recall Diagram \eqref{2OpenX} as follows:
\begin{equation}\label{twoOpenX}
\begin{matrix}
&&&&\\
&&\Sp_{4n}&&\\
&&&&\\
&&\downarrow&&\\
&&&&\\
M^{\ab}_\Del w_\Del N_\Del&\rightarrow& X_{P_\Del}&\leftarrow& M^{\ab}_\Del(\Sp_{2n}\times\{\RI_{2n}\})\cong \BG_m\times\Sp_{2n}\\
&&&&
\end{matrix}
\end{equation}
which is a reformulation of Diagram \eqref{twoOpen},
where $X_{P_\Del}:=[P_\Del,P_\Del]\bks\Sp_{4n}$ is as in Section \ref{ssec-FoL2} and the group $G=\BG_m\times\Sp_{2n}$ naturally appears on the right-hand side of the diagram.
In order to figure out the geometric relation between the two Zariski open subvarieties $M^{\ab}_\Del w_\Del N_\Del$ and $M^{\ab}_\Del(\Sp_{2n}\times\{\RI_{2n}\})$ in $X_{P_\Del}$,
we have to go back to Diagram \eqref{twoOpen} to figure out the geometric relation between the two Zariski open subvarieties $P_\Del\bks P_\Del w_\Del N_\Del$ and $P_\Del\bks P_\Del(\Sp_{2n}\times\{\RI_{2n}\})$ in the
flag variety $P_{\Del}\bs\Sp_{4n}$, which, as shown in Lemma \ref{lem:cc}, is essentially the {\sl classical Cayley transform} in this setting. For the analytic purpose, we calculate the Jacobian of $\CC^{-1}$ in
Lemma \ref{lem:jacobian}.

\begin{dfn}\label{dfn:cc}
Write a generic element $w_{\Del}n_{\Del}$ in $w_{\Del}N_{\Del}(F)$ as $w_{\Del}n_{\Del} = p_{\Del}h$ for unique $p_{\Del}\in P_{\Del}(F)$ and $h=h(w_\Del n_\Del)\in \Sp_{2n}(F)\times \{ \RI_{2n}\}$ and define
a morphism $\CC$, which is a bi-rational morphism,  via the decomposition:
\begin{align*}
\CC: w_{\Del}N_{\Del}(F)&\to \Sp_{2n}(F)\times \{\RI_{2n}\}
\\
w_{\Del}n_{\Del}&\mapsto h=h(w_\Del n_\Del).
\end{align*}
\end{dfn}

As in Section \ref{ssec-lzidm}, we take the following ordered basis of $W=V^{+}\oplus V^-$ as in \eqref{sbasis}
\begin{equation*}
\{ e_{1}, e_{2},...,e_{n}, f_{1}, f_{2},...,f_{n}, e_{n+1},e_{n+2},...,e_{2n}, -f_{n+1},-f_{n+2},...,-f_{2n}\}.
\end{equation*}
This leads to a natural embedding of $\Sp(V^+)\times \Sp(V^-)$ into $\Sp(W)$ as defined in \eqref{action-2n2n} and an explicit expression for $\Sp_{2n}(F)\times\Sp_{2n}(F)$ embedded into $\Sp_{4n}(F)$:
\begin{align*}
(
\left(
\begin{smallmatrix}
A	&	B\\
C	&	D
\end{smallmatrix}
\right)
,
\left(
\begin{smallmatrix}
M	&	N\\
P	&	Q
\end{smallmatrix}
\right)
)
\mapsto
\left(
\begin{smallmatrix}
A 	&	&	B	&	\\
	&M	&		&-N	\\
C 	&	&	D	&	\\
 	&-P	&		&Q	\\
\end{smallmatrix}
\right).
\end{align*}
In the standard Lagrangian $L_{\std}$ in $W$, we may take an ordered basis
$$
\{ e_{n+1},e_{n+2},...,e_{2n},-f_{n+1},-f_{n+2},...,-f_{2n} \}.
$$
With the ordered basis \eqref{sbasis},  $P_{\std}(F)$ consists of matrices of the form as in \eqref{siegelp}.
On the other hand, the diagonal Lagrangian subspace $L_{\Del}$ has an ordered basis
$$
\{ e_{1}+f_{1},e_{2}+f_{2},...,e_{2n}+f_{2n} \}.
$$
We write the Levi decomposition $P_{*} = M_{*}N_{*}$, where $* = \std$ or $\Del$. If we take $g_0$ and $g_0^{-1}$ as in \eqref{g0} and \eqref{g0-1}, respectively, then we have
that for any $n_{\Del}\in N_{\Del}(F)$, $p_{\Del}\in P_{\Del}(F)$,
$$
w_{\Del} = g_{0}^{-1}w_{\std}g_{0},
\
n_{\Del} = g_{0}^{-1}n_{\std}g_{0},
\ {\rm and}\ \
p_{\Del} = g_{0}^{-1}p_{\std}g_{0},
$$
where $n_{\std}\in N_{\std}, p_{\std}\in P_{\std}$, and $w_*$ is the Weyl element for $*=\Del$ or $*=\std$, which takes the parabolic $P_*$ to the opposite.
Here
\[
w_{\std}=\ppair{\begin{smallmatrix}
&&&\frac{-\RI_n}{2}\\
&&\frac{I_n}{2}&\\
&2\RI_n&&\\
-2\RI_n&&&\\
\end{smallmatrix}}.
\]

We may write the expression
$w_{\Del}n_{\Del} = p_{\Del}h$
in Definition \ref{dfn:cc} as
$$
g_{0}^{-1}w_{\std}g_{0}\cdot g_{0}^{-1}n_{\std}g_{0} = g_{0}^{-1}p_{\std}g_{0} \cdot h,
$$
which can be simplified as
\begin{align}\label{simplifiedexp}
w_{\std}n_{\std} = p_{\std}\cdot g_{0}hg_{0}^{-1}.
\end{align}
Using the embedding (\ref{embsp}), we identify $h=\left(\begin{smallmatrix}
A &	B \\
C &    D
\end{smallmatrix}
\right)\in \Sp_{2n}(F)$ with $h\in \Sp_{2n}(F)\times \{\RI_{2n} \}$ by
$$
h = \left(
\begin{smallmatrix}
A &	&	B&	 \\
   &\RI_{n}&	  &  \\
C &    & D    &	\\
    &    &	   & \RI_{n}
\end{smallmatrix}
\right).
$$
It follows from a direct calculation that
$$
g_{0}\left(
\begin{smallmatrix}
A &	&	B&	 \\
   &\RI_{n}&	  &  \\
C &    & D    &	\\
    &    &	   & \RI_{n}
\end{smallmatrix}
\right)
g_{0}^{-1}
=
\left(
\begin{smallmatrix}
\frac{1}{2}(D+\RI_{n}) & -\frac{1}{2}C\ & -\frac{1}{4}C      & \frac{1}{4}(\RI_{n}-D)  \\
-\frac{1}{2}B     &    \frac{1}{2}(A+\RI_{n}) & \frac{1}{4}(A-\RI_{n}) & \frac{1}{4}B \\
-B  & A-\RI_{n} &	\frac{1}{2}(A+\RI_{n}) & \frac{1}{2}B \\
-D+\RI_{n} & C  &       \frac{1}{2}C      & \frac{1}{2}(D+\RI_{n})
\end{smallmatrix}
\right).
$$
From the expression (\ref{simplifiedexp}), we have
$
p_{\std}^{-1} = g_{0}hg_{0}^{-1}\cdot n_{\std}^{-1}\cdot w_{\std}^{-1}.
$
Here we set
$$
n_{\std} =
\left(
\begin{smallmatrix}
\RI_{n} &  &  \alp & \bet \\
        & \RI_{n} & {}^t\!\bet & \gam\\
        &         & \RI_{n} & \\
        &	  &          & \RI_{n}
\end{smallmatrix}
\right)\in N_{\std}(F)
$$
with
$
\left(
\begin{smallmatrix}
  \alp & \bet \\
 {}^t\!\bet & \gam\\
\end{smallmatrix}
\right)
\in \Sym^{2}(F^{2n}),
$
i.e.  $\alp = {}^t\!\alp, \gam= {}^t\!\gam$.
Since 
$
n_{\std}^{-1}
$
is equal to
$
\left(
\begin{smallmatrix}
\RI_{n} &  &  -\alp & -\bet \\
        & \RI_{n} & -{}^t\!\bet & -\gam\\
        &         & \RI_{n} & \\
        &	  &          & \RI_{n}
\end{smallmatrix}
\right),
$
we have
$$
n_{\std}^{-1}w_{\std}^{-1}=
\left(
\begin{smallmatrix}
 2 \beta & -2 \alpha & 0 & \frac{-\RI_n}{2} \\
 2 \gamma & -2 {}^t\!\beta & \frac{\RI_n}{2} & 0 \\
 0 & 2\RI_n & 0 & 0 \\
 -2\RI_n & 0 & 0 & 0 \\
\end{smallmatrix}
\right).
$$
The element  $p_{\std}^{-1} = g_{0}hg_{0}^{-1}\cdot n_{\std}^{-1}\cdot w_{\std}^{-1}$
can be written explicitly as the following product:
$$
\left(
\begin{smallmatrix}
\frac{1}{2}(D+\RI_{n}) & -\frac{1}{2}C\ & -\frac{1}{4}C      & \frac{1}{4}(\RI_{n}-D)  \\
-\frac{1}{2}B     &    \frac{1}{2}(A+\RI_{n}) & \frac{1}{4}(A-\RI_{n}) & \frac{1}{4}B \\
-B  & A-\RI_{n} &	\frac{1}{2}(A+\RI_{n}) & \frac{1}{2}B \\
-D+\RI_{n} & C  &       \frac{1}{2}C      & \frac{1}{2}(D+\RI_{n})
\end{smallmatrix}
\right)
\left(
\begin{smallmatrix}
2 \beta & -2 \alpha & 0 & \frac{-\RI_n}{2} \\
 2 \gamma & -2 {}^t\!\beta & \frac{\RI_n}{2} & 0 \\
 0 & 2\RI_n & 0 & 0 \\
 -2\RI_n & 0 & 0 & 0 \\
\end{smallmatrix}
\right),
$$
which is equal to
$$
\left(
\begin{smallmatrix}
(D+\RI_n)\beta-C\gamma-\frac{\RI_n-D}{2}&-(D+\RI_n)\alpha+C\cdot{}^t\!\beta-\frac{C}{2}&-\frac{C}{4}&-\frac{D+\RI_n}{4}\\
-B\beta+(A+\RI_n)\gamma-\frac{B}{2}&B\alpha-(A+\RI_n)\cdot{}^t\!\beta+\frac{A-\RI_n}{2}&\frac{A+\RI_n}{4}&\frac{B}{4}\\
-2B\beta+2(A-\RI_n)\gamma-B&2B\alpha-2(A-\RI_n)\cdot {}^t\!\beta+(A+\RI_n)&\frac{A-\RI_n}{2}&\frac{B}{2}\\
2(-D+\RI_n)\beta+2C\gamma-(D+\RI_n)&2(D-\RI_n)\alpha-2C\cdot {}^t\!\beta+C&\frac{C}{2}&\frac{D-\RI_n}{2}\\	
\end{smallmatrix}
\right).
$$
By definition, this last expression is equal to $p_{\std}^{-1}$.
Hence the lower-left $2n\times 2n$ block should be zero, which means that
\[
0_{2n}
=
\left(
\begin{smallmatrix}
-2B\beta+2(A-\RI_n)\gamma-B&2B\alpha-2(A-\RI_n)\cdot {}^t\!\beta+(A+\RI_n)	\\
2(-D+\RI_n)\beta+2C\gamma-(D+\RI_n)&2(D-\RI_n)\alpha-2C\cdot {}^t\!\beta+C
\end{smallmatrix}
\right).	
\]
Note that the lower-left $2n\times 2n$ block equals
\begin{align*}
&
2\left(
\begin{smallmatrix}
B & A-\RI_n \\
D-\RI_{n} & C
\end{smallmatrix}
\right)
\left(
\begin{smallmatrix}
\alp & -\bet \\
-{}^t\!\bet & \gam
\end{smallmatrix}
\right)
+
\left(
\begin{smallmatrix}
A+\RI_n & -B  \\
C    &   -(D+\RI_n)
\end{smallmatrix}
\right)
\\
=&\left(
\begin{smallmatrix}
A+\RI_{n} & B  \\
C    &  D+\RI_{n}
\end{smallmatrix}
\right)
+
2\left(
\begin{smallmatrix}
B & \RI_{n}-A \\
D-\RI_{n} & -C
\end{smallmatrix}
\right)
\left(
\begin{smallmatrix}
\alp & \bet \\
{}^t\!\bet & \gam
\end{smallmatrix}
\right)
\\
=&
\left(
\begin{smallmatrix}
\RI_{n} &  \\
   & \RI_{n}
\end{smallmatrix}
\right)
+
\left(
\begin{smallmatrix}
A & B \\
C  & D
\end{smallmatrix}
\right)
+
2(
\left(
\begin{smallmatrix}
	& \RI_{n} \\
-\RI_{n}&
\end{smallmatrix}
\right)
+
\left(
\begin{smallmatrix}
B & -A \\
D & -C
\end{smallmatrix}
\right)
)
\left(
\begin{smallmatrix}
\alp & \bet \\
{}^t\!\bet & \gam
\end{smallmatrix}
\right)\\
=&
(\RI_{2n}+h)+2(\RJ_n-h\RJ_n)\left(
\begin{smallmatrix}
\alp & \bet \\
{}^t\!\bet & \gam
\end{smallmatrix}
\right).
\end{align*}
Hence we obtain the following relation:
\begin{equation}\label{relation}
\left(
\begin{smallmatrix}
\alp & \bet \\
{}^t\! \bet & \gam
\end{smallmatrix}
\right)
= \frac{1}{2}\RJ_{n}(\RI_{2n}-h)^{-1}(\RI_{2n}+h).
\end{equation}
From this relation,  we write $h$ in terms of
$X = \left(
\begin{smallmatrix}
\alp & \bet \\
{}^t\!\bet & \gam
\end{smallmatrix}
\right) \in \Sym^{2}(F^{2n})$.
In fact, from \eqref{relation}, we have
$$
(\RI_{2n}-h)(-2\RJ_nX)=\RI_{2n}+h
$$
and
$$
h(2\RJ_nX-\RI_{2n})=\RI_{2n}+2X\RJ_n.
$$
Hence we obtain
$$
h=  (2\RJ_{n}X+\RI_{2n})(2\RJ_{n}X-\RI_{2n})^{-1}.
$$
Moreover, using the fact that ${}^t\!\RJ_{n} = -\RJ_{n}$, we find that $\RJ_{n}X\in \mathfrak{sp}_{2n}(F)$, the Lie algebra of $\Sp_{2n}(F)$. This means that
$$
(X\RJ_{n})\RJ_{n}+\RJ_{n}{}^t\!(X\RJ_{n}) = 0.
$$
The discussion above proves the following lemma.

\begin{lem}\label{lem:cc}
The morphism $\CC$ defined in Definition \ref{dfn:cc} is given explicitly as follows.
For any $n_\Del$ in $N_\Del(F)$, we write $n_{\Del}=g_{0}^{-1}n_{\std}g_{0}$ and
$
n_{\std} =
\left(
\begin{smallmatrix}
\RI_{2n}	&	X	\\
		&	\RI_{2n}
\end{smallmatrix}
\right)\in N_{\std}(F),
$
with $X=X(n_\Del)\in \Sym^{2}(F^{2n})$ depending on $n_\Del$. Then $\CC(w_\Del n_\Del)=h=(h,\RI_{2n})\in\Sp_{2n}(F)\times\{\RI_{2n}\}$ is given by taking
$$
h= (2\RJ_{n}X+\RI_{2n})(2\RJ_{n}X-\RI_{2n})^{-1}.
$$
Note that the morphism is defined over an open dense subset.
\end{lem}

Note that for $X\in \Sym^{2}(F^{2n})$, $Y=\RJ_nX$ belongs to the Lie algebra $\mathfrak{sp}_{2n}(F)$, and hence the map
$$
Y\in\mathfrak{sp}_{2n}(F) \mapsto h = (2Y+\RI_{2n})(2Y-\RI_{2n})^{-1} \in \Sp_{2n}(F)
$$
is the classical Cayley transform.

\begin{rmk}\label{gmfactor}
Using the notation above we can also describe the matrices $p_{\std}^{-1}$ and the Levi part of $p_{\std}$ as follows.
By previous calculations,
\begin{align*}
p_{\std}^{-1}
=
\left(
\begin{smallmatrix}
(D+\RI_n)\beta-C\gamma-\frac{\RI_n-D}{2}&-(D+\RI_n)\alpha+C\cdot{}^t\!\beta-\frac{C}{2}&-\frac{C}{4}&-\frac{D+\RI_n}{4}\\
-B\beta+(A+\RI_n)\gamma-\frac{B}{2}&B\alpha-(A+\RI_n)\cdot{}^t\!\beta+\frac{A-\RI_n}{2}&\frac{A+\RI_n}{4}&\frac{B}{4}\\
0&0&\frac{A-\RI_n}{2}&\frac{B}{2}\\
0&0&\frac{C}{2}&\frac{D-\RI_n}{2}\\	
\end{smallmatrix}
\right).
\end{align*}
Here we notice that the lower-right $2n\times 2n$ block can be written as
$$
\left(
\begin{smallmatrix}
\frac{A-\RI_n}{2}&\frac{B}{2}\\
\frac{C}{2}&\frac{D-\RI_n}{2}\\
\end{smallmatrix}
\right)
=\frac{1}{2}(h-\RI_{2n}).
$$
By \eqref{siegelp}, the Levi part of $p_{\std}$ must have the following expression
\begin{equation}\label{levi-std}
\left(
\begin{smallmatrix}
 \frac{1}{2}({}^t\!h-\RI_{2n})& 0  \\
	0& 2(h-\RI_{2n})^{-1}
\end{smallmatrix}
\right).
\end{equation}
\end{rmk}

Next we are going to calculate the Jacobian of the morphism $\CC^{-1}$.
The $F$-birational morphism $\CC$ from $w_\Del N_\Del$ to $\Sp_{2n}\times\{\RI_{2n}\}$ as defined in Definition \ref{dfn:cc} has been explicated in Lemma \ref{lem:cc}. We have to discuss the relations of various measures
involved in the calculation. From the fixed Haar measure ${\rm d} g$ on $\Sp_{4n}(F)$, we have a unique natural measure on $P_\Del(F)\bs\Sp_{4n}(F)$, which may still be denoted by ${\rm d} g$. In Diagram \eqref{twoOpen},
the measure ${\rm d} g$ induces a measure ${\rm d}_\Del(w_\Del n_\Del)$ on $w_\Del N_\Del(F)$ and a right-invariant measure ${\rm d}_\Del h$ on $\Sp_{2n}(F)$.
As fixed at the beginning of this paper, the Haar measure ${\rm d} h$ on $\Sp_{2n}(F)$ is normalized so that $\vol(\Sp_{2n}(\CO))=1$. Hence ${\rm d}_\Del h$ and ${\rm d} h$ are different by a constant.
It is clear that the measure ${\rm d}_\Del(w_\Del n_\Del)$ on $w_\Del N_\Del(F)$ can be induced from the measure ${\rm d} n_\Del$ on $N_\Del(F)$. Via the $F$-birational morphism $\CC$, we have
\begin{align}\label{FF-8}
\ud n_\Del=\Fj_{\CC^{-1}}(h)\ud h,
\end{align}
where $\Fj_{\CC^{-1}}$ denotes the Jacobian of $\CC^{-1}$. As we explained right below Lemma \ref{lem:cc}, the $F$-birational morphism $\CC$ is essentially the classical Cayley transform. The following lemma computes $\Fj_{\CC^{-1}}(h)$ explicitly.

\begin{lem}\label{lem:jacobian}
The Jacobian of the morphism $\CC^{-1}$ is given by
\begin{align}\label{FF-9}
\Fj_{\CC^{-1}}(h)=c_0|\det(h-\RI_{2n})|^{-(2n+1)}
\end{align}
where $c_0 = \prod_{i=1}^{n}\frac{1}{\zet_{F}(2i)}$ and $\zet_{F}(s)$ is the local Dedekind zeta function of $F$.
\end{lem}

\begin{proof}
By Lemma \ref{lem:cc}, we only need to compute the Jacobian of the Cayley transform
$$
Y\in \Fs\Fp_{2n}(F)\to h=(2Y+\RI_{2n})(2Y-\RI_{2n})^{-1}\in \Sp_{2n}(F).
$$
Let $\RD h$ and $\RD Y$ be the differentials of $h\in \Sp_{2n}(F)$ and $Y\in \Fs\Fp_{2n}(F)$, respectively.

The differential $\RD Y$ admits the following description. The affine space $\Fs\Fp_{2n}$ has a natural smooth scheme structure over $\CO$, whose $F$-points is the vector space $\Fs\Fp_{2n}(F)$. The sheaf of   K\"ahler differentials $\Ome_{\Fs\Fp_{2n}/\CO}$ is free of rank $\dim \Fs\Fp_{2n} = n(2n+1)$. Let $\RD Y_{\CO}$ be a constant section associated to a unit element in $\Gam(\Fs\Fp_{2n},\Ome_{\Fs\Fp_{2n}/\CO})$ which is the space of global sections of $\Ome_{\Fs\Fp_{2n}/\CO}$. Up to a unit element in $\Gam(\Fs\Fp_{2n},\CO_{\Fs\Fp_{2n}})$ where $\CO_{\Fs\Fp_{2n}}$ is the structure sheaf of $\Fs\Fp_{2n}$ over $\CO$, its image in 
\[
\Gam(\Fs\Fp_{2n},\Ome_{\Fs\Fp_{2n}/F}) \simeq \Gam(\Fs\Fp_{2n},\Ome_{\Fs\Fp_{2n}/\CO}\times_{\Spec(\CO)}\Spec(F))
\] 
is equal to $\RD Y$.

The differential $\RD h$ admits the following description. The $F$-variety $\Sp_{2n}$ has a natural smooth scheme structure over $\CO$, whose $F$-points is $\Sp_{2n}(F)$. As a scheme over $\CO$, $\Sp_{2n}$ embeds into $\RM_{2n}$ in a natural way. We let $i:\Sp_{2n}\hookrightarrow \RM_{2n}$ be the closed embedding. Let $\RD X^{2n}_{\CO}$ be a constant section associated to a unit element in $\Gam(\RM_{2n},\Ome_{\RM_{2n}/\CO})$. Then its image under the natural pull-back morphism 
\[
i^{*}\Ome_{\RM_{2n}/\CO}\to \Ome_{\Sp_{2n}/\CO}
\]
is a unit element in $\Gam(\Sp_{2n},\Ome_{\Sp_{2n}/\CO})$. Up to a unit element in $\Gam(\Sp_{2n},\CO_{\Sp_{2n}})$ where $\CO_{\Sp_{2n}}$ is the structure sheaf of $\Sp_{2n}$ over $\CO$, the image in 
\[
\Gam(\Sp_{2n},\Ome_{\Sp_{2n}/F})\simeq \Gam(\Sp_{2n},\Ome_{\Sp_{2n}/\CO}\times_{\Spec(\CO)}\Spec(F))
\] 
is equal to $\RD Y$.

After taking differential on both sides of the equation 
\[
h(2Y-\RI_{2n}) = 2Y+\RI_{2n},
\]
we obtain 
$$
\RD h(2Y-\RI_{2n}) +2h\RD Y = 2\RD Y.
$$
Therefore, we have 
$$
\RD h(Y-\frac{\RI_{2n}}{2}) = (\RI_{2n}-h)\RD Y.
$$
By explicit computation, we obtain 
$$
Y-\frac{\RI_{2n}}{2} = (h-\RI_{2n})^{-1}.
$$
Hence
\begin{equation}\label{eq:dif-rln}
\RD h = -(\RI_{2n}-h)\RD Y(\RI_{2n}-h).
\end{equation}
After applying the morphism
$
\bigwedge^{\max}
$
to \eqref{eq:dif-rln}, the left-hand side of \eqref{eq:dif-rln} gives a top degree differential form ${\rm d}^{\CC}h$ on $\Sp_{2n}(F)$. Similarly, the right-hand side of \eqref{eq:dif-rln} gives a top degree differential form ${\rm d}^{\CC}Y$ on $\Fs\Fp_{2n}(F)$. Therefore we get the equality ${\rm d}^{\CC}h = \ud^{\CC}Y$.

By \cite{W82}, the differential forms induce  measures on $\Sp_{2n}(F)$ and $\Fs\Fp_{2n}(F)$, which we still denote as ${\rm d}^{\CC}h$ and ${\rm d}^{\CC}Y$. As we have already fixed the Haar measure ${\rm d} h$ (resp. ${\rm d} Y$) on $\Sp_{2n}(F)$ (resp. $\Fs\Fp_{2n}(F)$) by $\vol(\Sp_{2n}(\CO)) = 1$ (resp. $\vol(\Fs\Fp_{2n}(\CO)) =1$), we only need to find the difference between ${\rm d}^{\CC}h$ (resp. ${\rm d}^{\CC}Y$) and ${\rm d} h$ (resp. ${\rm d} Y$).

We first show the following statement.
\begin{itemize}
\item As measures on $\Fs\Fp_{2n}(F)$,
${\rm d}^{\CC}Y = |\det(h-\RI_{2n})|^{2n+1}\ud Y$.
\end{itemize}
By definition, the volume form $\bigwedge^{\max}\RD Y$ is obtained from a constant section $\bigwedge^{\max}\RD Y_{\CO}$ associated to a unit element in $\Gam(\Fs\Fp_{2n},\bigwedge^{\max}\Ome_{\Fs\Fp_{2n}/\CO})$, which in particular is a gauge form in the sense of \cite[2.2.2]{W82}. By \cite[Theorem~2.2.5]{W82}, $\vol(\Fs\Fp_{2n}(\CO), \bigwedge^{\max}\RD Y) = 1$. Therefore ${\rm d} Y = \bigwedge^{\max}\RD Y$.

For the matrix $\RI_{2n}-h$, there exist elementary matrices $P,Q$ and a diagonal matrix $A$, such that $(\RI_{2n}-h) = PAQ$. By direct computation,
$$
\bigwedge^{\max}P\RD YQ = \pm \bigwedge^{\max}\RD Y = \pm \ud Y,
$$
and
$$
\bigwedge^{\max}A\RD YA = \det(A)^{2n+1}\bigwedge^{\max}\RD Y = \det(A)^{2n+1}\ud Y.
$$
Therefore
$$
\ud^{\CC}Y=\bigwedge^{\max}(-(\RI_{2n}-h)\RD Y(\RI_{2n}-h))=\pm \det(\RI-h)^{2n+1}\ud Y.
$$
In particular, as measures on $\Fs\Fp_{2n}(F)$, ${\rm d}^{\CC}Y = |\det(h-\RI_{2n})|^{2n+1}\ud Y$.

Then we establish the following statement.

\begin{itemize}
\item As measures on $\Sp_{2n}(F)$, ${\rm d}^{\CC}h =c_{0} \ud h$ where $c_{0}= \prod_{i=1}^{n}\frac{1}{\zet_{F}(2i)}$.
\end{itemize}
By definition, $\RD h$ is obtained from a constant section associated to a unit element $\RD h_{\CO}$ in $\Gam(\Sp_{2n},\Ome_{\Sp_{2n}/\CO})$. Therefore ${\rm d}^{\CC}h = \bigwedge^{\max}\RD h$, which lies in $\Gam(\Sp_{2n},\bigwedge^{\max}\Ome_{\Sp_{2n}/F})$, is obtained from the constant section associated to the unit element $\bigwedge^{\max}\RD h_{\CO}$ in $\Gam(\Sp_{2n},\bigwedge^{\max}\Ome_{\Sp_{2n}/\CO})$. In particular ${\rm d}^{\CC}h$ is a gauge form in the sense of \cite[2.2.2]{W82}, and the measure induced by ${\rm d}^{\CC}h$ is a Haar measure of $\Sp_{2n}(F)$. Hence as measures on $\Sp_{2n}(F)$, ${\rm d}^{\CC}h$ and ${\rm d} h$ differ by a nonzero constant, which can be calculated through the volume of $\Sp_{2n}(\CO)$ using ${\rm d}^{\CC}h$. By \cite[Theorem~2.2.5]{W82}, we have 
\[
\vol(\Sp_{2n}(\CO),\ud^{\CC}h) = \frac{|\Sp_{2n}(\BF_{q})|}{q^{\dim \Sp_{2n}}}= \frac{q^{n^{2}}\prod_{i=1}^{n}(q^{2i}-1)}{q^{n(2n+1)}}= \frac{1}{\prod_{i=1}^{n}\zet_{F}(2i)}, 
\]
where $\BF_{q}=\CO/\CP$. Therefore ${\rm d}^{\CC}h = c_{0}\ud h$.
This proves the lemma.
\end{proof}

\subsection{Fourier operator and Schwartz space on $G$}\label{ssec-FoSsG}

Via the embedding $G\to X_{P_\Del}$ in Diagram \eqref{twoOpenX},
$G$ is open dense in $X_{P_\Del}$. More precisely, for any $(a,h)\in G(F)$, the image of $(a,h)$ through that embedding is $\Fs_a^{-1}\cdot(h,\RI_{2n})$ in $X_{P_\Del}(F)$, where
$\Fs_a$ is the section associated to the abelianization morphism $\Fa$ as defined in \eqref{section:Fs} and \eqref{M-Mab}, respectively. For any $f\in\CS_{\pvs}(X_{P_\Del})$, define
\begin{align}\label{f-phi}
\phi(a,h)=\phi_f(a,h):=f(\Fs_a^{-1}\cdot(h,\RI))|a|^{\frac{2n+1}{2}}
\end{align}
with $\RI=\RI_{2n}$. It is clear that any smooth function $f$ on $X_{P_\Del}(F)$ is completely determined by its values $f(\Fs_a^{-1}(h,\RI))$ with all $(a,h)\in G(F)$.

\begin{dfn}[$\rho$-Schwartz Space on $G$]\label{dfn:rho-SsG}
The $\rho$-Schwartz space $\CS_{\rho}(G)$ on $G(F)$ is defined to be
$$
\CS_\rho(G):=
\{\phi_f\ |\ \forall\ f\in\CS_\pvs(X_{P_\Del})\},
$$
where $\phi_f(a,h)$ is defined in \eqref{f-phi}.
\end{dfn}

With the help of the generalized function $\eta_{\pvs,\psi}(x)$ on $\BG_m(F)=F^\times$ as in Theorem \ref{thm:eta}, we make the following definitions.

\begin{dfn}[Function $\Phi_{\rho,\psi}$ on $G$]\label{dfn-SDFT}
Let $\eta_{\pvs,\psi}(x)$ be the generalized function on $F^\times$ as defined in \eqref{def:eta-rho-psi}. Define a function $\Phi_{\rho,\psi}$ on $G(F)=F^\times\times\Sp_{2n}(F)$ to be
\begin{align}\label{def:SDG}
\Phi_{\rho,\psi}(a,h):=
c_0\cdot\eta_{\pvs,\psi}(a\det(h+\RI))\cdot|\det(h+\RI)|^{-\frac{2n+1}{2}},
\end{align}
with the constant $c_0 = \prod_{i=1}^{n}\frac{1}{\zet_{F}(2i)}$ as in \eqref{FF-9}, where $\zet_{F}(s)$ is the local Dedekind zeta function of $F$.
\end{dfn}

\begin{pro}\label{pro:Phi}
The function $\Phi_{\rho,\psi}$ defined in Definition \ref{dfn-SDFT} enjoys the following properties.
\begin{itemize}
\item[(1)] It is locally integrable on $G(F)$.
\item[(2)] It is $G(F)$-invariant under the adjoint action of $G(F)$ on $G(F)$.
\item[(3)] For any $(a,h)\in G(F)$, $\Phi_{\rho,\psi}((a,h))=\Phi_{\rho,\psi}((a,h^{-1}))$.
\end{itemize}
\end{pro}

\begin{proof}
Since the generalized function $\eta_{\pvs,\psi}(x)$ on $F^\times$ is locally constant (Theorem \ref{thm:eta}), it is clear that the function $\Phi_{\rho,\psi}$ is locally integrable on $G(F)$. This proves
Part (1). Parts (2) and (3) follow directly from the definition.
\end{proof}

Hence the function $\Phi_{\rho,\psi}(g)$ on $G(F)$ defines a $G(F)$-invariant distribution on the space $\CC_c^\infty(G)$. We refer to Section \ref{ssec-pf-Thm13} for more discussions about this distribution.

\begin{dfn}[Fourier Operator on $G$]\label{dfn-SDFTG}
Let $\Phi_{\rho,\psi}$ be the $G(F)$-invariant distribution on $G(F)$ as in Definition \ref{dfn-SDFT}.
The Fourier operator $\CF_{\rho,\psi}$ over $G(F)$ with $\Phi_{\rho,\psi}$ as the convolution kernel function is defined by
\begin{align}\label{def:FTG}
\CF_{\rho,\psi}(\phi_f)(a,h)
&:=
(\Phi_{\rho,\psi}*\phi_f^\vee)(a,h)\nonumber\\
&=
\int_{F^\times}^\pv\int_{\Sp_{2n}(F)}\Phi_{\rho,\psi}(t,y)\phi_f^\vee(t^{-1}a,y^{-1}h)\ud y\ud^*t,
\end{align}
for all $f\in\CC^\infty_c(X_{P_\Del})$ with $\phi_f$ as defined in \eqref{f-phi}, where $\phi^\vee(g)=\phi(g^{-1})$.
\end{dfn}

\subsection{Compatibility of $\CF_{\rho,\psi}$ with $\CF_{X,\psi}$ and $\RM_{w_\Del}^\dag(s,\chi,\psi)$}\label{ssec-CFG-CFX-RM}

The compatibility of the Fourier operator $\CF_{\rho,\psi}$ over $G(F)$ with Fourier transform $\CF_{X,\psi}$ over $X_{P_\Del}(F)$ and with the normalized local intertwining operator
$\RM_{w_\Del}^\dag(s,\chi,\psi)$ is the core technical result for the proof of basic results on harmonic analysis for $\Phi_{\rho,\psi}$ and $\CF_{\rho,\psi}$ in the rest of this paper.

For $f\in\CS_{\pvs}(X_{P_{\Del}})$, take $\phi_f(a,h)=f(\Fs_a^{-1}(h,\RI))|a|^{\frac{2n+1}{2}}$ as defined in \eqref{f-phi}, for $g=(a,h)\in G(F)$.
From Definition \ref{dfn:FTX}, for $f\in \CC^{\infty}_c(X_{P_{\Del}})$, the Fourier transform $\CF_{X,\psi}(f)(\Fs_a^{-1}(h,\RI))$ is equal to
\begin{align}\label{FF-1}
\int_{F^\times}^\pv\eta_{\psi}(t)|t|^{-\frac{2n+1}{2}}
\int_{N_\Del(F)}
f(w_\Del n_\Del  \Fs_t \Fs_a^{-1}(h,\RI))\ud n_\Del \ud^* t,
\end{align}
where $\eta_\psi(t)=\eta_{\pvs,\psi}(t)$ is the generalized function on $F^\times$ as defined in Theorem \ref{thm:eta}.
We write
\[
w_\Del n_\Del \cdot \Fs_t\Fs_a^{-1}=w_\Del \Fs_t\Fs_a^{-1} \cdot \Fs_a\Fs_t^{-1} n_\Del  \Fs_t\Fs_a^{-1}.
\]
By changing the variable $n_\Del\mapsto  \Fs_t\Fs_a^{-1} n_\Del \Fs_a\Fs_t^{-1}$, we obtain that \eqref{FF-1} is equal to
\begin{align}\label{FF-2}
|a|^{-(2n+1)}
\int_{F^\times}^\pv\eta_{\psi}(t)|t|^{\frac{2n+1}{2}}
\int_{N_\Del(F)}
f(w_\Del \Fs_t\Fs_a^{-1}  n_\Del  (h,\RI)) \ud n_\Del \ud^* t.
\end{align}
Since $w_\Del \Fs_t\Fs_a^{-1}=\Fs_t^{-1}\Fs_a w_\Del$, we obtain that  \eqref{FF-2} is equal to
\begin{align}\label{FF-3}
|a|^{-(2n+1)}
\int_{F^\times}^\pv\eta_{\psi}(t)|t|^{\frac{2n+1}{2}}
\int_{N_\Del(F)}
f(\Fs_t^{-1}\Fs_a w_\Del  n_\Del  (h,\RI)) \ud n_\Del \ud^* t.
\end{align}

In order to convert the Fourier transform $\CF_{X,\psi}$ over $X_{P_\Del}(F)$ to the Fourier operator $\CF_{\rho,\psi}$ over $G(F)$, we consider the inner integral of \eqref{FF-3}:
\begin{align}\label{FF-4}
\int_{N_\Del(F)}
f(\Fs_t^{-1}\Fs_a w_\Del  n_\Del (h,\RI)) \ud n_\Del,
\end{align}
which is closely related to the Radon transform $\RR_{X}(f)$ as defined in \eqref{Radon}.
As in Definition \ref{dfn:cc} we write $w_\Del n_\Del=p_\Del y=p_\Del(y,\RI)$ with $y\in\Sp_{2n}(F)$. Then we have
\begin{align*}
f(\Fs_t^{-1}\Fs_a w_\Del  n_\Del)=
f(\Fs_t^{-1}\Fs_a p_\Del (y,\RI))=
f(m_\Del\Fs_t^{-1}\Fs_a  (y,\RI)),
\end{align*}
where $m_\Del$ is the Levi part of $p_\Del$.
Taking $g_0$ and $g_0^{-1}$ to be as in \eqref{g0} and \eqref{g0-1}, respectively, we have that $M_\Del=g_0^{-1}M_{\std} g_0$, and in particular, $m_\Del=g_0^{-1} m_{\std} g_0$.
From Lemma \ref{lem:cc} and \eqref{levi-std}, $m_{\std}$ is equal to
$$
\left(
\begin{smallmatrix}
 \frac{1}{2}({}^t\!y-\RI_{2n})& 0  \\
	0& 2(y-\RI_{2n})^{-1}
\end{smallmatrix}
\right).
$$
By the normalization of the abelianization morphism $\Fa$ as in \eqref{M-Mab}, we have that  $\Fa(m_\Del)=2^{-2n}\det(y-\RI_{2n})$.
It follows that
\begin{equation}\label{FF-5}
f(\Fs_t^{-1}\Fs_a w_\Del  n_\Del)
=f(\Fs_{\Fa(m_\Del)}\Fs_t^{-1}\Fs_a(y,\RI))=f(\Fs_{\det(y-\RI_{2n})}\Fs_{t}^{-1}\Fs_{2^{-2n}a}(y,\RI)),
\end{equation}
for any $f\in\CC^\infty_c(X_{P_\Del})$. Hence we obtain that \eqref{FF-3} is equal to
\begin{align}\label{FF-6}
|a|^{-(2n+1)}
\int_{F^\times}^\pv\eta_{\psi}(t)|t|^{\frac{2n+1}{2}}
\int_{N_\Del(F)}
f(\Fs_{\det(y-\RI_{2n})}\Fs_{t}^{-1}\Fs_{2^{-2n}a}(yh,\RI)) \ud n_\Del \ud^* t.
\end{align}
By changing the variable $t\mapsto t\det(y-\RI)$, we have that \eqref{FF-6} is equal to
\begin{align}\label{FF-7}
&|a|^{-(2n+1)}
\int_{F^\times}^\pv\int_{N_\Del(F)}
\eta_{\psi}(t\det(y-\RI))|t\det(y-\RI)|^{\frac{2n+1}{2}}\nonumber\\
&\qquad\qquad\qquad\qquad\qquad\qquad\qquad\cdot f(\Fs_{t}^{-1}\Fs_{2^{-2n}a}(yh,\RI)) \ud n_\Del \ud^* t.
\end{align}

By \eqref{FF-8}, we write ${\rm d} n_\Del=\Fj_{\CC^{-1}}(y)\ud y$, because of $w_\Del n_\Del=p_\Del y$ with $y\in\Sp_{2n}(F)$.
The Jacobian $\Fj_{\CC^{-1}}$ of $\CC^{-1}$ is explicitly calculated in Lemma \ref{lem:jacobian}, which is given by
\begin{align}\label{Jac}
\Fj_{\CC^{-1}}(y)=c_0|\det(y-\RI_{2n})|^{-(2n+1)}
\end{align}
where $c_0 = \prod_{i=1}^{n}\frac{1}{\zet_{F}(2i)}$ and $\zet_{F}(s)$ is the local Dedekind zeta function of $F$.
By putting those calculations together and changing the variable $h\mapsto -h$, we deduce that \eqref{FF-7} is equal to
\begin{align}\label{FF-10}
&c_0|a|^{-(2n+1)}
\int_{F^\times}^\pv\int_{\Sp_{2n}(F)}
\eta_{\psi}(t\det(y+\RI))|\det(y+\RI)|^{-\frac{2n+1}{2}}\nonumber\\
&\qquad\qquad\qquad\qquad\qquad\cdot |t|^{\frac{2n+1}{2}}f(\Fs_{t}^{-1}\Fs_{2^{-2n}a}(-yh,\RI)) \ud y \ud^* t.
\end{align}
By \eqref{f-phi}, we have
\begin{align*}
|t|^{\frac{2n+1}{2}}f(\Fs_{t}^{-1}\Fs_{2^{-2n}a}(-yh,\RI))
&=
\phi_f(t(2^{2n}a^{-1}),-yh)|2^{2n}a^{-1}|^{-\frac{2n+1}{2}}\\
&=\phi_f^\vee(t^{-1}(2^{-2n}a),(-h^{-1})y^{-1})|2^{2n}a^{-1}|^{-\frac{2n+1}{2}}.
\end{align*}
Hence we deduce that \eqref{FF-10} is equal to
\begin{align}\label{FF-11}
&c_0|2|^{-2n(2n+1)}
\int_{F^\times}^\pv\int_{\Sp_{2n}(F)}
\eta_{\psi}(t\det(y+\RI))|\det(y+\RI)|^{-\frac{2n+1}{2}}\nonumber\\
&\qquad\qquad\qquad\cdot|2^{-2n}a|^{-\frac{2n+1}{2}}\phi_f^\vee(t^{-1}(2^{-2n}a),(-h^{-1})y^{-1}) \ud y \ud^* t.
\end{align}
By the expression of the Fourier operator $\CF_{\rho,\psi}$ over $G(F)$ in Definition \ref{dfn-SDFTG}, we are able to write \eqref{FF-11} as
\begin{align*} 
|2|^{-2n(2n+1)}|2^{-2n}a|^{-\frac{2n+1}{2}}\cdot \CF_{\rho,\psi}(\phi_f)(2^{-2n}a,-h^{-1})
\end{align*}
for $g=(a,h)\in G(F)$, which is equal to
\begin{align}
|2|^{-2n(2n+1)}f_{\CF_{\rho,\psi}(\phi_f)}(\Fs_{2^{-2n}a}^{-1}\cdot(-h^{-1},\RI)).
\end{align}
This establishes the compatibility of the Fourier transform $\CF_{X,\psi}$ over $X_{P_\Del}(F)$ and the Fourier operator $\CF_{\rho,\psi}$ over $G(F)$.

\begin{thm}[Compatibility of $\CF_{\rho,\psi}$ with $\CF_{X,\psi}$]\label{thm:FT-FT}
For any function $f\in\CC^\infty_c(X_{P_\Del})$, let $\phi_f\in\CS_\rho(G)$ be as defined in \eqref{f-phi}. The Fourier transform $\CF_{X,\psi}$ over $X_{P_\Del}(F)$ and
the Fourier operator $\CF_{\rho,\psi}$ over $G(F)$ are related by the following identities:
\begin{align}\label{CF-f}
\CF_{X,\psi}(f)(\Fs_a^{-1}(h,\RI))
&=
|2|^{-2n(2n+1)}f_{\CF_{\rho,\psi}(\phi_f)}(\Fs_{2^{-2n}a}^{-1}(-h^{-1},\RI));
\end{align}
and equivalently
\begin{align}\label{phi-CF-f}
\phi_{\CF_{X,\psi}(f)}(a,h)
=
|2|^{-n(2n+1)}\CF_{\rho,\psi}(\phi_f)(2^{-2n}a,-h^{-1}),
\end{align}
for any $g=(a,h)\in G(F)$.
\end{thm}
It is clear that $f_{\CF_{\rho,\psi}(\phi_f)}$ makes sense according to the transition relation \eqref{f-phi} if we prove that $\CF_{\rho,\psi}(\phi_f)$ belongs to the $\rho$-Schwartz space $\CS_\rho(G)$,
which is the subject matter of Theorem \ref{thm:FOG-Ss}. At this point, it is understood through the following identity:
\begin{align}\label{f-phi-g}
f_{\CF_{\rho,\psi}(\phi_f)}(\Fs_{a}^{-1}(-h^{-1},\RI))=|a|^{-\frac{2n+1}{2}}\CF_{\rho,\psi}(\phi_f)(a,-h^{-1})
\end{align}
for any $g=(a,h)\in G(F)$. Since the left-hand side of \eqref{CF-f} is the restriction of $\CF_{X,\psi}(f)\in\CS_\pvs(X_{P_\Del})$, the function on the right-hand side of \eqref{CF-f} makes sense.

From Theorem \ref{thm:FT-FT} we are able to extend the Fourier operator $\CF_{\rho,\psi}$, which was defined for $\phi_f$ with $f\in \CC_c^\infty(X_{P_\Del})$ in Definition \ref{dfn-SDFTG},
to the whole $\rho$-Schwartz space $\CS_{\rho}(G)$ (with $f\in\CS_\pvs(X_{P_\Del})$).

\begin{dfn}\label{dfn:FTtoSG}
For any $f\in \CS_{\pvs}(X_{P_{\Del}})$, let $\phi_{f}$ be defined in \eqref{f-phi}. Define
$\CF_{\rho,\psi}(\phi_{f})$ through identity \eqref{phi-CF-f},
\begin{equation}\label{FTtoSG}
\CF_{\rho,\psi}(\phi_{f})(a,h) =
|2|^{n(2n+1)}\phi_{\CF_{X,\psi}(f)}(2^{2n}a,-h^{-1}),
\end{equation}
for any $g=(a,h)\in G(F)$.
\end{dfn}
As remarked below \eqref{f-phi}, any smooth function $f$ on $X_{P_\Del}(F)$ is completely determined by its values $f(\Fs_a^{-1}(h,\RI))$ for all $g=(a,h)\in G(F)$. The extension of the Fourier operator
$\CF_{\rho,\psi}$ to the whole $\rho$-Schwartz space $\CS_{\rho}(G)$ is well-defined.

\begin{cor}\label{67-1}
For the Fourier operator $\CF_{\rho,\psi}$ as defined in Definition \ref{dfn:FTtoSG} for $\phi\in\CS_\rho(G)$, there exists a unique $f\in\CS_\pvs(X_{P_\Del})$ such that $\phi=\phi_f$ and $f=f_\phi$ via
\eqref{f-phi}, and
\[
\CF_{X,\psi}(f)(\Fs_a^{-1}(h,\RI))
=
|2|^{-2n(2n+1)}f_{\CF_{\rho,\psi}(\phi_f)}(\Fs_{2^{-2n}a}^{-1}(-h^{-1},\RI))
\]
for any $g=(a,h)\in G(F)$.
\end{cor}

Combining Theorem \ref{thm:FT-FT} and Corollary \ref{67-1} with Theorem \ref{thm:CFP-M}, we obtain the compatibility of the Fourier operator $\CF_{\rho,\psi}$ over $G(F)$ with the normalized local intertwining operator
$\RM^{\dag}_{w_\Del}(s,\chi,\psi)$ as given in \eqref{eq:IT-eta-rho}.

\begin{cor}[Compatibility of $\CF_{\rho,\psi}$ with $\RM^{\dag}_{w_\Del}(s,\chi,\psi)$]\label{cor:ComOnG}
For $h\in\Sp_{2n}(F)$ and $f\in\CS_{\pvs}(X_{P_{\Del}})$, $\CP_{\chi_s^{-1}}(f_{\CF_{\rho,\psi}(\phi_{f})})((-h^{-1},\RI))$ is well-defined for $\Re(s)$ sufficiently negative, and has the following identity:
$$
(\RM^{\dag}_{w_\Del}(s,\chi,\psi)\circ\CP_{\chi_s})(f)((h,\RI))
=
\CP_{\chi_s^{-1}}(f_{\CF_{\rho,\psi}(\phi_{f})})((-h^{-1},\RI)),
$$
which holds for all $s\in\BC$ by meromorphic continuation.
\end{cor}

\begin{proof}
By Definition \ref{dfn:FTtoSG} and Corollary \ref{67-1}, for $(a,h)\in G(F)$ and $f\in \CS_{\pvs}(X_{P_{\Del}})$, we have
\[
\CF_{X,\psi}(f)(\Fs_a^{-1}(h,\RI))
=
|2|^{-2n(2n+1)}f_{\CF_{\rho,\psi}(\phi_{f})}(\Fs_{2^{-2n}a}^{-1}(-h^{-1},\RI)).
\]
We first calculate $\CP_{\chi_s^{-1}}\circ\CF_{X,\psi}(f)(h,\RI_{2n})$. By Theorem \ref{thm:CFP-M}, whenever $\Re(s)$ is sufficiently negative, it is equal to the following absolutely convergent integral,
\begin{align}\label{FF-12}
&\int_{F^\times}\chi_s(t)^{-1}|t|^{\frac{2n+1}{2}}\CF_{X,\psi}(f)(\Fs_t^{-1}(h,\RI))\ud^*t\nonumber \\
&=\int_{F^\times}\chi_s(t)^{-1}|t|^{\frac{2n+1}{2}}
|2|^{-2n(2n+1)}f_{\CF_{\rho,\psi}(\phi_{f})}(\Fs_{2^{-2n}t}^{-1}(-h^{-1},\RI))\ud^*t.
\end{align}
Note that in Theorem \ref{thm:CFP-M}, the absolute convergence of the integral in \eqref{FF-12} for $\Re(s)$  sufficiently negative is proved only for functions in $\CF_{X,\psi}(\CC^{\infty}_{c}(X_{P_{\Del}}))$.
By Definition \ref{dfn:SSFX}, we have that
\[
\CS_{\pvs}(X_{P_{\Del}}) = \CC^{\infty}_{c}(X_{P_{\Del}})+\CF_{X,\psi}(\CC^{\infty}_{c}(X_{P_{\Del}})).
\]
Hence the absolute convergence of the integral in \eqref{FF-12} for $\Re(s)$  sufficiently negative holds for any $f\in \CS_{\pvs}(X_{P_{\Del}})$.

By changing variable $t\mapsto 2^{2n}t$ in the integral in \eqref{FF-12}, we deduce that the right-hand side of \eqref{FF-12} is equal to
\begin{align}\label{FF-13}
&\chi_{s}(2)^{-2n}|2|^{-n(2n+1)}\int_{F^\times}\chi_s(t)^{-1}|t|^{\frac{2n+1}{2}}f_{\CF_{\rho,\psi}(\phi_{f})}(\Fs_{t}^{-1}(-h^{-1},\RI))\ud^*x\nonumber\\
&\qquad\qquad\qquad=\chi_{s}(2)^{-2n}|2|^{-n(2n+1)}\CP_{\chi_s^{-1}}(f_{\CF_{\rho,\psi}(\phi_{f})})((-h^{-1},\RI)).
\end{align}
Hence $\CP_{\chi_s^{-1}}(f_{\CF_{\rho,\psi}(\phi_{f})})((-h^{-1},\RI))$ is well-defined for $\Re(s)$ sufficiently negative.
On the other hand, by Theorem \ref{thm:CFP-M}, we have that $\CP_{\chi_s^{-1}}\circ\CF_{X,\psi}(f)((h,\RI))$ is also equal to
\[
\chi_{s}(2)^{-2n}|2|^{-n(2n+1)}
(\RM^{\dag}_{w_\Del}(s,\chi,\psi)\circ\CP_{\chi_s})(f)((h,\RI)).
\]
By comparing with \eqref{FF-13}, we obtain the desired identity. We are done.
\end{proof}


\subsection{Plancherel theorem for $\CF_{\rho,\psi}$}\label{ssec-PT-CFG}

We start to prove Theorem \ref{thm:L2} here.
First we define 
\[
\CC_{c,X}^{\infty}(G): = \{ \phi_{f}\ |\ f\in \CC^{\infty}_{c}(X_{P_{\Del}}) \}, 
\]
which is a subspace of the $\rho$-Schwartz space $\CS_\rho(G)$, and prove the following lemma.

\begin{lem}\label{lem:FOG-Ss0}
For any $\phi\in \CC_{c,X}^{\infty}(G)$, the function $\CF_{\rho,\psi}(\phi)(a,h)$ belongs to the $\rho$-Schwartz space $\CS_\rho(G)$.
\end{lem}

\begin{proof}
It is equivalent to prove that for any $f\in \CC^{\infty}_{c}(X_{P_{\Del}})$, the function $\CF_{\rho,\psi}(\phi_{f})(a,h)$ belongs to $\CS_{\rho}(G)$.

By Theorem \ref{thm:FT-FT}, we have that
$$
\phi_{\CF_{X,\psi}(f)}(a,h) =
|2|^{-n(2n+1)}
\CF_{\rho,\psi}(\phi_{f})(2^{-2n}a,-h^{-1}).
$$
By Definition \ref{dfn:rho-SsG}, we have that $\phi_{\CF_{X,\psi}(f)}(a,h)\in \CS_{\rho}(G)$. Hence we obtain that
as a function in variable $(a,h)$, the function $\CF_{\rho,\psi}(\phi_{f})(a,-h^{-1})$ belongs to the space $\CS_{\rho}(G)$.

It remains to show that the function $\CF_{\rho,\psi}(\phi_{f})(a,h)$ belongs to the space $\CS_{\rho}(G)$. From Definition \ref{dfn-SDFTG}, we have
$$
\CF_{\rho,\psi}(\phi_{f})(a,-h^{-1}) =
\int^{\pv}_{F^{\times}}
\int_{\Sp_{2n}(F)}
\Phi_{\rho,\psi}(t,y)\phi^{\vee}_{f}(t^{-1}a,-y^{-1}h^{-1})\ud y\ud^{*}t.
$$
From the formula in \eqref{def:SDG}, we must have that $\Phi_{\rho,\psi}(a,h) = \Phi_{\rho,\psi}(a,h^{-1})$ for $h\in\Sp_{2n}(F)$. After changing variables: $y\to y^{-1}$ and $y\to hyh^{-1}$,
\begin{align}\label{Fphi-Fphi'}
\CF_{\rho,\psi}(\phi_{f})(a,-h^{-1})
&=
\int^{\pv}_{F^{\times}}
\int_{\Sp_{2n}(F)}
\Phi_{\rho,\psi}(t,y)\phi^{\vee}_{f}(t^{-1}a,-h^{-1}y)\ud y\ud^{*}t\nonumber
\\
&=
\int^{\pv}_{F^{\times}}
\int_{\Sp_{2n}(F)}
\Phi_{\rho,\psi}(t,y)\phi^{\p\vee}(t^{-1}a,y^{-1}h)\ud y\ud^{*}t\nonumber\\
&=\CF_{\rho,\psi}(\phi')(a,h),
\end{align}
where $\phi^{\p}(a,h) = \phi_{f}(a,-h^{-1})$.
It turns out to be sufficient to show that
for any $f\in \CC^{\infty}_{c}(X_{P_{\Del}})$, the function $\phi^{\p}(a,h) = \phi_{f}(a,-h^{-1})$ belongs to $\CC^{\infty}_{c,X}(G)$; or equivalently, there exists $f^{\p}\in \CC^{\infty}_{c}(X_{P_{\Del}})$
such that
\begin{align}\label{f'-f}
f^{\p}(\Fs^{-1}_{a}(h,\RI_{2n})) = f(\Fs^{-1}_{a}(-h^{-1},\RI_{2n}))
\end{align}
for any $(a,h)\in G(F)$. This implies that
\begin{align}\label{phi'-phi}
\phi_{f'}(a,h)=\phi'(a,h)=\phi_f(a,-h^{-1}).
\end{align}
To prove \eqref{f'-f}, we only need to find a homeomorphism from $X_{P_{\Del}}(F)$ to itself that sends $\Fs^{-1}_{a}(h,\RI_{2n})$ to $\Fs^{-1}_{a}(-h^{-1},\RI_{2n})$
for any $(a,h)\in G(F)$.

Following from Sections \ref{ssec-lzidm} and \ref{ssec-CT-J}, we consider the homeomorphism from $\Sp_{4n}(F)$ to itself:
\begin{align*}
\mathrm{inv}_{X_{P_{\Del}}}: \Sp_{4n}(F)&\to \Sp_{4n}(F)
\\
g&\to w_{\mathrm{swap}}g w_{\Del} w_{\mathrm{swap}}
\end{align*}
where
$
w_{\mathrm{swap}} = w_{\mathrm{swap}}^{-1}=
\left(
\begin{smallmatrix}
0 & \RI_{n} & 0 & 0\\
\RI_{n} & 0 & 0 & 0\\
0 & 0 & 0 & \RI_{n}\\
0 & 0 & \RI_{n} & 0\\
\end{smallmatrix}
\right)\in \Sp_{4n}(F)
$
and 
\[
w_{\Del} = (\RI_{2n},-\RI_{2n})\in \Sp_{2n}(F)\times \Sp_{2n}(F)
\] 
is as given in \eqref{wDel0}, which takes the parabolic subgroup $P_\Del$ to its opposite $P_\Del^-$.

We first show that $\mathrm{inv}_{X_{P_{\Del}}}$ descends to a homeomorphism from $X_{P_{\Del}}(F)$ to itself.
For any $p\in [P_{\Del},P_{\Del}]$ and $g\in \Sp_{4n}(F)$,
\begin{align*}
\inv_{X_{P_{\Del}}}(pg)
&= w_{\mathrm{swap}}pgw_{\Del}w_{\mathrm{swap}}
\\
&=
w_{\mathrm{swap}}pw_{\mathrm{swap}}
w_{\mathrm{swap}}gw_{\Del}w_{\mathrm{swap}}
\\
&=
w_{\mathrm{swap}}pw_{\mathrm{swap}}
\inv_{X_{P_{\Del}}}(g).
\end{align*}
It is enough to show that $w_{\mathrm{swap}}[P_{\Del},P_{\Del}]w_{\mathrm{swap}} = [P_{\Del},P_{\Del}]$.
By taking $g_0$ and $g_0^{-1}$ as in \eqref{g0} and \eqref{g0-1} respectively, we have that
$$
g^{-1}_{0} P_{\std}(F) g_{0} = P_{\Del}(F).
$$
By a straightforward computation, we obtain that
$$
w_{\mathrm{swap}}g^{-1}_{0} =
g^{-1}_{0}
\left(
\begin{smallmatrix}
\RI_{n} & & &\\
	& -\RI_{n} & &\\
	& & \RI_{n} & \\
	& &	& -\RI_{n}
\end{smallmatrix}
\right).
$$
Hence we obtain that
\begin{align*}
w_{\mathrm{swap}}[P_{\Del},P_{\Del}]w_{\mathrm{swap}}
&= w_{\mathrm{swap}}
g^{-1}_{0}
[P_{\std},P_{\std}]
g_{0}
w_{\mathrm{swap}}
\\
&=g^{-1}_{0}
\left(
\begin{smallmatrix}
\RI_{n} & & &\\
	& -\RI_{n} & &\\
	& & \RI_{n} & \\
	& &	& -\RI_{n}
\end{smallmatrix}
\right)
[P_{\std},P_{\std}]
\left(
\begin{smallmatrix}
\RI_{n} & & &\\
	& -\RI_{n} & &\\
	& & \RI_{n} & \\
	& &	& -\RI_{n}
\end{smallmatrix}
\right)
g_{0}
\\
&=g^{-1}_{0}
[P_{\std},P_{\std}]g_{0} =[P_{\Del},P_{\Del}].
\end{align*}
It follows that $\inv_{X_{P_{\Del}}}$ descends to a homeomorphism from $X_{P_{\Del}}(F)$ to itself.

It remains to show that $\mathrm{inv}_{X_{P_{\Del}}}$ sends $\Fs^{-1}_{a}(h,\RI_{2n})$ to $\Fs^{-1}_{a}(-h^{-1},\RI_{2n})$ for any $(a,h)\in G(F)$.

Since $\Sp_{2n}(F)\times \Sp_{2n}(F)$ acts on $X_{P_{\Del}}(F)$ with $\Sp_{2n}(F)^{\Del}$ as the stabilizer at the point $[P_\Del,P_\Del]$, we only need to show that
$\mathrm{inv}_{X_{P_{\Del}}}$ sends $\Fs^{-1}_{a}(h,\RI_{2n})$ to $\Fs^{-1}_{a}(\RI_{2n},-h)$ for any $(a,h)\in G(F)$.
But this can be deduced from the explicit description of the embedding \eqref{embsp}. More precisely, by a straightforward computation,
$w_{\Del}$ sends $(\RI_{2n},h)$ to $(\RI_{2n},-h)$ for any $h\in \Sp_{2n}(F)$, and the conjugation action by $w_{\mathrm{swap}}$ swaps $\Sp_{2n}(F)\times \{ \RI_{2n} \}$ and $\{ \RI_{2n} \}\times \Sp_{2n}(F)$, i.e. $w_{\mathrm{swap}}(h_1,h_2)w_{\mathrm{swap}} = (h_2,h_1)$ for any $(h_1,h_2)\in \Sp_{2n}(F)\times \Sp_{2n}(F)$. We are done.
\end{proof}

Now we are ready to prove the following result.

\begin{thm}[Plancherel Theorem]\label{thm:Pl-FOG}
The Fourier operator
$\CF_{\rho,\psi}$ extends to a unitary operator on $L^2(G,\ud^{*}x\ud h)$ and has the property that $\CF_{\rho,\psi^{-1}}\circ \CF_{\rho,\psi}=\Id$.
\end{thm}

\begin{proof}
We first prove that $\|\CF_{\rho,\psi}(\phi)\|^2=\|\phi\|^2$ for any $\phi\in\CC_{c,X}^\infty(G)$.
For any $f\in \CC^{\infty}_{c}(X_{P_{\Del}})$, from Proposition \ref{pro:Fx-norm}, we have that
$$
\|2^{n(2n+1)}\CF_{X,\psi}(f)\|  = \|f\|,
$$
where $\|\cdot\|$ is the $L^{2}$-norm of the space $L^{2}(X_{P_{\Del}})$ as defined in \eqref{eq:L2onXp}:
$$
\|f\|^{2} =
\int_{K_\Del}
\int_{F^{\times}}
|f(\Fs_{x}k)|^{2}\del_{P_{\Del}}(\Fs_{x})^{-1}\ud^{*}x\ud k.
$$
After changing variable $x\to x^{-1}$, we can rewrite the integral as
\begin{align*}
\|f\|^{2} &=
\int_{K_\Del}
\int_{F^{\times}}
|f(\Fs^{-1}_{x}k)|^{2}\del_{P_{\Del}}(\Fs_{x})\ud^{*}x\ud k
\\
&=
\int_{K_\Del}
\int_{F^{\times}}
|f(\Fs^{-1}_{x}k)|x|^{\frac{2n+1}{2}}|^{2}\ud^{*}x\ud k.
\end{align*}
We make a comment on the relevant measures here.
Up to constant,
there is a unique $\Sp_{4n}(F)$-invariant measure on $P_{\Del}(F)\bs \Sp_{4n}(F)$, which can be obtained via the constant sections of the canonical bundle
$\bigwedge^{\max}\Ome_{P_{\Del}(F)\bs \Sp_{4n}(F)}$. From the Iwasawa decomposition, $P_{\Del}(F)\bs \Sp_{4n}(F)$ is a $K_\Del$-homogeneous variety,
and the Haar measure ${\rm d} k$ on $K_\Del$ can induce the $\Sp_{4n}(F)$-invariant measure on $P_{\Del}(F)\bs \Sp_{4n}(F)$. On the other hand,
as $\Sp_{2n}(F)\times \{ \RI\}$ embeds as an open dense subset into $P_{\Del}(F)\bs \Sp_{4n}(F)$, the Haar measure ${\rm d} h$ on $\Sp_{2n}(F)\times \{ \RI\}$ can also induce
the $\Sp_{4n}(F)$-invariant measure on $P_{\Del}(F)\bs \Sp_{4n}(F)$. Therefore there exists a nonzero constant, which we denote by $c_{K_\Del}$, such that
$$
\ud k=c_{K_\Del}\cdot\ud h
$$
as measures on $P_{\Del}(F)\bs \Sp_{4n}(F)$.
Hence we can write $\|f\|^{2}$ as
$$
\|f\|^{2} =
c_{K_\Del}
\int_{\Sp_{2n}(F)}
\int_{F^{\times}}
|f(\Fs^{-1}_{x}\cdot(h,\RI))|x|^{\frac{2n+1}{2}}|^{2}
\ud^{*}x\ud h.
$$
From the relation in \eqref{f-phi}, we have that $f(\Fs_x^{-1}\cdot (h,\RI))|x|^{\frac{2n+1}{2}} = \phi_{f}(x,h)$. Hence we obtain that
$$
\|f\|^{2} =
c_{K_\Del}
\int_{\Sp_{2n}(F)}
\int_{F^{\times}}
|\phi_{f}(x,h)|^{2}\ud^{*}x\ud h.
$$
It follows that the $\phi_f$ belongs to the space $L^2(G,\ud^*x\ud h)$.

On the other hand, from \eqref{phi-CF-f}, for any $f\in \CC^{\infty}_{c}(X_{P_{\Del}})$, the following identity
$$
\phi_{2^{n(2n+1)}\CF_{X,\psi}(f)}(a,h) =
|2|^{n(2n+1)}\phi_{\CF_{X,\psi}(f)}
(a,h) =\CF_{\rho,\psi}
(\phi_{f})(2^{-2n}a,-h^{-1})
$$
holds for any $(a,h)\in G(F)$. Hence we obtain that
\begin{align*}
\|2^{n(2n+1)}\CF_{X,\psi}(f)\|^{2}
&=
c_{K_\Del}
\int_{\Sp_{2n}(F)}
\int_{F^{\times}}
|2^{n(2n+1)}\phi_{\CF_{X,\psi}(f)}(x,h)|^{2}\ud^{*}x\ud h
\\
&=
c_{K_\Del}
\int_{\Sp_{2n}(F)}
\int_{F^{\times}}
|\CF_{\rho,\psi}(\phi_{f})(2^{-2n}x,-h^{-1})|^{2}\ud^{*}x\ud h.
\end{align*}
After changing variables $x\to 2^{2n}x$ and $h\to -h^{-1}$, we get that
\begin{align*}
\|2^{n(2n+1)}\CF_{X,\psi}(f)\|^{2}
=
c_{K_\Del}
\int_{\Sp_{2n}(F)}
\int_{F^{\times}}
|\CF_{\rho,\psi}(\phi_{f})(x,h)|^{2}\ud^{*}x\ud h.
\end{align*}
From Proposition \ref{pro:Fx-norm} that $\|2^{n(2n+1)}\CF_{X,\psi}(f)\| = \|f\|$, we get the desired identity
$$
\int_{\Sp_{2n}(F)}
\int_{F^{\times}}
|\phi_{f}(x,h)|^{2}\ud^{*}x\ud h
=
\int_{\Sp_{2n}(F)}
\int_{F^{\times}}
|\CF_{\rho,\psi}(\phi_{f})(x,h)|^{2}\ud^{*}x\ud h.
$$
for any $f\in\CC_c^\infty(X_{P_\Del})$. In other words, we obtain that $\|\CF_{\rho,\psi}(\phi)\|=\|\phi\|$ for any $\phi\in\CC_{c,X}^\infty(G)$. By Theorem \ref{thm:asymSf}, it is clear that the space
$\CC_c^\infty(G)$ is contained in $\CC_{c,X}^\infty(G)$. As just proved above, the space  $\CC_{c,X}^\infty(G)$ is contained in $L^{2}(G, \ud^{*}x\ud h)$. Therefore we obtain that
\[
\|\CF_{\rho,\psi}(\phi)\|=\|\phi\|
\]
holds for all $\phi\in L^{2}(G, \ud^{*}x\ud h)$.

Now we prove the second part that $\CF_{\rho,\psi^{-1}}\circ \CF_{\rho,\psi}=\Id$.
By Corollary \ref{Fx-invs}, for any $f\in \CC^{\infty}_{c}(X_{P_{\Del}})$, we have 
\[
\CF_{X,\psi^{-1}}\circ \CF_{X,\psi}(f) = |2|^{-2n(2n+1)}f.
\]
By Definition \ref{dfn:FTtoSG} and \eqref{CF-f}, we have that
\begin{align*}
|2|^{-2n(2n+1)}f(\Fs^{-1}_{a}(h,\RI))
&=
\CF_{X,\psi^{-1}}\circ \CF_{X,\psi}(f)
(\Fs^{-1}_{a}(h,\RI))
\\
&=
|2|^{-2n(2n+1)}
f_{\CF_{\rho,\psi^{-1}}(\phi_{\CF_{X,\psi}(f)})}(\Fs^{-1}_{2^{-2n}a}(-h^{-1},\RI)),
\end{align*}
for $(a,h)\in G(F)$.
By the transition relation in \eqref{f-phi}, we obtain that
\begin{align*}
\phi_{f}(a,h)
&=
|a|^{\frac{2n+1}{2}}
f_{\CF_{\rho,\psi^{-1}}(\phi_{\CF_{X,\psi}(f)})}(\Fs^{-1}_{2^{-2n}a}(-h^{-1},\RI))
\\
&=
|2|^{n(2n+1)}
\CF_{\rho,\psi^{-1}}(\phi_{\CF_{X,\psi}(f)})(2^{-2n}a,-h^{-1}).
\end{align*}
After changing variables $a\to 2^{2n}a$ and $h\to -h^{-1}$, we obtain that
\begin{equation}\label{eq:FGId-1}
|2|^{-n(2n+1)}\phi_{f}(2^{2n}a,-h^{-1})
=
\CF_{\rho,\psi^{-1}}
(\phi_{\CF_{X,\psi}(f)})(a,h).
\end{equation}
From \eqref{phi-CF-f}, we have that
\begin{equation}\label{eq:FGId-2}
\phi_{\CF_{X,\psi}(f)}(a,h) = |2|^{-n(2n+1)}
\CF_{\rho,\psi}(\phi_{f})(2^{-2n}a,-h^{-1}).
\end{equation}
By Definition \ref{dfn-SDFTG}, we write $\CF_{\rho,\psi}(\phi_{f})(2^{-2n}a,-h^{-1})$ as
\begin{align*}
\int_{F^{\times}}^{\pv}
\int_{\Sp_{2n}(F)}
\Phi_{\rho,\psi}(t,y)\phi^{\vee}_{f}
(t^{-1}2^{-2n}a,y^{-1}(-h^{-1}))\ud y \ud^{*}t.
\end{align*}
From Proposition \ref{pro:Phi}, we have that $\Phi_{\rho,\psi}(t,y) = \Phi_{\rho,\psi}(t,y^{-1})$ for any $y\in\Sp_{2n}(F)$.
After changing variable $y\to y^{-1}$ and $y\to h^{-1}yh$, we obtain that $\CF_{\rho,\psi}(\phi_{f})(2^{-2n}a,-h^{-1})$ is equal to
\begin{align*}
\int_{F^{\times}}^{\pv}
\int_{\Sp_{2n}(F)}
\Phi_{\rho,\psi}(t,y)\phi^{\vee}_{f}(t^{-1}2^{-2n}a,(-h^{-1}) y)\ud y \ud^{*}t.
\end{align*}
The function $\phi^{\vee}_{f}(t^{-1}2^{-2n}a,(-h^{-1}) y)$ can be calculated as follows:
\begin{align*}
\phi^{\vee}_{f}(t^{-1}2^{-2n}a,-h^{-1}y)=
\phi_f(2^{2n}ta^{-1},-y^{-1}h)=\phi_{\Fr_wf}(2^{2n}ta^{-1},y^{-1}h),
\end{align*}
where $w:=-w_{\Del} = (-\RI_{2n}, \RI_{2n})\in \Sp_{2n}(F)\times \Sp_{2n}(F)$ and $\Fr_wf$ is the right translation of $f$ by $w$.
Set $\phi^{\p}(a,h): = \phi_{\Fr_{w}f}(2^{2n}a,h^{-1})$ and obtain that
\[
\phi^{\vee}_{f}(t^{-1}2^{-2n}a,-h^{-1}y)={\phi'}^{\vee}(t^{-1}a,y^{-1}h).
\]
Putting it back the integral, we obtain that
\begin{align*}
\CF_{\rho,\psi}(\phi_{f})(2^{-2n}a,-h^{-1})
=&\int_{F^{\times}}^{\pv}
\int_{\Sp_{2n}(F)}
\Phi_{\rho,\psi}(t,y)\phi^{\p\vee}
(t^{-1}a,y^{-1}h) \ud y \ud^{*}t
\\
=&\CF_{\rho,\psi}(\phi^{\p})(a,h).
\end{align*}
Hence the equation in \eqref{eq:FGId-2} becomes the following equation:
$$
\phi_{\CF_{X,\psi}(f)}(a,h) = |2|^{-n(2n+1)}\CF_{\rho,\psi}(\phi^{\p})(a,h).
$$
Plugging the above formula into \eqref{eq:FGId-1}, we get that
$$
\phi_{f}(2^{2n}a,-h^{-1})
=
\CF_{\rho,\psi^{-1}}
\circ \CF_{\rho,\psi}(\phi^{\p})(a,h)
$$
for any $(a,h)\in G(F)$.
Because
\[
\phi_{f}(2^{2n}a,-h^{-1}) = \phi_{\Fr_{w}f}(2^{2n}a,h^{-1})=\phi'(a,h),
\]
we finally get that
\begin{align}\label{FF-I}
\phi^{\p} = \CF_{\rho,\psi^{-1}}\circ \CF_{\rho,\psi}(\phi^{\p}).
\end{align}
Since $\CC_c^\infty(G)\subset\CC_{c,X}^\infty(G)$, and for any $\phi\in\CC_c^\infty(G)$, one has that $\phi'(a,h)=\phi(2^{2n}a,-h^{-1})\in\CC_c^\infty(G)$. It implies that
the identity in \eqref{FF-I} holds for all $\phi'\in\CC_c^\infty(G)$. Therefore we obtain that
$$
\CF_{\rho,\psi^{-1}}\circ \CF_{\rho,\psi} = \Id
$$
as operators in the space $L^2(G,\ud^*x\ud h)$.
We are done.
\end{proof}

To complete the proof of Theorem \ref{thm:L2}, we prove the following result.

\begin{thm}\label{thm:FOG-Ss}
The $\rho$-Schwartz space $\CS_\rho(G)$ is stable under the Fourier operator $\CF_{\rho,\psi}$.
\end{thm}

\begin{proof}
For any $\phi\in\CS_\rho(G)$, via the transition relation in \eqref{f-phi}, there is a unique $f\in\CS_\pvs(X_{P_\Del})$ such that $\phi=\phi_f$ and $f=f_\phi$.
By Definition \ref{dfn:SSFX}, for any $f\in \CS_{\pvs}(X_{P_{\Del}})$, there exist  $f_{1},f_{2}\in \CC^{\infty}_{c}(X_{P_{\Del}})$, such that
$$
f= f_{1}+\CF_{X,\psi}(f_{2}).
$$
By linearity of \eqref{59-3} and \eqref{phi-CF-f},
$$
\phi_{f}(a,h) = \phi_{f_{1}}(a,h)+|2|^{-n(2n+1)}\CF_{\rho,\psi}(\phi_{f_{2}})(2^{-2n}a,-h^{-1})
$$
for any $(a,h)\in G(F)$. By Lemma \ref{lem:FOG-Ss0}, the function $\CF_{\rho,\psi}(\phi_{f_{1}})(a,h)$ belongs to the space $\CS_{\rho}(G)$.

For $f_{2}\in \CC^{\infty}_{c}(X_{P_{\Del}})$, write $\phi'(a,h)=\CF_{\rho,\psi}(\phi_{f_{2}})(a,-h^{-1})$.
It is enough to show that $\CF_{\rho,\psi}(\phi')(a,h)$ is well-defined and belongs to the space $\CS_\rho(G)$.
From the proof of Lemma \ref{lem:FOG-Ss0}, in particular, of the identities in \eqref{phi'-phi}, \eqref{f'-f}, and \eqref{Fphi-Fphi'}, there exists $f_{3}\in \CC^{\infty}_{c}(X_{P_{\Del}})$ such that
$$
\phi'(a,h)=\CF_{\rho,\psi}(\phi_{f_{2}})(a,-h^{-1}) = \CF_{\rho,\psi}(\phi_{f_{3}})(a,h).
$$
By Lemma \ref{lem:FOG-Ss0} again, the function $\phi'(a,h)$ belongs to the space $\CS_\rho(G)$, and hence $\CF_{\rho,\psi}(\phi')(a,h)$ is well-defined, by Definition \ref{dfn:FTtoSG} or Theorem \ref{thm:Pl-FOG},  and
\begin{align}\label{Fphi'}
\CF_{\rho,\psi}(\phi')(a,h)=
(\CF_{\rho,\psi}\circ\CF_{\rho,\psi})(\phi_{f_{3}})(a,h)
\end{align}
for any $(a,h)\in G(F)$. It remains to show that $(\CF_{\rho,\psi}\circ \CF_{\rho,\psi})(\phi_{f_{3}})(a,h)$ belongs to the $\rho$-Schwartz space $\CS_{\rho}(G)$, for any $f_{3}\in \CC^{\infty}_{c}(X_{P_{\Del}})$.
From \eqref{59-3} and \eqref{phi-CF-f}, we obtain that
$$
\iota\circ \CF_{\rho,\psi}(\phi_{f_{3}})(a,h) = \CF_{\rho,\psi^{-1}}(\phi_{f_{3}})(a,h)
$$
where $\iota\circ\phi_{f_{3}}(a,h)  = \phi_{f_{3}}(-a,h)$ for any $(a,h)\in G(F)$. Combining with Theorem \ref{thm:Pl-FOG}, we obtain that
$$
\CF_{\rho,\psi}\circ \CF_{\rho,\psi}(\phi_{f_{3}})(a,h) = \iota\circ \CF_{\rho,\psi^{-1}}\circ \CF_{\rho,\psi}(\phi_{f_{3}})(a,h) = \iota\circ \phi_{f_{3}}(a,h).
$$
Since $\iota$ stabilizes the $\rho$-Schwartz space $\CS_{\rho}(G)$, we obtain that
the function $\CF_{\rho,{\psi}}\circ \CF_{\rho,\psi}(\phi_{f_{3}})(a,h)$ belongs to the space $\CS_{\rho}(G)$, and so does the function $\CF_{\rho,\psi}(\phi_{f})(a,h)$. We are done.
\end{proof}

This finishes our proof of Theorem \ref{thm:L2} and hence completes our verification of Conjecture 5.4 of \cite{BK00} for the case under consideration.


\section{Multiplicity One and Gamma Functions}\label{sec-MOGF}


We discuss the {\sl multiplicity one conjecture} (Conjecture \ref{cnj:mo}) and its relation to local gamma functions, following the line of the classical theory of Gelfand and Graev for
$\GL_1(F)$ (\cite{GG63}, \cite{GGPS}, and \cite{ST66}). This completes the theory of harmonic analysis and gamma functions.

\subsection{Multiplicity one}\label{ssec-mo}

As a key result of the functional analysis related to the $G(F)$-invariant distribution $\Phi_{\rho,\psi}$, the Fourier operator $\CF_{\rho,\psi}$ and the Schwartz space $\CS_\rho(G)$, one wishes to prove
the {\sl multiplicity one conjecture} (Conjecture \ref{cnj:mo}) for the case under consideration. For the purpose of this section, we prove a weaker
version of the multiplicity one result, which holds for general reductive groups over $F$.

Let $H$ be any reductive algebraic group defined over $F$. As before, let $\Pi(H)$ be the set of equivalence classes of irreducible admissible representations of $H(F)$.

\begin{lem}\label{lem:wmo}
Let $\CC^{\infty}_{c}(H)$ be the space of smooth, compactly supported functions on $H(F)$.
For $\sig_{1},\sig_{2}\in \Pi(H)$, the following inequality
$$
\dim \Hom_{H(F)\times H(F)}(\CC^{\infty}_{c}(H)\otimes \sig_1\otimes \sig_2,\BC)\leq 1
$$
holds. Moreover, equality holds if and only if $\sig_1\simeq \wt{\sig}_2$, where $\wt{\sig}_2$ is the contragredient of $\sig_2$.
\end{lem}

\begin{proof}
In this proof, we may use $H$ for $H(F)$ to simplify the notation.
First, one may identify the space $\CC^{\infty}_c(H)$ with the smooth and compact induction $\ind^{H\times H}_{H} (1_H)$, as representation of $H\times H$,
where $1_H$ is the trivial representation of $H$ and $H\hookrightarrow H\times H$ is the diagonal embedding.  Then one has that
\begin{align*}
\Hom_{H\times H}(\CC^{\infty}_c(H)\otimes \sig_{1}\otimes \sig_2,\BC)\simeq
\Hom_{H\times H}(\ind^{H\times H}_H(1_H)\otimes \sig_1\otimes \sig_2,\BC).
\end{align*}
Since the smooth dual of $\ind^{H\times H}_H(1_{H})$ is isomorphic to the smooth induction $\Ind^{H\times H}_H(1_H)$, we obtain that
$$
\Hom_{H\times H}(\ind^{H\times H}_H(1_H)\otimes \sig_1\otimes \sig_2,\BC)\simeq
\Hom_{H\times H}(\sig_1\otimes \sig_2, \Ind^{H\times H}_H(1_H)).
$$
By the Frobenius reciprocity, we deduce that
\begin{align*}
\Hom_{H\times H}(\sig_1\otimes \sig_2, \Ind^{H\times H}_H(1_H))
&\simeq
\Hom_H(\sig_1\otimes\sig_2,1_H)\\
&\simeq \Hom_H(\sig_1,\wt{\sig}_2).
\end{align*}
Finally, the result follows from Schur's lemma for the pair $(\sig_1,\wt{\sig}_2)$.
\end{proof}

\subsection{Functional equation and gamma function}\label{ssec-sfegf}

Now we take $H=G=\BG_m\times\Sp_{2n}$ as considered in this paper.
For $\chi\otimes\pi\in\Pi(G)$,  denote by $\vphi(a,h) = \chi(a)\langle w,\pi(h)v \rangle$ a matrix coefficient of $\chi\otimes\pi$,
associated to a pair of smooth vectors $v\in \pi,w\in \wt{\pi}$.
For $\phi\in\CS_\rho(G)$, as in \eqref{LZI-0} the local zeta integral is defined to be
\begin{equation}\label{lzi}
\CZ(s,\phi,\vphi) =\int_{F^{\times}\times \Sp_{2n}(F)} \phi(a,h)\vphi(a,h) |a|^{s-\frac{1}{2}}\ud^{*} a \ud h.
\end{equation}
Note that when $n=0$ we consider $\Sp_{2n}(F)$ as the trivial group and then $\CZ(s,\phi,\varphi)$ is the Tate-integral.

\begin{pro}\label{pro:zicm}
For any matrix coefficient $\vphi(a,h) = \chi(a)\langle w,\pi(h)v \rangle$ of $\chi\otimes\pi\in\Pi(G)$, and for any $\phi\in\CS_\rho(G)$,
the local zeta integral $\CZ(s,\phi,\vphi)$ as in \eqref{lzi} converges absolutely for $\Re(s)$ sufficiently positive, admits meromorphic continuation to $s\in \BC$
\end{pro}

\begin{proof}
By the definition of Schwartz space $\CS_\rho(G)$ (Definition \ref{dfn:rho-SsG}) and the asymptotics of functions in $\CS_\pvs(X_{P_\Del})$ (Theorem \ref{thm:asymSf}), it is standard to show that the local zeta
integral $\CZ(s,\phi,\vphi)$ as in \eqref{lzi} converges absolutely for $\Re(s)$ sufficiently positive and is a rational function in $q^{-s}$, and hence admits meromorphic continuation to $s\in \BC$. Another way to do this is
to make a connection of the zeta integrals in \eqref{lzi} with the Piatetski-Shapiro and Rallis zeta integrals, as in the proof of Theorem \ref{thm-lfe} below. We omit the details here.
\end{proof}

It is clear that when $\phi\in\CC_c^\infty(G)$, the zeta integral $\CZ(s,\phi,\vphi)$ as in \eqref{lzi} converges absolutely for any $s\in\BC$ and hence is entire as function in $s$. It follows that
the zeta integral $\CZ(s,\phi,\vphi)$ yields a non-zero linear functional in the $\Hom$-space:
\[
\Hom_{G(F)\times G(F)}(\CC^{\infty}_c(G)\otimes (\chi_{s-\frac{1}{2}}^{-1}\otimes\wt{\pi})\otimes(\chi_{s-\frac{1}{2}}\otimes\pi),\BC).
\]
\begin{pro}\label{pro:lfe}
For any matrix coefficient $\vphi(a,h) = \chi(a)\langle w,\pi(h)v \rangle$ of $\chi\otimes\pi\in\Pi(G)$, and for any $\phi\in\CC^\infty_c(G)$,
The local zeta integral $\CZ(1-s, \CF_{\rho,\psi}(\phi), \vphi^\vee)$ defines a non-zero linear functional in the $\Hom$-space:
\[
\Hom_{G(F)\times G(F)}(\CC^{\infty}_c(G)\otimes (\chi_{s-\frac{1}{2}}^{-1}\otimes\wt{\pi})\otimes(\chi_{s-\frac{1}{2}}\otimes\pi),\BC),
\]
via meromorphic continuation in $s\in\BC$, where $\CF_{\rho,\psi}$ is the Fourier operator defined in Definition \ref{dfn-SDFTG} and
${\vphi}^\vee(a,h)=\chi^{-1}(a)\langle w,\pi(h^{-1})v\rangle$ is the matrix coefficient associated to the contragredient $\chi^{-1}\otimes\wt{\pi}$ of $\chi\otimes\pi$.
Moreover, there exists a meromorphic function $\Gamma_{\rho,\psi}(s,\chi\otimes\pi)$ in $s$ such that the following functional equation:
$$
\CZ(1-s,\CF_{\rho,\psi}(\phi),\vphi^\vee)
=\Gamma_{\rho,\psi}(s,\chi\otimes\pi)
\cdot\CZ(s,\phi,\vphi)
$$
holds.
\end{pro}

\begin{proof}
For any $\phi\in\CC_c^\infty(G)$, by Theorem \ref{thm:FOG-Ss}, we must have that $\CF_{\rho,\psi}(\phi)$ belongs to the space $\CS_\rho(G)$. By Proposition \ref{pro:zicm}, the zeta integral
$\CZ(1-s,\CF_{\rho,\psi}(\phi),\vphi^\vee)$ converges absolutely for $\Re(s)$ sufficiently negative and admits a meromorphic continuation to $s\in\BC$. It is clear from the definition of the zeta integral in \eqref{lzi},
$\CZ(1-s,\CF_{\rho,\psi}(\phi),\vphi^\vee)$ yields a linear functional in the $\Hom$-space
\[
\Hom_{G(F)\times G(F)}(\CC^{\infty}_c(G)\otimes (\chi_{s-\frac{1}{2}}^{-1}\otimes\wt{\pi})\otimes(\chi_{s-\frac{1}{2}}\otimes\pi),\BC)
\]
for almost all $s$.
By Lemma \ref{lem:wmo}, the above $\Hom$-space is one-dimensional. Hence there exists a meromorphic function in $s$, which is denoted by $\Gamma_{\rho,\psi}(s,\chi\otimes\pi)$, such that
\[
\CZ(1-s,\CF_{\rho,\psi}(\phi),\vphi^\vee)
=\Gamma_{\rho,\psi}(s,\chi\otimes\pi)
\cdot\CZ(s,\phi,\vphi)
\]
holds.
\end{proof}


\subsection{Gamma function in the sense of Gelfand and Graev}\label{ssec-gfgg}

We are going to show that the gamma function $\Gamma_{\rho,\psi}(s,\chi\otimes\pi)$ is a gamma function in the sense of Gelfand and Graev (\cite{GG63}, \cite{GGPS}, and \cite{ST66}).

For any $\chi\otimes\pi\in\Pi(G)$ and any $s\in\BC$, we have a matrix coefficient
\[
\varphi_{\chi_s\otimes\pi}((a,h)):=\chi_s(a)\apair{w,\pi(h)v}
\]
of $\chi_s\otimes\pi$ associated to a pair of smooth vectors $v\in\pi$ and $w\in\wt{\pi}$. The matrix coefficient $\varphi_{\chi_s\otimes\pi}$ defines a distribution:
\begin{align}\label{75-1}
(\varphi_{\chi_s\otimes\pi},\phi):=\int_{G(F)}\varphi_{\chi_s\otimes\pi}(g)\phi(g)\ud^*g
\end{align}
for any $\phi\in\CC_c^\infty(G)$, where $\ud^*g=\ud^*a\ud h$ for $g=(a,h)\in G(F)$. Then the local zeta integral defined in \eqref{lzi} can be written as
\begin{align}\label{75-2}
\CZ(s,\phi,\varphi_{\chi\otimes\pi})=(\varphi_{\chi_{s-\frac{1}{2}}\otimes\pi},\phi).
\end{align}
Since $\phi\in\CC_c^\infty(G)$, the identities in \eqref{75-1} and \eqref{75-2} hold for all $s\in\BC$.
Similarly, we have
\begin{align}\label{75-3}
\CZ(1-s,\CF_{\rho,\psi}(\phi),\varphi_{\chi\otimes\pi}^\vee)=(\varphi_{\chi^{-1}_{s-\frac{1}{2}}\otimes\wt{\pi}},\CF_{\rho,\psi}(\phi)),
\end{align}
which converges for $\Re(s)$ sufficiently negative and can be extended to all $s\in\BC$ by meromorphic continuation of the local zeta integral (Proposition \ref{pro:lfe}).
By definition, the Fourier operator $\CF_{\rho,\psi}$ defined for $\phi\in\CC_c^\infty(G)$ induces a Fourier operator $\CF_{\rho,\psi}^*$ on distributions $\varphi_{\chi^{-1}_{s-\frac{1}{2}}\otimes\wt{\pi}}$
by the following formula:
\[
(\CF_{\rho,\psi}^*(\varphi_{\chi^{-1}_{s-\frac{1}{2}}\otimes\wt{\pi}}),\phi):=(\varphi_{\chi^{-1}_{s-\frac{1}{2}}\otimes\wt{\pi}},\CF_{\rho,\psi}(\phi)).
\]
Because of Proposition \ref{pro:lfe} again, the distribution $\CF_{\rho,\psi}^*(\varphi_{\chi^{-1}_{s-\frac{1}{2}}\otimes\wt{\pi}})$ is defined for $\Re(s)$ sufficiently negative
and admits a meromorphic continuation to all $s\in\BC$.
By applying the functional equation in Proposition \ref{pro:lfe} to the distributions in \eqref{75-2} and \eqref{75-3}, we obtain that
\[
(\CF_{\rho,\psi}^*(\varphi_{\chi^{-1}_{s-\frac{1}{2}}\otimes\wt{\pi}}),\phi)=\Gamma_{\rho,\psi}(s,\chi\otimes\pi)
\cdot (\varphi_{\chi_{s-\frac{1}{2}}\otimes\pi},\phi)
\]
for all $\phi\in\CC_c^\infty(G)$, via meromorphic continuation in $s$.
Hence we obtain the following result.

\begin{cor}[Gelfand-Graev Gamma Function]\label{cor:gfgg}
Let $\varphi_{\chi_s\otimes\pi}$ be a matrix coefficient of any irreducible admissible representation $\chi_s\otimes\pi$ of $G(F)$. Then $\varphi_{\chi_s\otimes\pi}^\vee=\varphi_{\chi_s^{-1}\otimes\wt{\pi}}$
is a matrix coefficient of the contragredient $\chi_s^{-1}\otimes\wt{\pi}$ of $\chi_s\otimes\pi$, and as distributions on $G(F)$, the following identity holds
\begin{align}\label{F-GGgamma}
\CF_{\rho,\psi}^*(\varphi_{\chi^{-1}_{s}\otimes\wt{\pi}})=\Gamma_{\rho,\psi}(\frac{1}{2},\chi_{s}\otimes\pi)
\cdot\varphi_{\chi_s\otimes\pi},
\end{align}
by meromorphic continuation. Moreover, the gamma function $\Gamma_{\rho,\psi}(s,\chi\otimes\pi)$
satisfies the following identity:
\begin{align}\label{gamma-id}
\Gamma_{\rho,\psi}(\frac{1}{2},\chi_{s}\otimes\pi)\cdot\Gamma_{\rho,\psi}(\frac{1}{2},\chi_{s}^{-1}\otimes\wt{\pi})
=1
\end{align}
as meromorphic functions in $s$.
\end{cor}

The identity in \eqref{gamma-id} can be deduced from the identity in \eqref{F-GGgamma} and the identity $\CF_{\rho,\psi}\circ\CF_{\rho,\psi^{-1}}=\Id$ (Theorem \ref{thm:Pl-FOG}) by considering
$(\CF^*_{\rho,\psi^{-1}}\circ\CF^*_{\rho,\psi})(\varphi_{\chi^{-1}_{s}\otimes\wt{\pi}})$.
We omit the details here.

It is clear that when $n=0$, this formula specializes to the classical formula (definition) for the gamma function on $F^\times$ of Gelfand-Graev (\cite{GG63}, \cite{GGPS}, and \cite{ST66}).
Hence Corollary \ref{cor:gfgg} proves that the gamma function $\Gamma_{\rho,\psi}(s,\chi\otimes\pi)$
is a gamma function in the sense of Gelfand and Graev.

When $\chi_s\otimes\pi$ is square-integrable, the matrix coefficient $\varphi_{\chi_s\otimes\pi}$ belongs to the space $L^2(G,\ud^*a\ud h)$. In this case, we have that
\[
(\varphi_{\chi_s\otimes\pi},\phi)=\apair{\phi,\ovl{\varphi}_{\chi_s\otimes\pi}}_G,
\]
where $\apair{\cdot,\cdot}_G$ is the hermitian inner product of the space $L^2(G,\ud^*a\ud h)$. By Theorem \ref{thm:Pl-FOG}, we have that
\[
(\varphi_{\chi_s\otimes\pi}^\vee,\CF_{\rho,\psi}(\phi))=\apair{\CF_{\rho,\psi}(\phi),\ovl{\varphi}_{\chi_s\otimes\pi}^\vee}_G=\apair{\phi,\CF_{\rho,\psi^{-1}}(\ovl{\varphi}^\vee_{\chi_s\otimes\pi})}_G.
\]
By Proposition \ref{pro:lfe}, we have
\[
(\varphi_{\chi_s\otimes\pi}^\vee,\CF_{\rho,\psi}(\phi))=\Gamma_{\rho,\psi}(\frac{1}{2},\chi_{s}\otimes\pi)
\cdot(\varphi_{\chi_s\otimes\pi},\phi).
\]
It follows that
\[
\apair{\phi,\CF_{\rho,\psi^{-1}}(\ovl{\varphi}^\vee_{\chi_s\otimes\pi})}_G=\Gamma_{\rho,\psi}(\frac{1}{2},\chi_{s}\otimes\pi)
\cdot\apair{\phi,\ovl{\varphi}_{\chi_s\otimes\pi}}_G.
\]
Hence we obtain that as functions in the space $L^2(G,\ud^*a\ud h)$ and also as distributions,
\[
\CF_{\rho,\psi^{-1}}(\ovl{\varphi}^\vee_{\chi_s\otimes\pi})=\Gamma_{\rho,\psi}(\frac{1}{2},\chi_{s}\otimes\pi)
\cdot\ovl{\varphi}_{\chi_s\otimes\pi}.
\]
By Corollary \ref{cor:gfgg}, we obtain that
\[
\CF_{\rho,\psi}^*(\varphi_{\chi_s\otimes\pi})=\CF_{\rho,\psi^{-1}}(\varphi_{\chi_s\otimes\pi})
\]
for any matrix coefficient $\varphi_{\chi_s\otimes\pi}$ of $\chi_s\otimes\pi$ when $\chi_s\otimes\pi$ is square-integrable.

\begin{cor}\label{cor:F*F-}
When an irreducible admissible representation $\chi_s\otimes\pi$ of $G(F)$ is square-integrable, as distributions and as functions in the space $L^2(G,\ud^*a\ud h)$, the identity
\[
\CF_{\rho,\psi}^*(\varphi_{\chi_s\otimes\pi})=\CF_{\rho,\psi^{-1}}(\varphi_{\chi_s\otimes\pi})
\]
holds for any matrix coefficient $\varphi_{\chi_s\otimes\pi}$ of $\chi_s\otimes\pi$.
\end{cor}

\subsection{$\chi_s\otimes\pi$-Fourier coefficient of $\Phi_{\rho,\psi}$}\label{ssec-pf-Thm13}

We first consider a natural topology on the space $\CC^\infty_c(G)$. For any non-empty open compact subset $\Omega$ of $G(F)$, and any open compact subgroup $\CJ$ of $G(F)$, define
$\CV_{\Omega,\CJ}$ to be a subspace of $\CC^\infty_c(G)$ consisting of functions $\phi\in\CC_c^\infty(G)$ that are bi-$\CJ$-invariant and have support $\supp(\phi)$ contained in $\Omega$. It is easy to check that
$\CV_{\Omega,\CJ}$ is finite-dimensional. The family $\{\CV_{\Omega,\CJ}\}_{\Omega,\CJ}$ yields a topology $\CT$ on $\CC^\infty_c(G)$. A sequence $\phi_n$ in $\CC^\infty_c(G)$ converges to $\phi\in\CC_c^\infty(G)$ under
the topology $\CT$ if there are a pair $(\Omega,\CJ)$ such that $\phi_n$ and $\phi$ belong to the space $\CV_{\Omega,\CJ}$ for $n$ large and $\phi_n$ converges to $\phi$ in the space $\CV_{\Omega,\CJ}$. A distribution
$D$ is a continuous linear functional of the space $\CV_{\Omega,\CJ}$. A generalized function $\Phi(g)$ on $G(F)$ can be viewed as a distribution via
\[
D_\Phi(\phi):=\int_{G(F)}\Phi(g)\phi(g)\ud^* g=(\Phi*\phi^\vee)(e)
\]
where $e$ is the identity element of $G(F)$ and $\ud^*g=\ud^*a\ud h$ for $g=(a,h)\in G(F)$. A distribution $D$ on $G(F)$ is called {\sl essentially compact} if for any $\phi\in\CC^\infty_c(G)$, both
$D(\Fl_g\cdot\phi^\vee)$ and $D(\Fr_{g^{-1}}\cdot\phi^\vee)$ belong to $\CC^\infty_c(G)$. Here $\phi^\vee(x)=\phi(x^{-1})$,
the left translation is given by $\Fl_g\cdot\phi(x)=\phi(g^{-1}x)$ and right translation is given by $\Fr_g\cdot\phi(x)=\phi(xg)$. A distribution $D$ is {\sl $G(F)$-invariant} if for any $g\in G(F)$,
\begin{align}\label{D-inv}
D((\Fl_g\circ\Fr_g)\cdot\phi)=D(\phi)
\end{align}
holds for any $\phi\in\CC^\infty_c(G)$. By the definition in \cite{Ber84} (see also \cite{MT07}), the {\sl Bernstein center} of $G(F)$, which is denoted by $\CZ(G)$,
consists of all $G(F)$-invariant essentially compact distributions on $G(F)$.

If a distribution $\Phi$ is essentially compact, one can define the $\chi_s\otimes\pi$-Fourier coefficient of $\Phi$ by
\[
(\chi_s\otimes\pi)(\Phi):=\int_{G(F)}\Phi(g)(\chi_s\otimes\pi)(g)\ud^*a\ud h
\]
where $g=(a,h)\in G(F)$.
From Definition \ref{dfn-SDFT}, the distribution $\Phi_{\rho,\psi}$ has the following expression:
$$
\Phi_{\rho,\psi}(a,h):=c_{0}
\cdot \eta_{\pvs, \psi}
(a\det(h+\RI_{2n}))
\cdot
|\det(h+\RI_{2n})|^{-\frac{2n+1}{2}}.
$$
Since the distribution $\Phi_{\rho,\psi}$ on $G(F)$ is {\sl not} essentially compact,
we have to define rigorously the convolution $(\chi_{s}\otimes \pi)(\Phi_{\rho,\psi})$ or the $\chi_s\otimes\pi$-Fourier coefficient of the distribution $\Phi_{\rho,\psi}$ on $G(F)$
for any $\chi_s\otimes\pi\in\Pi(G)$.
To this end,
we define, for any $\ell\in \BZ$,
\begin{equation}\label{73-2}
G_{\ell}:= \{ (a,h)\in G(F) = F^{\times}\times \Sp_{2n}(F) \mid\ |a|  =q^{-\ell}\}.
\end{equation}
Let $\mathbf{1}_{\ell}$ be the characteristic function of the open subset $G_{\ell}$ of $G(F)$. We first consider the cut-off function $\Phi_{\rho,\psi}\cdot \mathbf{1}_{\ell}$.

\begin{pro}\label{pro:Phiell}
For any $\ell\in \BZ$, the distribution $\Phi_{\rho,\psi,\ell} := \Phi_{\rho,\psi}\cdot \mathbf{1}_{\ell}$
belongs to the Bernstein center $\CZ(G)$ of $G(F)$.
\end{pro}

\begin{proof}
It is clear from the definition that for any $\ell\in \BZ$, the distribution $\Phi_{\rho,\psi,\ell}$ is $G(F)$-invariant, because of the $G(F)$-invariance of $\Phi_{\rho,\psi}$.
We remain to show that $\Phi_{\rho,\psi,\ell}$ is essentially compact on $G(F)$.

For any open compact subgroup $K$ of $G(F)$, let $\ch_K$ be the characteristic function of $K$. For any $g_{1},g_{2}\in G(F)$, we have
\begin{align*}
\Phi_{\rho,\psi,\ell}*\mathrm{ch}_{g_{1}Kg_{2}}(g)
&=
\int_{G(F)}
\Phi_{\rho,\psi,\ell}(y)\mathrm{ch}_{g_{1}Kg_{2}}(y^{-1}g)\ud y
\\
&=\int_{G(F)}
\Phi_{\rho,\psi,\ell}(y)\mathrm{ch}_{K}(g_{1}^{-1}y^{-1}gg_{2}^{-1})\ud y.
\end{align*}
Note that $y^{-1}g\in g_{1}Kg_{2}$ if and only if $g^{-1}_{1}y^{-1}gg_{2}^{-1}\in K$. Since the distribution $\Phi_{\rho,\psi,\ell}$ is $G(F)$-invariant,
after changing variable $y\to g_{1}yg_{1}^{-1}$, we get
\begin{align}\label{73-3}
\Phi_{\rho,\psi,\ell}*\mathrm{ch}_{g_{1}Kg_{2}}(g)
& =
\int_{G(F)}
\Phi_{\rho,\psi,\ell}(y)\mathrm{ch}_K(y^{-1}g_1^{-1}gg_2^{-1})\ud y\nonumber\\
&=\Phi_{\rho,\psi,\ell}*\mathrm{ch}_{K}(g_{1}^{-1}gg_{2}^{-1}).
\end{align}
It is clear that $\Phi_{\rho,\psi,\ell}*\mathrm{ch}_{g_{1}Kg_{2}}\in \CC^{\infty}_{c}(G)$ if and only if $\Phi_{\rho,\psi,\ell}*\mathrm{ch}_{K}\in \CC^{\infty}_{c}(G)$. Hence it is enough to show that
$\Phi_{\rho,\psi,\ell}*\mathrm{ch}_{K}\in \CC^{\infty}_{c}(G)$ for any open compact subgroup $K$ of the maximal open compact subgroup 
$$
K_G = \CO^\times \times \Sp_{2n}(\CO)
$$ of $G(F)$.
From the identity in \eqref{73-3}, we take $g_1=k_1, g_2=k_2\in K$, and obtain
\[
\Phi_{\rho,\psi,\ell}*\mathrm{ch}_{K}(g)
=
\Phi_{\rho,\psi,\ell}*\mathrm{ch}_{k_{1}Kk_{2}}(g)
=
\Phi_{\rho,\psi,\ell}*\mathrm{ch}_{K}(k_{1}^{-1}gk_{2}^{-1}).
\]
It follows that the function $\Phi_{\rho,\psi,\ell}*\mathrm{ch}_{K}(g)$ is bi-$K$-invariant, and hence is smooth on $G(F)$.

It remains to show that the function $\phi_{\ell,K}(g):=\Phi_{\rho,\psi,\ell}*\mathrm{ch}_{K}(g)$ is compactly supported on $G(F)$. In order to apply Theorem \ref{thm:asymSf} to the current situation, we have to use
the transition relation in \eqref{f-phi}. This means that the $\ch_K(a,h)$ on $G(F)$ defines a function
$$
f_{\ch_K}(\Fs_a^{-1}(h,\RI))=|a|^{-\frac{2n+1}{2}}\ch_K(a,h)=\ch_K(a,h).
$$
because of $K=K_1\times K_2$ with $K_1$ an open compact subgroup of $\CO^\times$ and $K_2$ an open compact subgroup of $\Sp_{2n}(\CO)$. Hence $f_{\ch_K}(x)\in\CC^\infty_c(X_{P_\Del})$. By Theorem \ref{thm:asymSf},
the function $f_{\ch_K}$ is right $K_\Del$-finite.
We write
\begin{align}\label{73-4}
\phi_{\ell,K}(g)=\Phi_{\rho,\psi,\ell}*\mathrm{ch}_{K}(g)
=
\int_{G_\ell}
\Phi_{\rho,\psi}(y)\mathrm{ch}_{K}(y^{-1}g)dy.
\end{align}
By writing $g=(a,h)$ and $y=(t,x)$, we obtain that $(t^{-1}a,x^{-1}h)\in K=K_1\times K_2$. We must obtain that
$|a|=|t|=q^{-\ell}$. Hence we obtain that $\supp(\phi_{\ell,K})$ is contained in $G_\ell$.

By \eqref{def:FTG}, we write
\begin{align*}
\CF_{\rho,\psi}(\mathrm{ch}_{K})
&=
\Phi_{\rho,\psi}*\mathrm{ch}_{K}^{\vee} =
\Phi_{\rho,\psi}*\mathrm{ch}_{K}
=
\sum_{\ell\in \BZ}\Phi_{\rho,\psi,\ell}*\mathrm{ch}_{K} = \sum_{\ell\in \BZ}\phi_{\ell,K}.
\end{align*}
This decomposition is well-defined because $\phi_{\ell_1,K}$ and $\phi_{\ell_2,K}$ have disjoint supports if $\ell_1\neq \ell_2$.
Therefore we can re-write
$$
\phi_{\ell,K} = \mathbf{1}_{\ell}\cdot \CF_{\rho,\psi}(\mathrm{ch}_{K}).
$$
From Definition \ref{dfn:rho-SsG}, $\CF_{\rho,\psi}(\mathrm{ch}_{K})$ can be extended to a smooth function on $X_{P_{\Del}}(F)$ via the translation relation \eqref{f-phi}. As the characteristic function of the open (and closed) subset $G_{\ell}\subset G(F)$ of $X_{P_{\Del}}(F)$, $\mathbf{1}_{\ell}$ can be extended to a smooth function on $X_{P_{\Del}}(F)$ as well.
By using the transition relation \eqref{f-phi} again, we have
\[
f_{\phi_{\ell,K}}(\Fs_a^{-1}(h,\RI))=\phi_{\ell,K}(a,h)q^{-{\frac{(2n+1)\ell}{2}}} =\mathbf{1}_{\ell}\cdot \CF_{\rho,\psi}(\mathrm{ch}_{K})(a,h)q^{-\frac{(2n+1)\ell}{2}},
\]
which can also be extended as a smooth function on $X_{P_\Del}(F)$.
From \eqref{73-4}, the function $f_{\phi_{\ell,K}}$ is $K_\Del$-finite. Since we just proved that $f_{\phi_{\ell,K}}(\Fs_a^{-1}k)$ belongs to $\CC^\infty_c(F^\times)$ for any fixed $k\in K_\Del$,
by Theorem \ref{thm:asymSf} again, the function $f_{\phi_{\ell,K}}$ belongs to $\CC^\infty_c(X_{P_\Del})$. Because the support $\supp(\phi_{\ell,K})$ is contained in $G_\ell$,
which is open in $G(F)$, we obtain that $\supp(\phi_{\ell,K})$ is compact in $G(F)$.
This finishes the proof.
\end{proof}

As in the proof of Proposition \ref{pro:Phiell}, we write
\begin{align}\label{Phi-Phiell}
\Phi_{\rho,\psi}  =
\sum_{\ell\in \BZ}
\Phi_{\rho,\psi,\ell},
\end{align}
which is well-defined,
because the family of distributions $\{\Phi_{\rho,\psi,\ell} \}_{\ell\in \BZ}$ have disjoint supports.
From Proposition \ref{pro:Phiell}, for any $\ell\in \BZ$ the distribution $\Phi_{\rho,\psi,\ell}$ belongs to the Bernstein center $\CZ(G)$ of $G(F)$.
By \cite[Theorem~2.13]{Ber84}, the action of $\Phi_{\rho,\psi,\ell}$ on $\chi\otimes \pi$ is well-defined, and there exists a regular function $f_{\ell}$ on the Bernstein variety of $G(F)$ such that
\begin{align}\label{73-5}
(\chi\otimes \pi)(\Phi_{\rho,\psi,\ell}) = f_{\ell}(\chi\otimes \pi)\Id_{\chi\otimes \pi}
\end{align}
for any $\chi\otimes\pi\in\Pi(G)$.

For any $\phi\in \CC^{\infty}_{c}(G)$, from Proposition \ref{pro:lfe}, the following functional equation holds after meromorphic continuation
\begin{equation}\label{lfe-73}
\CZ(1-s,\CF_{\rho,\psi}(\phi),\vphi^{\vee})
=\Gamma_{\rho,\psi}(s,\chi\otimes\pi)
\cdot\CZ(s,\phi,\vphi),
\end{equation}
for $\vphi\in \CC(\chi\otimes \pi)$, the space of matrix coefficients of $\chi\otimes \pi$.
Since $\phi\in \CC^{\infty}_{c}(G)$, the local zeta integral introduced in \eqref{lzi}
$$
\CZ(s,\phi,\vphi) = \int_{F^{\times}\times \Sp_{2n}(F)}
\phi(a,h)\vphi(a,h)|a|^{s-\frac{1}{2}}\ud^{*}a\ud h
$$
is absolutely convergent and holomorphic for any $s\in \BC$, and also in
Definition \ref{dfn-SDFT}, the principal value integral that defines
$$
\CF_{\rho,\psi}(\phi)(a,h) = \Phi_{\rho,\psi}*\phi^{\vee}(a,h),
$$
can be replaced by the absolutely convergent integral.

When $\Re(s)$ is sufficiently negative, the left-hand side of \eqref{lfe-73} is absolutely convergent. From \eqref{73-4}, for $\ell\in \BZ$, the functions
$$
\Phi_{\rho,\psi,\ell}*\phi^{\vee} = \mathbf{1}_{\ell}\cdot \Phi_{\rho,\psi,\ell}*\phi^{\vee}
$$
have disjoint supports. Therefore the following identity
\begin{align}\label{zetaell}
\CZ(1-s,\CF_{\rho,\psi}(\phi),\vphi^{\vee}) =
&
\CZ(1-s, \Phi_{\rho,\psi}*\phi^{\vee},\vphi^{\vee})\nonumber
\\
=
&
\CZ(1-s,\sum_{\ell\in \BZ}\Phi_{\rho,\psi,\ell}*\phi^{\vee},\vphi^{\vee})\nonumber
\\
=
&
\sum_{\ell\in \BZ}
\CZ(1-s,\Phi_{\rho,\psi,\ell}*\phi^{\vee},\vphi^{\vee})
\end{align}
holds whenever $\Re(s)$ is sufficiently negative, and the summations are always absolutely convergent. We are going to prove the following lemma by using
the argument in the proof of \cite[Lemma~2.4.4]{Luo19}, which is now rigorous since the distribution $\Phi_{\rho,\psi,\ell}$ belongs to the Bernstein center $\CZ(G)$ of $G(F)$.

\begin{lem}\label{lem:Phiell}
For any $\phi\in \CC^{\infty}_{c}(G)$, the following identity
$$
\CZ(1-s, \Phi_{\rho,\psi,\ell}*\phi^{\vee},\vphi^{\vee}) = f_{\ell}(\chi^{-1}_{s-\frac{1}{2}}\otimes \wt{\pi})\cdot
\CZ(s,\phi,\vphi)
$$
holds for any $s\in \BC$, and for any $\ell\in\BZ$.
\end{lem}
\begin{proof}
For $\phi\in \CC^{\infty}_{c}(G)$, since $\Phi_{\rho,\psi,\ell}$ is essentially compact, the function $\Phi_{\rho,\psi,\ell}*\phi^{\vee}$ belongs to $\CC^{\infty}_{c}(G)$.
Hence the following integral
\begin{align*}
\CZ(1-s,\Phi_{\rho,\psi,\ell}*\phi^{\vee},\vphi^{\vee})
=
\int_{G(F)}
\Phi_{\rho,\psi,\ell}*\phi^{\vee}
(a,h)
\vphi^{\vee}(a,h)
|a|^{\frac{1}{2}-s}\ud^{*}a\ud h
\end{align*}
converges absolutely for any $s\in \BC$.
Up to a finite linear combination, we may write
$\vphi(g) = \chi(a)\langle w,\pi(h)v\rangle$ for a pair of fixed vectors $v\in\pi$ and $w\in\wt{\pi}$. The local zeta integral $\CZ(1-s,\Phi_{\rho,\psi,\ell}*\phi^{\vee},\vphi^{\vee})$ can be written as
\[
\int_{G(F)}
\Phi_{\rho,\psi,\ell}*\phi^{\vee}(a,h)\chi^{-1}(a)\langle w,\pi(h^{-1})v\rangle |a|^{\frac{1}{2}-s}
\ud^{*}a\ud h
\]
which is equal to
\[
\int_{G(F)}
\Phi_{\rho,\psi,\ell}*\phi^{\vee}(a,h)\chi_{s-\frac{1}{2}}^{-1}(a)\langle\wt{\pi}(h)w,v\rangle
\ud^{*}a\ud h.
\]
Hence the local zeta integral $\CZ(1-s,\Phi_{\rho,\psi,\ell}*\phi^{\vee},\vphi^{\vee})$ can be written as
\begin{align}\label{73-6}
\langle \int_{G(F)}
\Phi_{\rho,\psi,\ell}*\phi^{\vee}(a,h)\chi_{s-\frac{1}{2}}^{-1}(a)\wt{\pi}(h)w
\ud^{*}a\ud h,v \rangle.
\end{align}
Note that
\begin{align}
\int_{G(F)}
\Phi_{\rho,\psi,\ell}*\phi^{\vee}(a,h)\chi_{s-\frac{1}{2}}^{-1}(a)\wt{\pi}(h)
\ud^{*}a\ud h
&=
(\chi^{-1}_{s-\frac{1}{2}}\otimes \wt{\pi})(\Phi_{\rho,\psi,\ell}*\phi^{\vee}).
\end{align}
It is clear that
\begin{align*}
(\chi^{-1}_{s-\frac{1}{2}}\otimes \wt{\pi})(\Phi_{\rho,\psi,\ell}*\phi^{\vee})
& =
(\chi^{-1}_{s-\frac{1}{2}}\otimes \wt{\pi})(\Phi_{\rho,\psi,\ell})
\cdot
(\chi^{-1}_{s-\frac{1}{2}}\otimes \wt{\pi})(\phi^{\vee})\\
&=
f_{\ell}(\chi^{-1}_{s-\frac{1}{2}}\otimes \wt{\pi})
\cdot(\chi^{-1}_{s-\frac{1}{2}}\otimes\wt{\pi})(\phi^{\vee})
\end{align*}
according to \eqref{73-5}. Hence we get that \eqref{73-6} can be rewritten as
$$
f_{k}(\chi^{-1}_{s-\frac{1}{2}}\otimes \wt{\pi})
\langle \chi^{-1}_{s-\frac{1}{2}}\otimes
\wt{\pi}(\phi^{\vee})w,v\rangle.
$$
Therefore, we obtain the following expression:
\begin{align}\label{73-7}
\CZ(1-s,\Phi_{\rho,\psi,\ell}*\phi^{\vee},\vphi^{\vee})=
f_{k}(\chi^{-1}_{s-\frac{1}{2}}\otimes \wt{\pi})
\langle  \chi^{-1}_{s-\frac{1}{2}}\otimes
\wt{\pi}(\phi^{\vee})w,v\rangle.
\end{align}
We are going to compute $\langle  \chi^{-1}_{s-\frac{1}{2}}\otimes\wt{\pi}(\phi^{\vee})w,v\rangle$ as follows. Write that
$$
\langle \chi^{-1}_{s-\frac{1}{2}}\otimes \wt{\pi}(\phi^{\vee})w,v\rangle
=\int_{G(F)}
\phi^{\vee}(a,h)
\chi^{-1}_{s-\frac{1}{2}}(a)
\langle \wt{\pi}(h)w,v\rangle \ud^{*}a\ud h.
$$
The right-hand side,  after changing variable $(a,h)\to (a^{-1},h^{-1})$, is equal to
\[
\int_{G(F)}
\phi(a,h)\chi_{s-\frac{1}{2}}(a)
\langle \wt{\pi}(h^{-1})w,v\rangle \ud^{*}a\ud h
\]
which is equal to
\[
\int_{G(F)}
\phi(a,h)\chi(a)\langle w,\pi(h)v\rangle |a|^{s-\frac{1}{2}}\ud^{*}a\ud h
=\CZ(s,\phi,\vphi).
\]
By combining with the identity in \eqref{73-7}, we obtain the desired identity
$$
\CZ(1-s,\Phi_{\rho,\psi,\ell}*\phi^{\vee},\vphi^{\vee}) =
f_{\ell}(\chi^{-1}_{s-\frac{1}{2}}\otimes \wt{\pi})\cdot
\CZ(s,\phi,\vphi).
$$
\end{proof}

Now we apply Lemma \ref{lem:Phiell} to the identity in \eqref{zetaell}, and obtain that
for any $\phi\in \CC^{\infty}_{c}(G)$, the following identity
\begin{align}
\CZ(1-s,\CF_{\rho,\psi}(\phi),\vphi^{\vee}) =
(\sum_{\ell}f_{\ell}(\chi^{-1}_{s-\frac{1}{2}}\otimes \wt{\pi}))\cdot
\CZ(s,\phi,\vphi)
\end{align}
holds whenever $\Re(s)$ is sufficiently negative.
Comparing with the right-hand side of the functional equation in \eqref{lfe-73}, the following equality
\[
\sum_{\ell}f_{\ell}(\chi^{-1}_{s-\frac{1}{2}}\otimes \wt{\pi} ) = \Gamma_{\rho,\psi}(s,\chi\otimes\pi)
\]
converges absolutely whenever $\Re(s)$ is sufficiently negative and extends to all $s\in\BC$ by meromorphic continuation. By a shift of the parameter $s$, we obtain that the following identity
\begin{align}\label{gammaell}
\sum_{\ell}f_{\ell}(\chi_{s}\otimes \pi) = \Gamma_{\rho,\psi}(\frac{1}{2}, \chi_{s}^{-1}\otimes \wt{\pi})
\end{align}
converges absolutely whenever $\Re(s)$ is sufficiently positive and extends to all $s\in\BC$ by meromorphic continuation.

Finally, from \eqref{Phi-Phiell}, the decomposition $\Phi_{\rho,\psi} = \sum_{\ell}\Phi_{\rho,\psi,\ell}$ is well-defined. For any $\chi_s\otimes\pi\in\Pi(G)$, we define
\begin{align}\label{def:Phi-action}
(\chi_s\otimes\pi)(\Phi_{\rho,\psi})
:=
\sum_{\ell\in\BZ}(\chi_s\otimes\pi)(\Phi_{\rho,\psi,\ell}).
\end{align}
By Proposition \ref{pro:Phiell} and Lemma \ref{lem:Phiell}, together with \eqref{gammaell}, we obtain that
\[
\sum_{\ell\in\BZ}(\chi_s\otimes\pi)(\Phi_{\rho,\psi,\ell})
=
\sum_{\ell}f_{\ell}(\chi_{s}\otimes \pi)\Id_{\chi_s\otimes\pi}
\]
which is absolutely convergent for $\Re(s)$ sufficiently positive and extends to all $s\in\BC$ by meromorphic continuation.
Hence the definition of the convolution $(\chi_s\otimes\pi)(\Phi_{\rho,\psi})$ for any $\chi_s\otimes\pi\in\Pi(G)$ makes sense
whenever $\Re(s)$ is sufficiently positive and the following identity:
\begin{align}\label{Phi-action}
(\chi_{s}\otimes \pi)(\Phi_{\rho,\psi}) = \Gamma_{\rho,\psi}(\frac{1}{2}, \chi_{s}^{-1}\otimes \wt{\pi})
\cdot\Id_{\chi_s\otimes \pi}
\end{align}
holds by meromorphic continuation.
This proves the following theorem.

\begin{thm}\label{thm:Phi-gamma}
The $G(F)$-invariant distribution $\Phi_{\rho,\psi}$ on $G(F)$ as introduced in Definition \ref{dfn-SDFT} enjoys the following properties.
\begin{enumerate}
\item For any $\ell\in \BZ$, the distribution $\Phi_{\rho,\psi,\ell} = \mathbf{1}_{\ell}\cdot \Phi_{\rho,\psi}$ belongs to the Bernstein center $\CZ(G)$ of $G(F)$,
where $\mathbf{1}_{\ell}$ is the characteristic function of $G_{\ell}$ as defined in \eqref{73-2}.
\item For any irreducible admissible representation $\chi\otimes \pi$ of $G(F)$, let $f_{\ell}(\chi\otimes \pi)$ be the regular function defined on the Bernstein variety of $G(F)$ such that
$$
(\chi\otimes \pi)(\Phi_{\rho,\psi,\ell}) = f_{\ell}(\chi\otimes \pi)\Id_{\chi\otimes \pi}.
$$
Then the summation
$
\sum_{\ell}f_{\ell}(\chi_{s}\otimes \pi)
$
converges absolutely when $\Re(s)$ is sufficiently positive, and admits meromorphic continuation to $s\in \BC$. Moreover, the following identity
$$
\sum_{\ell}f_{\ell}(\chi_{s}\otimes \pi)=\Gamma_{\rho,\psi}(\frac{1}{2}, \chi_{s}^{-1}\otimes \wt{\pi})
$$
holds after meromorphic continuation to $s\in\BC$.
\item For $\chi_s\otimes\pi\in\Pi(G)$, the $\chi_s\otimes\pi$-Fourier coefficient of the $G(F)$-invariant distribution $\Phi_{\rho,\psi}$ or its convolution with $\chi_s\otimes\pi$, as defined in \eqref{def:Phi-action},
has property that
\[
(\chi_{s}\otimes \pi)(\Phi_{\rho,\psi}) = \Gamma_{\rho,\psi}(\frac{1}{2}, \chi_{s}^{-1}\otimes \wt{\pi})
\cdot\Id_{\chi_s\otimes \pi}
\]
holds for $\Re(s)$ sufficiently positive, and then for all $s\in\BC$ away from the possible poles of $\Gamma_{\rho,\psi}(\frac{1}{2}, \chi_{s}^{-1}\otimes \wt{\pi})$
by meromorphic continuation.
\end{enumerate}
\end{thm}

This completes the theory of harmonic analysis and gamma functions. By using Corollary \ref{gam-gam} that the gamma function $\Gamma_{\rho,\psi}(s, \chi\otimes \pi)$ is the same as the Langlands $\gamma$-function
$\gam(s,\chi\otimes\pi,\rho,\psi)$, we know that Theorem \ref{thm:Phi-gamma} implies Theorem \ref{thm:LGF}.


\section{Theorems \ref{thm:LGF}, \ref{thm:LFE-CF-BK}, and \ref{thm:LLF}}\label{sec-Thm-1-2-3}


We are going to finish the proofs of Theorems \ref{thm:LGF}, \ref{thm:LFE-CF-BK}, and \ref{thm:LLF} by using the well developed results from the local theory of the Piatetski-Shapiro and Rallis zeta integrals,
and give a quick confirmation (Corollary \ref{gam-gam}) that the gamma function $\Gamma_{\rho,\psi}(s, \chi\otimes \pi)$ as defined in Proposition \ref{pro:lfe} is equal to the Langlands $\gamma$-function
$\gam(s,\chi\otimes\pi,\rho,\psi)$.

\subsection{Proof of Theorems \ref{thm:LGF} and \ref{thm:LFE-CF-BK}}\label{sec:lfe-dm}
For a Schwartz function $\phi\in\CS_\rho(G)$ and a matrix coefficient $\varphi\in\CC(\chi\otimes\pi)$,
we prove Theorem \ref{thm:LFE-CF-BK}, which is re-stated as follows.

\begin{thm}\label{thm-lfe}
The local zeta integral $\CZ(s,\phi,\vphi)$ as in \eqref{lzi} converges absolutely for $\Re(s)$ sufficiently positive, admits meromorphic continuation to $s\in \BC$ and satisfies the following functional equation:
$$
\CZ(1-s, \CF_{\rho,\psi}(\phi) , \vphi^\vee)  =\gam(s,\chi\otimes \pi,\rho,\psi) \CZ(s,\phi,\vphi),
$$
where $\CF_{\rho,\psi}$ is the Fourier operator defined in Definition \ref{dfn-SDFTG}, and
\[
{\vphi}^\vee(a,h)=\chi^{-1}(a)\langle w,\pi(h^{-1})v\rangle
\]
is the matrix coefficient associated to the contragredient $\chi^{-1}\otimes\wt{\pi}$ of $\chi\otimes\pi$.
\end{thm}

\begin{proof}
Take a matrix coefficient $\vphi(a,h) = \chi(a)\langle w,\pi(h)v \rangle$ of $\chi\otimes\pi$ associated to a pair of smooth vectors $v\in \pi,w\in \wt{\pi}$.  Then
$$
{\vphi}^\vee(a,h)=\chi^{-1}(a)\langle w,\pi(h^{-1})v\rangle
$$
is a matrix coefficient of the contragredient of $\chi\otimes\pi$. Take $\varphi'$ to be the restriction of $\varphi$ into $\{1\}\times\Sp_{2n}(F)$, i.e., $\varphi'(h)=\apair{w,\pi(h)v}$ for $h\in \Sp_{2n}(F)$.
For $\phi\in \CS_\rho(G)$, from the relation in \eqref{f-phi}, the function $\phi(a,g)$ is obtained from a smooth function $f$ in $\CS_{\pvs}(X_{P_\Del})$ with
\begin{align}\label{71-1}
\phi(a,h) = f(\Fs^{-1}_{a} (h,\RI))|a|^{\frac{2n+1}{2}}.
\end{align}
By taking the projection $f_{\chi_s}=\CP_{\chi_s}(f)$, as defined in \eqref{proj-X-I}, we first prove the following identity:
\begin{equation}\label{RHS}
\CZ(s+\frac{1}{2},\phi,\vphi)=\RZ(f_{\chi_{s}},\vphi'),
\end{equation}
which implies that the local zeta integral $\CZ(s,\phi,\vphi)$ converges absolutely for $\Re(s)$ sufficiently positive and admits a meromorphic continuation to $s\in\BC$.

In fact, for $\Re(s)$ sufficiently positive, we have
\begin{align}\label{71-2}
\CZ(s+\frac{1}{2},\phi,\varphi)=
\int_{F^{\times}\times \Sp_{2n}(F)}
\phi(a,h)\vphi(a,h)|a|^{s}\ud^{*}a\ud h,
\end{align}
which, by \eqref{71-1} and the definition of $\varphi(a,h)$, can be written as
\begin{align}\label{71-3}
\int_{F^\times\times\Sp_{2n}(F)}\chi_{s}(a)|a|^{\frac{2n+1}{2}}	f(\Fs_{a}^{-1}\cdot(h,\RI_{2n}))\apair{w,\pi(h)v}\ud^{*} a\ud h.
\end{align}
By the definition of the projection $\CP_{\chi_s}$ as in \eqref{proj-X-I}, the integral in \eqref{71-3} is equal to
\[
\int_{\Sp_{2n}(F)} \CP_{\chi_s}(f)((h,\RI_{2n}))\apair{w,\pi(h)v}\ud h=\RZ(f_{\chi_s},\varphi').
\]
This proves the identity in \eqref{RHS} for $\Re(s)$ sufficiently positive, which still holds for all $s\in\BC$ by meromorphic continuation.

When $\Re(s)$ is sufficiently negative, the left-hand side of the functional equation
in Theorem \ref{thm-lfe} is given by
\begin{equation}\label{71-4}
\CZ(1-s, \CF_{\rho,\psi}(\phi) , {\vphi}^\vee)
=
\int_{F^{\times}\times \Sp_{2n}(F)} \CF_{\rho,\psi}(\phi)(a,h){\vphi^\vee}(a,h) |a|^{\frac{1}{2}-s}\ud a^{*}\ud h.
\end{equation}
We intend to show that the following identity
\begin{equation}\label{LHS}
\CZ(\frac{1}{2}-s, \CF_{\rho,\psi}(\phi) , {\vphi}^\vee)
=
\pi(-\RI_{2n})\cdot\RZ(\RM^{\dag}_{w_{\Del}}(s,\chi,\psi)f_{\chi_{s}},\vphi').
\end{equation}
holds for $\Re(s)$ sufficiently negative and for all $s\in\BC$ by meromorphic continuation.

In fact, by definition, whenever $\Re(s)$ is sufficiently negative, the right-hand side of \eqref{LHS} is equal to
\begin{align}\label{71-5}
\pi(-\RI_{2n})\cdot\RZ(\RM^{\dag}_{w_{\Del}}(s,\chi,\psi)(f_{\chi_{s}}),\vphi')\nonumber\qquad\qquad\qquad\qquad\qquad\qquad
\\
=
\int_{\Sp_{2n}(F)}\RM^{\dag}_{w_{\Del}}(s,\chi,\psi)(f_{\chi_{s}})((h,\RI_{2n}))\langle w,\pi(-h)v\rangle \ud h.
\end{align}
By taking a change of variable $-h^{-1}\mapsto h$, it becomes
\begin{equation}\label{71-6}
\int_{\Sp_{2n}(F)}\RM^{\dag}_{w_{\Del}}(s,\chi,\psi)(f_{\chi_{s}})((-h^{-1},\RI_{2n}))\langle w,\pi(h^{-1})v\rangle \ud h.
\end{equation}
By Corollary \ref{cor:ComOnG}, we have
\begin{align*}
\RM^{\dag}_{w_{\Del}}(s,\chi,\psi)(f_{\chi_{s}})((-h^{-1},\RI))
&=\CP_{\chi^{-1}_s}(f_{\CF_{\rho,\psi}(\phi)})((h,\RI)) \\
=&\int_{F^\times}\chi^{-1}_{s}(a)\delta_{P_\Del}^{\frac{1}{2}}(\Fs_{a})
f_{\CF_{\rho,\psi}(\phi)}(\Fs_a^{-1}\cdot(h,\RI))\ud^{*} a,
\end{align*}
where $f_{\CF_{\rho,\psi}(\phi)}$ is obtained via $\CF_{\rho,\psi}(\phi)$ from the relation in \eqref{f-phi}.
By plugging the above formula into \eqref{71-6}, we obtain that the right-hand side of \eqref{LHS}, $\pi(-\RI_{2n})\cdot\RZ(\RM^{\dag}_{w_{\Del}}(s,\chi,\psi)f_{\chi_{s}},\vphi)$ is equal to
\begin{align}\label{71-7}
\int_{F^\times\times \Sp_{2n}(F)}
\chi^{-1}(a)|a|^{-s}|a|^{\frac{2n+1}{2}}
f_{\CF_{\rho,\psi}(\phi)}(\Fs_{a}^{-1}(h,\RI)) \apair{w,\pi(h^{-1})v}\ud^{*} a\ud h,
\end{align}
which, by the relation in \eqref{f-phi} again, can be written as
\begin{align}\label{71-8}
\int_{F^\times\times\Sp_{2n}(F)}\CF_{\rho,\psi}(\phi)(a,h)\chi^{-1}(a)\langle w,{\pi}(h^{-1})v\rangle |a|^{-s}\ud^{*} a\ud h.
\end{align}
Now the integral in \eqref{71-8} is exactly equal to $\CZ(\frac{1}{2}-s, \CF_{\rho,\psi}(\phi) , {\vphi^\vee})$, which is the left-hand side of \eqref{LHS}.
This proves \eqref{LHS}.

Recall from Corollary \ref{FE-Langlands gamma}, the local functional equation for the local zeta integrals of Piatetski-Shapiro and Rallis is
\begin{equation}\label{PSR-lfe}
\pi(-\RI_{2n})\cdot\RZ(\RM^{\dag}_{w_{\Del}}(s,\chi,\psi)f_{\chi_{s}},\vphi') = \gam(s+\frac{1}{2},\chi\otimes\pi,\rho,\psi)\cdot\RZ(f_{\chi_{s}},\vphi'),
\end{equation}
where $\varphi'$ is the restriction of $\varphi$ into $1\times\Sp_{2n}(F)$, i.e., $\varphi'(h)=\apair{w,\pi(h)v}$ for $h\in \Sp_{2n}(F)$.
Hence from \eqref{LHS} and \eqref{RHS}, we obtain the following functional equation
\begin{equation}\label{BK-lfe}
\CZ(\frac{1}{2}-s, \CF_{\rho,\psi}(\phi) , {\vphi^{\vee}})  =\gam(s+\frac{1}{2},\chi\otimes\pi,\rho,\psi) \CZ(s+\frac{1}{2},\phi,\vphi).
\end{equation}
By a shift: $s\mapsto s-\frac{1}{2}$, we obtain the functional equation in Theorem \ref{thm-lfe}. We are done.
\end{proof}

From Proposition \ref{pro:lfe} and Theorem \ref{thm-lfe}, we obtain

\begin{cor}\label{gam-gam}
The gamma function
$\Gamma_{\rho,\psi}(s, \chi\otimes \pi)$ as defined in Proposition \ref{pro:lfe} is equal to the Langlands $\gamma$-function $\gam(s,\chi\otimes\pi,\rho,\psi)$, as meromorphic functions in $s$,
for any $\chi\otimes\pi\in\Pi(G)$.
\end{cor}

By Corollary \ref{gam-gam}, we deduce that Theorem \ref{thm:Phi-gamma} implies Theorem \ref{thm:LGF}.

\subsection{$L$-functions}\label{ssec-pf-Thm12}

We start to prove Theorem \ref{thm:LLF} here, which will be finished in Section \ref{ssec-bf}.

From the proof of Theorem \ref{thm-lfe}, we can deduce, by using the work of Yamana (\cite[\S 5]{Y14}), that the poles of the local zeta integrals $\CZ(s,\phi,\vphi)$ are completely determined by the poles of the
Langlands local $L$-function $L(s,\chi\otimes \pi,\rho)$. This indicates that the Schwartz space $\CS_{\rho}(G)$ is a candidate for the $\rho$-Schwartz space in the sense of the Braverman-Kazhdan proposal (\cite{BK00}) and of Ng\^o (\cite{N20}).

\begin{thm}\label{thm:rho-Ss}
For any irreducible admissible representation $\chi\otimes \pi$ of $G(F)$, the following set
$$
\CI_{\chi\otimes \pi} = \{
\CZ(s,\phi,\vphi)\mid \phi\in \CS_{\rho}(G),\vphi\in \CC(\chi\otimes\pi)
\}
$$
is a fractional ideal of $\BC[q^{s},q^{-s}]$ that admits a greatest common denominator $P_{\chi\otimes\pi}(q^{-s})$ with $P_{\chi\otimes\pi}(0)=1$ and
\[
P_{\chi\otimes\pi}(q^{-s})^{-1}=L(s,\chi\otimes \pi,\rho),
\]
which is the Langlands local $L$-factor $L(s,\chi\otimes \pi,\rho)$.
\end{thm}

\begin{proof}
Take any $\phi = \phi_{f}\in \CS_{\rho}(G)$ with $f\in \CS_{\pvs}(X_{P_{\Del}})$, and $\vphi\in \CC(\chi\otimes\pi)$, the space of matrix coefficients of $\chi\otimes\pi$.
From \eqref{RHS} in the proof of Theorem \ref{thm-lfe}, for $\Re(s)$ sufficiently positive,
the local zeta integral $\CZ(s,\phi,\vphi)$ can be identified with the doubling local zeta integral:
$$
\CZ(s,\phi_{f},\vphi) = \RZ(f_{\chi_{s-\frac{1}{2}}},\vphi^{\p})
$$
where $\vphi^{\p} = \vphi|_{\{ 1\}\times \Sp_{2n}(F)}$ and $\CP_{\chi_{s}}(f) = f_{\chi_{s}}\in \RI^{\dag}(s,\chi)$, the space of good sections on $\Sp_{4n}(F)$.
Hence $\CZ(s,\phi,\vphi)$ converges absolutely when $\Re(s)$ is sufficiently positive, and admits meromorphic continuation to $s\in \BC$.
By Proposition \ref{pro:FTX}, $\CP_{\chi_{s}}$ is surjective from $\CS_\pvs(X_{P_{\Del}})$ to $\RI^{\dag}(s,\chi)$. Hence we obtain another description of the set $\CI_{\chi\otimes \pi}$:
$$
\CI_{\chi\otimes \pi} = \{ \RZ(f_{\chi_{s-\frac{1}{2}}},\vphi^{\p})\ |\ f_{\chi_{s}}\in \RI^{\dag}(s,\chi),\vphi\in \CC(\pi) \}.
$$
According to the local theory of the doubling zeta integrals as explicitly discussed in \cite[\S 5]{Y14}, one has that
$$
\{ \RZ(f_{\chi_{s-\frac{1}{2}}},\vphi^{\p})\ |\ f_{\chi_{s}}\in \RI^{\dag}(s,\chi),\vphi\in \CC(\pi) \} = L(s,\chi\otimes \pi,\rho)\cdot \BC[q^{s},q^{-s}].
$$
It follows that
$\CI_{\chi\otimes \pi}=L(s,\chi\otimes \pi,\rho)\cdot \BC[q^{s},q^{-s}]$ is a fractional ideal.
We are done.
\end{proof}

\subsection{Basic function}\label{ssec-bf}

We continue our proof of Theorem \ref{thm:LLF}, related to properties of the basic function $\BL_\rho$.
By definition, the basic function is a unique bi-$K_G$-invariant function $\BL_\rho$ on $G(F)$ with property that
\begin{eqnarray}\label{basicf}
\CZ(s,\BL_\rho,\vphi_\circ)&=&\int_{F^{\times}\times \Sp_{2n}(F)}
\BL_{\rho}(a,h)\vphi_\circ(a,h)|a|^{s-\frac{1}{2}}
\ud^{*} a \ud g\nonumber\\
&=&L(s,\chi\otimes \pi,\rho),
\end{eqnarray}
where $K_G:=\CO^\times\times\Sp_{2n}(\CO)$, $\vphi_\circ(a,h)$ is the normalized unramified matrix coefficient of an unramified $\chi\otimes\pi$ with the normalization $\vphi_{\circ}(1,\RI_{2n}) = 1$, and
$\CZ(\cdots)$ is the local zeta integral as defined in \eqref{lzi}. The construction and the uniqueness of the such basic functions for a given pair $(G,\rho)$ have been worked out in \cite{Li17} and \cite{Luo19}
(see also the relevant discussions in \cite{BNS16}, \cite{Gz18}, \cite{Sak18}, and \cite{Sh18}).

We intend to show here that there exists a unique basic function $\BL_\rho$ that belongs to the space $\CS_\rho(G)$, and is fixed under the action
of the Fourier operator $\CF_{\rho,\psi}$. Based on our strategy in this paper, we are going to use the relevant known results about the doubling local zeta integrals via Theorem \ref{thm:FT-FT} and Corollary \ref{cor:ComOnG}.

In the proof of Theorem \ref{thm-lfe}, we show in \eqref{RHS} that
\begin{equation}\label{RHS-1}
\CZ(s+\frac{1}{2},\phi,\vphi)=\RZ(f_{\chi_{s}},\vphi')
\end{equation}
holds for $\Re(s)$ sufficiently positive and for $s\in\BC$ via meromorphic continuation, where $f_{\chi_s}=\CP_{\chi_s}(f)$ with a function $f\in\CS_{\pvs}(X_{P_\Del})$ that is determined by $\phi$ as in \eqref{f-phi}, and
$\vphi'(h)=\vphi(1,h)$, which is a matrix coefficient of $\pi$.

Since the relevant results for doubling local zeta integrals are only known for the case where the residue characteristic of $F$ is odd (\cite{LR05}, for instance), for the rest of this section,
we assume the ground field $F$ has an odd residue characteristic, and may come back to the even residue characteristic situation as needed in future.
From \cite[Proposition 3]{LR05}, the unramified calculation of the local zeta integral of Piatetski-Shapiro and Rallis implies that
$$
\RZ(f_{\chi_s}^\circ,\vphi'_\circ) =
\frac{L(s+\frac{1}{2},\chi\otimes \pi,\rho)}{b_{2n}(s,\chi)},
$$
where $b_{2n}(s,\chi)$ is given in \eqref{eq:M-Unramified} and $f^{\circ}_{\chi_{s}}$ is the normalized spherical vector in $\RI(s,\chi)$, which in particular is $K_{\Del}$-invariant. Here $\pi$ (resp. $\chi$) is an unramified representation of $\Sp_{2n}(F)$ (resp. $F^{\times}$). In other words,
$$
\RZ(b_{2n}(s,\chi)\cdot f_{\chi_s}^\circ,\vphi'_\circ) =L(s+\frac{1}{2},\chi\otimes \pi,\rho).
$$
Since $b_{2n}(s,\chi)\cdot f_{\chi_s}^\circ$ is a unique $K_\Del$-invariant good section in $\RI^\dag(s,\chi)$ satisfying the above identity, by Theorem \ref{pro:FTX}, there exists a unique bi-$K_\Del$-invariant function
$\BL_\rho'$ in $\CS_\pvs(X_{P_\Del})$ such that
$$
b_{2n}(s,\chi)\cdot f_{\chi_s}^\circ=\CP_{\chi_s}(\BL_\rho').
$$
By the transition relation in \eqref{f-phi}, there exists a unique function $\BL_\rho$ in the Schwartz space $\CS_\rho(G)$ corresponding to the function $\BL_\rho'$. We claim that such a constructed
function $\BL_\rho$ is the {\sl basic function} for the local unramified $L$-function $L(s,\chi\otimes \pi,\rho)$.

First of all, under the assumption that the residue characteristic of $F$ is odd, we have that $g_{0}\in K_{4n} = \Sp_{4n}(\CO)$, where $g_0$ is defined in \eqref{g0}, which relates the variety $X_{P_\Del}$ to the
variety $X_{P_{\std}}$. Hence we have that $K_{\Del} = g^{-1}_{0}K_{4n}g_{0} = K_{4n}$. In particular the image of the embedding of $K_{G}\times K_{G}\hookrightarrow \Sp_{4n}$ induced from $\Sp_{2n}(F)\times \Sp_{2n}(F)\hookrightarrow \Sp_{4n}(F)$ lies in $K_{4n}$. Therefore, $\BL_{\rho}$ is bi-$K_{G}$-invariant.
It is clear from \eqref{RHS-1} that
$$
\CZ(s,\BL_\rho,\vphi_\circ)=L(s,\chi\otimes \pi,\rho).
$$
This proves the claim above and that $\BL_\rho\in\CS_\rho(G)$ is the {\sl basic function} for $(G,\rho)$.
Note that the uniqueness of the basic function $\BL_\rho\in\CS_\rho(G)$ can be verified by using the Mellin inversion formula applied to $\chi_{s}$ (see \cite{Luo19}, for instance).

It remains to show that the basic function $\BL_\rho\in\CS_\rho(G)$ is fixed under the action of the Fourier operator $\CF_{\rho,\psi}$.
From \eqref{eq:M-Unramified} again, we know that for the normalized spherical section $f_{\chi_s}^\circ$ in $\RI(s,\chi)$, the normalized intertwining operator $\RM^{\dag}_{w_{\Del}}$ sends the good section
$b_{2n}(s,\chi)f_{\chi_s}^\circ\in \RI^\dag(s,\chi)$ to the good section $b_{2n}(-s,\chi^{-1})f_{\chi^{-1}_{s}}^\circ\in\RI^\dag(-s,\chi^{-1})$. By Theorem \ref{thm:CFP-M}, we must have $\CF_{X,\psi}(\BL_\rho')=\BL_\rho'$,
because of the uniqueness of $\BL_\rho'$.  Finally, by Theorem \ref{thm:FT-FT}, we have that $\CF_{\rho,\psi}(\BL_\rho)=\BL_\rho$ as functions on $G(F)$.

We summarize the above discussion as a proposition.

\begin{pro}\label{pro:FT-bf}
Assume that the residue characteristic of $F$ is odd.
There exists a unique basic function $\BL_\rho$ belonging to the space $\CS_\rho(G)$, with the following two properties.
\begin{enumerate}
\item For the normalized matrix coefficient $\vphi_\circ$ of an unramified representation $\chi\otimes\pi$, the local zeta integral as in \eqref{lzi} gives exactly the local $L$-function:
$$
\CZ(s,\BL_\rho,\vphi_\circ)=L(s,\chi\otimes \pi,\rho).
$$
\item The Fourier operator $\CF_{\rho,\psi}$ preserves the basic function $\BL_\rho$:
\[\CF_{\rho,\psi}(\BL_\rho)=\BL_\rho.
\]
\end{enumerate}
\end{pro}

This finishes our proof of Theorem \ref{thm:LLF}.

From the proof of Proposition \ref{pro:FT-bf}, when the residue characteristic of $F$ is even, the {\sl basic function} $\BL_\rho$ may have a different structure. We omit the discussion here.

\bibliographystyle{amsalpha}

\end{document}